\documentclass[letterpaper,12pt,reqno]{amsart}
\usepackage{amssymb}
\usepackage{cases}
\usepackage{amsmath}
\usepackage{mathrsfs}
\usepackage{amsfonts}
\usepackage[all]{xy}
\DeclareFontFamily{OT1}{pzc}{}
\DeclareFontShape{OT1}{pzc}{m}{it}{<-> [1.15] rpzcmi}{}
\DeclareMathAlphabet{\mathzc}{OT1}{pzc}{m}{it}

\def\End{\operatorname{End}\kern-.5pt}

\def\KK{{\mathzc K\kern0pt}}
\def\vv{{\mathzc v\kern.5pt}}

\def\aq{/\kern-2pt/}
\def\Hom{\operatorname{Hom}}

\def\La{{\Lambda}}
\def\fT{{\mathfrak T}}
\def\fS{{\mathfrak S}}
\def\End{{\text{\rm End}}}
\def\Hom{{\text{\rm Hom}}}
\def\modh{{\text{{\sf Mod}-}}}
\def\hmod{{\text{-{\sf Mod}}}}
\def\bsi{{\boldsymbol i}}
\def\bsj{{\boldsymbol j}}
\def\bsu{{\boldsymbol u}}
\def\up{{\boldsymbol{\upsilon}}}
\def\sH{{\mathcal H}}
\def\sS{{\mathcal S}}
\def\sZ{{\mathcal Z}}
\def\sD{{\mathcal D}}
\def\sT{{\mathcal T}}
\def\sA{{\mathcal A}}
\def\sI{{\mathcal I}}
\def\sJ{{\mathcal J}}
\def\la{{\lambda}}

\def\whd{{\widehat{d}}}
\def\vecr{{\vec r}}
\def\sP{{\mathcal P}}
\def\sR{{\mathcal R}}

\def\sB{{\mathcal B}}
\def\sT{{\mathcal T}}
\def\rmL{{\mathrm L}}
\def\vep{{\varepsilon}}
\def\ttM{{\mathtt M}}
\def\reg{{\text{-reg}}}

\newtheorem{theorem}{Theorem}[section]
\newtheorem{lemma}[theorem]{Lemma}
\newtheorem{proposition}[theorem]{Proposition}
\newtheorem{corollary}[theorem]{Corollary}

\theoremstyle{definition}
\newtheorem{definition}[theorem]{Definition}

\newtheorem{remark}[theorem]{Remark}
\newtheorem{remarks}[theorem]{Remarks}
\newtheorem{notation}[theorem]{Notation}
\numberwithin{equation}{theorem}

\begin{document}

\baselineskip16pt
\def\hei{\relax}

 \title[Irreducible representations of $q$-Schur superalgebras]{Irreducible representations of $q$-Schur superalgebras
at a root of unity}
\author{Jie Du, Haixia Gu and Jianpan Wang}
\address{J.D., School of Mathematics and Statistics,
University of New South Wales, Sydney NSW 2052, Australia}
\email{j.du@unsw.edu.au}
\address{H.G. \& J.W., Department of Mathematics,
East China Normal University, Shanghai, China}
\email{alla0824@126.com,jpwang@admin.ecnu.edu.cn}
\date{\today}
\thanks{The first author gratefully acknowledge support from ARC under grant DP120101436. The work was completed while the second author was visiting UNSW.
She would like to thank the China Scholarship Council and ARC for financial support and UNSW for its hospitality during her visit.}


\begin{abstract}
Under the assumption that the quantum parameter $q$ is an $l$th primitive root of unity with $l$ odd in a field $F$ of characteristic 0 and $m+n\geq r$,
 we obtained a complete classification of  irreducible
modules of the $q$-Schur superalgebra $\sS_F(m|n,r)$ introduced
in \cite{DR}.
\end{abstract}

 \maketitle

\section{Introduction}

The investigation on Schur superalgebras and their quantum analogue has achieved significant progress in the last decade; see, e.g., \cite{BKv, BK, D, M, DR, TK}. This includes an establishment of a super analogue of the Schur--Weyl reciprocity  by Mitsuhashi,  a Kazhdan--Lusztig type cell approach to the representation theory of $q$-Schur superalgebras by Rui and the first author, and, recently, a presentation of $q$-Schur superalgebras by Turkey and Kujawa.
As part of a super version of the quantum Schur--Weyl theory, these developments are all in the generic case where the quantum parameter is generic or not a root of unity.

This paper attempts to investigate the structure and representations of $q$-Schur superalgebra at a root of unity. As is seen from \cite{DR} and \cite{TK}, $q$-Schur superalgebras share many of the properties of $q$-Schur algebras. For example, both have several  definitions in terms of Hecke algebras via (super) $q$-permutation modules, quantum enveloping algebras via tensor (super)spaces, and quantum coordinate (super)algebras via the dual of homogenous components; both have integral versions and the base change property; both have Drinfeld-Jimbo type presentations; etc. However, the structure of $q$-Schur superalgebras is fundamentally different. For example, no quasihereditary structure or cellular structure is seen naturally. Thus, classifying irreducible modules requires an approach that is different from the usual quasihereditary or cellular approach.

In this paper, we will generalise the approach used in \cite{DU,DU3} for $q$-Schur algebras to $q$-Schur superalgebras. This approach was developed as  a $q$-deformation of the modular permutation representation theory established by L. Scott in 1973 \cite{S}. One important feature of this approach is that the resulting classification of irreducible representations is analogous to  Alperin's weight conjecture \cite{A} for modular representations of a finite group. More precisely, we consider a $q$-analogue of an endomorphism algebra of a signed permutation module instead of a group algebra, consider $l$-parabolic subgroups instead of $p$-subgroups, and consider representations of a quotient algebra defined by an $l$-parabolic subgroup as ``local'' representations. Thus, the isomorphism types of irreducible modules are determined by an $l$-parabolic subgroup $P$ together with a projective indecomposable module of the quotient algebra associated with $P$.

The main ingredients in this approach are the notion of relative norms for representations of Hecke algebras of type $A$, which was introduced by P. Hoefsmit and L. Scott in 1977, and its corresponding properties which include transitivity, Mackey decomposition, Frobenius reciprocity, Nakayama relation, Higman criterion; see \cite{J}.

By describing a basis for a $q$-Schur superalgebra in relative norms,
we will attach a defect group, which is an $l$-parabolic subgroup, to a basis element  and introduce a filtration of ideals in terms of defect groups. Thus, every primitive idempotent $e$ lies in a minimal ideal in the sequence. We then attach to $e$ a defect group $D(e)$. This defect group is the vertex of the corresponding indecomposable module for the Hecke algebra. If $D(e)$ is trivial, then the indecomposable module is projective. Hence, it is determined by an $l$-regular partition. If $D(e)$ is nontrivial, then the image of $e$ in the quotient algebra by the ideal right below $e$ has the trivial defect group. Thus, once we show that the quotient algebra
is again an endomorphism algebra part of  which is again a $q$-Schur superalgebra, we may determine all primitive idempotents with the trivial defect group in the quotient algebra. In this way, we completely determine the isomorphism types of irreducible modules for a $q$-Schur superalgebras $\sS_F(m|n,r)$ with $m+n\geq r$ by the index set $\mathcal P_r$ (Theorem \ref{ML4}).

We point out that classification of irreducible supermodules for Schur superalgebras was first completed by Donkin \cite{D} under the condition $m,n\geq r$. Brundan and Kujawa extended the result to arbitrary $m,n,r$ (\cite[Th.~6.5]{BK}) in their beautiful paper for a proof of the Mullineux conjecture. Their approach is quite different from ours, involving a certain category equivalence, Mullineux conjugation function,  a process of removing nodes from $p$-rims of a partition, and a couple of other relevant functions. See Appendix II for more details, where we will also make a comparison between the two classification index sets $\La^{++}(m|n,r)$ and $\mathcal P_r$ when $m+n\geq r$ and $l=p$. 


Like Alperin's weight conjecture, our approach is not  constructive for irreducible modules. However, it indicates a certain tensor product structure. We plan to address this issue in a different paper; see Remark \ref{Alperin}. One possible approach is to quantise the work of Brundan--Kujawa. However, it would be interesting to find an explicit construction compatible  with the index set $\mathcal P_r$. We also point out that we only consider the $q$-Schur superalgebra as an algebra and representations as algebra representation in this paper. 

The contents of the paper are organised as follows:

 \begin{enumerate}
 \item Introduction
 \item Quantum Schur superalgebras
 \item Relative norms: the first properties
 \item Some vanishing properties of relative norms on $V_R(m|n)^{\otimes r}$
 \item Bases for $\sS_R(m|n,r)$ in relative norms
 \item A filtration of ideals of $\sS_R(m|n,r)$
 \item Alternative characterisation of the ideals $I_F(P,r)$
 \item Quantum matrix superalgebras
 \item Frobenius morphism and Brauer homomorphisms
 \item Shrinking defect groups via Brauer homomorphisms
 \item Classification of irreducible $\sS_F(m|n,r)$-modules
 \item Appendix I: Brauer homomorphisms without Frobenius
 \item Appendix II: A comparison with the classification of Brundan-Kujawa
   \end{enumerate}

Throughout the paper,
let $\sZ=\mathbb{Z}[\up,\up^{-1}]$ be the Laurent polynomial ring in indeterminate $\up$ and let  $R$ be a domain which is a $\sZ$-module via a ring homomorphism $\sZ\to R$. We further assume that the image $q$ of $\up$ in $R$ is a primitive $l$th root of 1 with $l$ odd. We will always assume the characteristic char$(R)\neq2$. From \S 7 onwards, $R=F$ is a field of characteristic 0. For $a,b\in\mathbb N$ with $a<b$, we often write
$$[a,b]=\{a,a+1,\ldots,b\}.$$

\section{Quantum Schur superalgebras}

We first recall the definitions of $q$-Schur superalgebras in terms of signed $q$-permutation modules and tensor superspaces.

 Let $(W,S)$ be the symmetric group on $r$ letters where $W=\fS_r=\fS_{\{1,2,\ldots,r\}}$ and $S=\{s_k\mid 1\leq k<r\}$ is the set of basic transpositions $s_k=(k,k+1),$ and let
$\ell:W\to\mathbb{N}$ be the length function with respect to $S$.

An $n$-tuple $\lambda=(\lambda_1,\lambda_2,\cdots,\lambda_n)\in{\mathbb N}^n$ is called a composition of $r$ into $n$ parts if $|\la|:=\sum_i\la_i=r$. A composition $\la$ of $r$ is called a partition of $r$ if $\la_1\geq\la_2\geq\ldots$. Let $\La(n,r)\subset \mathbb{N}^n$ be the set of all compositions of $r$ into $n$ parts and let $\La^+(n,r)$ be the subset of partitions in $\La(n,r)$. In particular, let $\La^+(r)=\La^+(r,r)$ be the set of all partitions of $r$. We sometimes use the notation $\la\models r$ or $\la\vdash r$ for a composition or partition of $r$.

The parabolic (or the standard Young) subgroup
$W_{\lambda}$ of $W$ associated with a composition $\la$ consists of the permutations of
$\{1,2,\cdots,r\}$ which leave invariant the following sets of
integers
$$\{1,2,\cdots,\lambda_1\},\{\lambda_1+1,\lambda_1+2,\cdots,\lambda_1+\lambda_2\},\{\lambda_1+\lambda_2+1,\lambda_1+\lambda_2+2,\cdots\},\cdots.$$

We will frequently use the following notation:
if $W'$ is a subgroup of $W$ and $W_\mu$ is another parabolic subgroup, the notation $W_\mu=_WW_\la$ means $W_\mu^x:=x^{-1}W_\mu x=W_\la$ for some $x\in W$, while the notation
$W_\mu\leq_WW'$ means that a parabolic conjugate of $W_\mu$ is a subgroup of $W'$, i.e.,$W_\mu^x\leq W'$ and $W_\mu^x$ is also parabolic for some $x\in W$.

We will also denote by $\sD_\la:=\mathcal{D}_{W_\la}$
the set of all distinguished (or shortest) coset representatives of the right
cosets of $W_\la$ in $W$. Let
$\mathcal{D}_{\lambda\mu}=\mathcal{D}_\lambda\cap\mathcal{D}^{-1}_{\mu}$.
Then $\mathcal{D}_{\lambda\mu}$ is the set of distinguished
$W_\lambda$-$W_\mu$ coset representatives.
For $d\in\mathcal{D}_{\la\mu}$, the subgroup $W_\la^d\cap W_\mu=d^{-1}W_\la d\cap W_\mu$ is a parabolic subgroup associated with a composition which is denoted by $\la d\cap\mu$. In other words, we define
\begin{equation}\label{ladmu}
W_{\la d\cap\mu}=W_\la^d\cap W_\mu.
\end{equation}

We will often regard a pair $(\lambda^{(0)},\lambda^{(1)})\in \Lambda(m,r_1)\times
\Lambda(n,r_2)$ of compositions  of $m$ parts and $n$ parts as a composition $\la$ of $m+n$ parts and write
$$\la=(\lambda^{(0)}|\lambda^{(1)})=(\lambda^{(0)}_1,\lambda^{(0)}_2,\cdots,\lambda^{(0)}_m|\lambda^{(1)}_1,
\lambda^{(1)}_2,\cdots,\lambda^{(1)}_n)$$ to indicate the``even'' and ``odd'' parts of $\la$. Let
\begin{equation*}
\begin{aligned}
\Lambda(m|n,r)&=\{\lambda=(\lambda^{(0)}|\lambda^{(1)}) \mid\la\in\La(m+n,r)\}\\
&=\bigcup_{r_1+r_2=r}(\La(m,r_1)\times\La(n,r_2))\\
\end{aligned}
\end{equation*}
By identifying $\Lambda(m|n,r)$ with $\Lambda(m+n,r)$, the notations $W_\la$, $\sD_\la$, $\sD_{\la\mu}$, etc., are all well defined for all $\la,\mu\in\La(m|n,r)$.

\begin{notation} \label{xxx} For $\rho=(\rho_1,\rho_2,\ldots,\rho_t)\models r$, we will sometimes write the parabolic subgroups
$$W_\rho=W_{\rho_1}\times W_{\rho_2}\times\cdots \times W_{\rho_t}$$
where $W_{\rho_1}=\fS_{\{1,\ldots,\rho_1\}}$, $W_{\rho_2}=\fS_{\{\rho_1+1,\ldots,\rho_1+\rho_2\}}$, and so on. For $\la_{(i)},\mu_{(i)}\models\rho_i$, the notation $\sD_{\la_{(i)}\mu_{(i)}}$, etc., are defined relative to $W_{\rho_i}$. In particular, for
$\lambda=(\lambda^{(0)}|\lambda^{(1)})\in\Lambda(m|n,r)$, we make the following notational convention:
\begin{equation}\label{notation}
W_\la=
W_{\la^{(0)}}W_{\la^{(1)}}=W_{\la^{(0)}}\times W_{\la^{(1)}},
\end{equation}
 where
$W_{\lambda^{(0)}}\leq\fS_{\{1,2,\ldots,m\}}$
and
$W_{\lambda^{(1)}}\leq\fS_{\{m+1,\ldots,m+n\}}$ are the even and odd parts of $W_\la$, respectively.
\end{notation}

\begin{lemma}\label{00-11}
For $\lambda=(\lambda^{(0)}\mid\lambda^{(1)}),\mu=(\mu^{(0)}\mid\mu^{(1)})\in\Lambda(m|n,r)$,  $d\in\mathcal{D}_{\la}$, and $d'\in\mathcal{D}_{\mu}$
 let $W_{\la d\cap\mu d'}:=W_\la^d\cap W_\mu^{d'}$ and $W_{\la d\cap\mu d'}^{ij}=W_{\la^{(i)}}^d\cap W_{\mu^{(j)}}^{d'}$ for all $i,j\in\{0,1\}$.  Then
 \begin{equation*}
W_{\la d\cap\mu d'}=W_{\la d\cap\mu d'}^{00} \times W_{\la d\cap\mu d'}^{11}\times W_{\la d\cap\mu d'}^{01}\times
W_{\la d\cap\mu d'}^{10}.
\end{equation*}
Moreover, if $s\in S\cap W_{\la d\cap\mu d'}$, then $s\in W_{\la d\cap\mu d'}^{ij}$ for some $i,j\in\{0,1\}$.
In particular, any parabolic subgroup $W_\eta$ of $W_{\la d\cap\mu d'}$ is a product of parabolic subgroups $W_{\eta^{ij}}=W_\eta\cap W_{\la d\cap\mu d'}^{ij}$ with $i,j\in\{0,1\}$.
\end{lemma}

\begin{proof} The assertion about direct product is clear. For $s\in S$, if $s\in W_\la^d$, then $s=d^{-1}s'd$ or $ds=s'd$ for some $s'\in W_\la$. Thus, $1+\ell(d)\leq\ell(s')+\ell(d)=\ell(ds)\leq 1+\ell(d)$ forces $\ell(s')=1$. Hence, $s'\in S$ and
$s'\in W_{\la^{(0)}}$ or $s'\in W_{\la^{(1)}}$ and so, $s\in W_{\la^{(0)}}^d$ or $s\in W_{\la^{(1)}}^d$.
Similarly, one proves that $s\in W_{\mu^{(0)}}^{d'}$ or $s\in W_{\mu^{(1)}}^{d'}$. Hence, $s\in W_{\la d\cap\mu d'}^{ij}$ for some $i,j\in\{0,1\}$.
The rest of the proof is clear.
\end{proof}
For $d\in\mathcal{D}_{\la\mu}$,  $W_{\la d\cap \mu}$ itself is a parabolic subgroup which
is decomposed into parabolic subgroups
 \begin{equation}\label{ladmu00-11}
W_{\la d\cap\mu}=W_{\la d\cap\mu}^{00} \times W_{\la d\cap\mu}^{11}\times W_{\la d\cap\mu}^{01}\times
W_{\la d\cap\mu}^{10}.
\end{equation}
In this case, the composition $\nu=\la d\cap \mu$ has the form $\nu=(\nu^1,\ldots,\nu^{m+n})$ with $\nu^i\in\La(m+n,\mu_i)$.

We say that $d$ satisfies the {\it even-odd trivial intersection property} if $W_{\la d\cap\mu}^{01}=1=W_{\la d\cap\mu}^{10}$. For $\lambda,\mu\in\Lambda(m|n,r)$, let
\begin{equation}\label{Dcirc}
\mathcal{D}^\circ_{\lambda\mu}=\{d\in\mathcal{D}_{\lambda\mu}\mid
W^d_{\lambda^{(0)}}\cap
W_{\mu^{(1)}}=1,W^d_{\lambda^{(1)}}\cap W_{\mu^{(0)}}=1\}.
\end{equation}
This set is the super version of the usual $\sD_{\la\mu}$.
For $d\in\mathcal{D}^\circ_{\lambda\mu}$, if we put
$W_{\nu^{(0)}}=W^d_{\lambda^{(0)}}\cap
W_{\mu^{(0)}},W_{\nu^{(1)}}=W^d_{\lambda^{(1)}}\cap
W_{\mu^{(1)}},\nu=(\nu^{(0)}|\nu^{(1)})$, then
$W_\nu=W_{\nu^{(0)}}\times W_{\nu^{(1)}}$. We have, in general, $\nu\in\La(m'|n',r)$ where
$m'=m(m+n)$ and $n'=n(m+n)$.

The Hecke algebra $\mathcal{H}_R=\sH_R(W)$ corresponding to $W=\fS_r$ is a free
$R$-module with basis $\{\mathcal{T}_w; w\in W\}$ and the
multiplication is defined by the rules: for $s\in S$,
\begin{equation}
\mathcal{T}_w\mathcal{T}_s=\left\{\begin{aligned} &\mathcal{T}_{ws},
&\mbox{if } \ell(ws)>\ell(w);\\
&(q-q^{-1})\mathcal{T}_w+\mathcal{T}_{ws}, &\mbox{otherwise}.
\end{aligned}
\right.
\end{equation}
Note that $\{T_w=q^{\ell(w)}\sT_w\}_{w\in W}$ is the usual defining basis satisfying relations
\begin{equation}
{T}_w{T}_s=\left\{\begin{aligned} &{T}_{ws},
&\mbox{if } \ell(ws)>\ell(w);\\
&(q^2-1)T_w+q^2{T}_{ws}, &\mbox{otherwise}.
\end{aligned}
\right.
\end{equation}

If $W'$ is a parabolic subgroup of $W$, then the $R$-module $\sum_{w\in
W'}R\mathcal{T}_w$ is a subalgebra of $\mathcal{H}_R$, which is
called a {\it parabolic subalgebra} of $\mathcal{H}_R$, denoted by
$\mathcal{H}_{W'}$. We will use the abbreviation
$\mathcal{H}_\lambda$ instead of $\mathcal{H}_{W_\lambda}$.

Let $W'$ be a parabolic subgroup of $W$ and let $d_{W'}$ denote the {\it Poincar\'e
polynomial} of $W'$, i.e.,
$$d_{W'}=d_{W'}(\bsu)=\sum_{w\in W'} \bsu^{\ell(w)}.$$ In particular, the Pioncar\'e polynomial of $W$ has the form $d_W=[\![r]\!]^{!}:=[\![1]\!][\![2]\!]\cdots[\![r]\!]$, where
$[\![i+1]\!]=1+\bsu+\bsu^2+\cdots+\bsu^i.$ For later use in \S9 define, for $0\leq t\leq s$, the Gaussian polynomials by
$$\left[\!\!\left[t\atop s\right]\!\!\right]=\frac{[\![s]\!]^!}{[\![t]\!]^![\![s-t]\!]^!}.$$
Thus, $d_W(q^2)=0=[\![l]\!]_{q^2}$ if $r\geq l$, where $q$ is a primitive $l$th root of 1 with $l$ odd.

For $\lambda=(\lambda^{(0)}\mid\lambda^{(1)})\in\Lambda(m|n,r)$,
define (see \eqref{notation})
$$x_{\lambda^{(0)}}=\sum_{w\in W_{\lambda^{(0)}}}T_w,\qquad y_{\lambda^{(1)}}=\sum_{w\in
W_{\lambda^{(1)}}}(-q^2)^{-\ell(w)}T_w,$$ where
$T_w=q^{\ell(w)}\mathcal{T}_w$.  We call the module $x_{\lambda^{(0)}}y_{\lambda^{(1)}}\mathcal{H}_{R}$ a {\it signed $q$-permutation module} (cf. \cite{D}).

\begin{definition}\label{S(m|n,r)} Let $\fT_R(m|n,r)=\bigoplus_{\lambda\in\Lambda(m|
n,r)}x_{\lambda^{(0)}}y_{\lambda^{(1)}}\mathcal{H}_{R}.$
The algebra
$$\sS_R(m|n,r):=
\End_{\mathcal{H}_R}(\fT_R(m|n,r))$$ is called a $q$-{\it Schur superalgebra} over
$R$ on which the $\mathbb Z_2$-graded structure is induced from the $\mathbb Z_2$-graded structure on $\fT_R(m|n,r)$ with
$$\fT_R(m|n,r)_i=\bigoplus_{{\lambda\in\Lambda(m|
n,r)}\atop {|\la^{(1)}|\equiv i(\text{mod}2)}}x_{\lambda^{(0)}}y_{\lambda^{(1)}}\mathcal{H}_{R} \;\;(i=0,1).$$

If $R=\sZ:={\mathbb Z}[\up,\up^{-1}]$, we simply write $\sS(m|n,r)$, $\fT(m|n,r)$, etc. for $\sS_\sZ(m|n,r)$, $\fT_\sZ(m|n,r)$, etc.
\end{definition}
Note that it is proved in \cite{DR} that $\sS_R(m|n,r)\cong\sS(m|n,r)\otimes_\sZ R$.

Following \cite{DR}, define, for $\lambda,\mu\in\Lambda(m|n,r)$ and
$d\in\mathcal{D}^\circ_{\lambda\mu}$,
$$T_{W_\lambda d W_\mu}:=\sum_{\substack{w_0w_1
\in W_\mu\cap \mathcal{D}_\nu,\\
w_0\in W_{\mu^{(0)}},w_1\in W_{\mu^{(1)}}}}(-q^2)^{-\ell(w_1)}x_{\lambda^{(0)}}y_{\lambda^{(1)}}T_dT_{w_0}T_{w_1}.$$
There exists $\mathcal{H}_{R}$-homomorphism
$\psi^d_{\lambda\mu}$ such that
$$(x_{\alpha^{(0)}}y_{\alpha^{(1)}}h)\psi_{\mu,\la}^d=\delta_{\mu,\alpha}T_{W_\lambda d
W_\mu}h, \forall \alpha\in\Lambda(m|
n,r),h\in\mathcal{H}_R.$$
Note that we changed the left hand notation $\phi^d_{\lambda\mu}$ used in \cite[(5.7.1)]{DR}
to the right hand notation $\psi^d_{\mu,\la}$ here for notational simplicity later on.\footnote{In this paper, for the Hom spaces $\Hom_{\sH_\la}(M,N)$ of {\it right} $\sH_\la$-modules $M,N$, the right hand function notation will be used to easily read the induced bimodule structure on the space: $(m)(xfy)=((m.x)f).y$ with $x,y\in\sH_\la,m\in M$, and $f\in\Hom_{\sH_\la}(M,N)$.}

The following result is given in {\cite[5.8]{DR}}. We will provide a different proof below in \S5 by using relative norms.
\begin{lemma}\label{DR5.8} The set $\{\psi^d_{\mu\lambda}\mid
\lambda,\mu\in\Lambda(m|n, r), d\in\mathcal{D}^\circ_{\lambda\mu}\}$
forms a $R$-basis for $\sS_R(m|n,r )$.
\end{lemma}

\begin{remark}\label{m'|n' case} For $m\leq m'$ and $n\leq n'$, we may embed $\La(m|n,r)$ into $\La(m'|n',r)$ by adding zeros at the end of each of the two sequences $\la^{(0)}$ and $\la^{(1)}$ for every $\la\in\La(m|n,r)$. Let $\vep=\sum_{\la\in\La(m|n,r)}\psi_{\la\la}^1$. Then $\sS(m|n,r)\cong \vep\sS(m'|n',r)\vep$. Thus, we simply regard $\sS(m|n,r)$ as a centralizer subalgebra of $\sS(m'|n',r)$.
\end{remark}

\begin{proposition}[{\cite[7.5,8.1]{DR}}]\label{DR8.1} Let $R=\mathbb{Q}(\up)$.

(1) The non-isomorphic irreducible $\sS_R(m|n,r)$-modules are indexed by the set
$$\La^+(r)_{m|n}=\{(\la_1,\la_2,\ldots)\in\La^+(r)\mid \la_{m+1}\leq n\}.$$

(2) Assume $m+n\geq r$.
 The $\sS_R(m|n,r )$-$\mathcal{H}_R$ bimodule
structure $\fT_R(m|n,r )$ satisfies the following double centralizer
property
$$\sS_R(m|n,r )=\End_{\mathcal{H}_R}(\fT_R(m|
n,r ))\mbox{ and } \mathcal{H}_R\cong \End_{\sS_R(m|
n,r )}(\fT_R(m|n,r )).$$ Moreover,  there is a
category equivalence
$$\Hom_{\mathcal{H}_R}(-,\fT_R(m|
n,r)):\modh\mathcal{H}_R\longrightarrow \sS_R(m|
n,r )\hmod.$$
\end{proposition}

We now describe $\sS_R(m|n,r )$ in terms of tensor superspaces.

Let $V_R(m|n)$ be a free $R$-module of rank $m+n$
with basis $v_1,v_2,\cdots,v_{m+n}$. The {\it parity map}, by setting
$\hat{i}={0}$ if $1\leq i\leq m$, and $\hat{i}={1}$
otherwise, gives $V_R(m|n)$ a $\mathbb{Z}_2$-graded structure:  $V_R(m|n)=V_0\oplus V_1$, where
$V_0$ is spanned by $v_1,\cdots,v_m$, and $V_1$ is spanned by
$v_{m+1},\cdots,v_{m+n}$. Thus, $V_R(m|n) $ becomes a ``superspace''.

Let
$$I(m|n,r)=\{\bsi=(i_1,i_2,\cdots,i_r)\in \mathbb{N}^r\mid 1\leq
i_j\leq m+n,\forall j\}.$$

For each $\lambda=(\lambda^{(0)}\mid\lambda^{(1)} )\in \Lambda(m|
n,r)$, define $\bsi_\lambda\in I(m|n,r)$ by
$$
\bsi_\lambda=(\underbrace{1,\ldots,1}_{\la^{(0)}_1},\ldots,\underbrace{m,\ldots,m}_{\la^{(0)}_m},\underbrace{m+1,\ldots,m+1}_{\la^{(1)}_1}\ldots,\underbrace{m+n,\ldots,m+n}_{\la^{(1)}_n}).$$ For convenience, denote
$v_\lambda:=v_{\bsi_\lambda}$.

The symmetric group $W=\fS_r$ acts on $I(m|n, r)$ by place
permutation. For $w\in W, \bsi\in I(m|n,r)$
$$\bsi w=(i_{w(1)},i_{w(2)},\cdots,i_{w(r)}).$$

For each $\bsi\in I(m|n,r)$, define $\lambda\in \Lambda (m|n,r)$
to be the weight $wt(\bsi)$ of $v_\bsi$ be setting $\lambda_k=\#\{k\mid
i_j=k,1\leq j\leq
r\},\lambda^{(0)}=(\lambda_1,\lambda_2,\cdots,\lambda_m),\lambda^{(1)}=(\lambda_{m+1},\lambda_{m+2},\cdots,\lambda_{m+n}).$

For $\bsi=(i_1,i_2,\cdots,i_r)\in I(m|
n,r)$, let
$$v_\bsi=v_{i_1}\otimes v_{i_2}\otimes\cdots \otimes v_{i_r}=v_{i_1}v_{i_2}\cdots v_{i_r}.$$
Clearly, the set $\{v_\bsi\}_{\bsi\in I(m|n,r)}$ form a basis of $V_R(m|
n)^{\otimes r}$.

Following \cite{M}, $V_R(m|n,r)^{\otimes r}$ is a right $\sH_R$-module with the following action:
\begin{equation}\label{action2}
v_\bsi\mathcal{T}_{s_k}=
\begin{cases}
(-1)^{\widehat{i_k}\widehat{i_{k+1}}} v_{\bsi s_k},
&\mbox{ if }
i_k<i_{k+1};\\
q v_\bsi, &\mbox{ if } i_k=i_{k+1}\leq m;\\
(-q^{-1}) v_\bsi, &\mbox{ if } i_k=i_{k+1}\geq m+1;\\
(-1)^{\widehat{i_k}\widehat{i_{k+1}}} v_{\bsi
s_k}+(q-q^{-1}) v_\bsi,&\mbox{ if } i_k>i_{k+1}.
\end{cases}
\end{equation}
where $s_k=(k,k+1)$.

Note that when $n=0$, this action coincides with the action on the usual tensor space as given in \cite[(14.6.4)]{DDPW}, commuting with the action of quantum $\mathfrak{gl}_n$.

\begin{definition}\label{dhat}
For $\lambda\in\Lambda(m|n,r), d\in \mathcal{D}_\lambda$, if
$\bsi_\lambda d=(i_1,i_2,\cdots,i_r)$, define
$$(\la,d)^\wedge =\sum_{k=1}^{r-1}\sum_{k<l,i_k>i_l}\widehat{i_k}\widehat{i_l}.$$
When $\la$ is clear from the context, we write $\whd=(\la,d)^\wedge$.
\end{definition}
The following relation will be repeatedly used in the sequel.
\begin{equation}\label{diaoyudao}
(-1)^{\whd}(-1)^{\widehat{i_k}\widehat{i_{k+1}}}=(-1)^{\widehat{ds_k}}\text{ for all }
d\in\mathcal{D}_\lambda, s_k\in S\text{ with }ds_k\in\mathcal{D}_\lambda.
\end{equation}


\begin{lemma}[{\cite[8.3, 8.4]{DR}}]\label{iso f}
The right
$\mathcal{H}_R$-module $V_R(m|n)^{\otimes r}$  is isomorphic to the
$\mathcal{H}_R$-module $\fT_R(m|n,r)$ defined in Definition \ref{S(m|n,r)} under the map:
\begin{equation*}
 f:V_R(m| n)^{\otimes r}\longrightarrow
\bigoplus_{\lambda\in\Lambda(m|
n,r)}x_{\lambda^{(0)}}y_{\lambda^{(1)}}\mathcal{H}_R,\quad(-1)^{\hat{d}}v_{\bsi_\lambda d}\longmapsto
x_{\lambda^{(0)}}y_{\lambda^{(1)}}\mathcal{T}_d.
\end{equation*}
This isomorphism induces a superalgebra isomorphism
$$\sS_R(m| n,r)\cong \End_{\mathcal{H}_R}(V_R(m| n)^{\otimes
r}),$$
sending $g\in\sS_R(m|n,r)$ to $fgf^{-1}$ for all $g\in\sS_R(m|n,r)$.
\end{lemma}

\begin{proof}
The action on a signed $q$-permutation module is given by
\begin{equation}
x_{\lambda^{(0)}}y_{\lambda^{(1)}}\mathcal{T}_d\mathcal{T}_{s_k}=
\begin{cases}
x_{\lambda^{(0)}}y_{\lambda^{(1)}}\mathcal{T}_{ds_k},&\mbox{ if
}ds_k\in\mathcal{D}_\lambda;\\
qx_{\lambda^{(0)}}y_{\lambda^{(1)}}\mathcal{T}_d, &\mbox{ if }
ds_k=s_ld, s_l\in W_{\lambda^{(0)}};\\
-q^{-1}x_{\lambda^{(0)}}y_{\lambda^{(1)}}\mathcal{T}_d, &\mbox{ if
} ds_k=s_ld, s_l\in W_{\lambda^{(1)}};\\
x_{\lambda^{(0)}}y_{\lambda^{(1)}}\mathcal{T}_{ds_k}+(q-q^{-1})x_{\lambda^{(0)}}y_{\lambda^{(1)}}\mathcal{T}_d,&\mbox{
if } ds_k<d.
\end{cases}
\end{equation}
Now, the relation
$(-1)^{\hat{d}}(-1)^{\widehat{i_k}\widehat{i_{k+1}}}=(-1)^{\widehat{ds_k}}$ implies that
$f$
is a right $\mathcal{H}_R$-module homomorphism.\end{proof}

From now on, we will identify
$\sS_R(m| n,r)$ with $\End_{\mathcal{H}_R}(V_R(m| n)^{\otimes r})$.

\section{Relative norms: the first properties}

In 1977, P. Hoefsmit and L. Scott introduced the notion of relative norms, which is the $q$-analogue of relative traces in group representations (see, e.g., \cite{Lan}) and use it to investigate the representation theory of Hecke algebras.
The following material is taken from their unpublished manuscript. A proof can now be found in \cite{J}.

\begin{definition}\label{rel norm}
(1) Let $\lambda,\mu$ be the compositions of $r$ such that
$W_\lambda\leq W_\mu$. Let $M$ be an
$\mathcal{H}_\mu$-$\mathcal{H}_\mu$-bimodule and $b\in M$. Define the {\it
relative norm}
$$N_{W_\mu,W_\lambda}(b)=\sum_{w\in\mathcal{D}_\lambda\cap W_\mu}
\mathcal{T}_{w^{-1}}b\mathcal{T}_w.$$

(2) For an $\mathcal{H}_\lambda$-$\mathcal{H}_\lambda$-bimodule $M$, we
define
$$Z_M(\mathcal{H}_\lambda)=\{m\in M\mid hm=mh,\forall h\in \mathcal{H}_\lambda\}.$$\end{definition}

It is easy to see that
$Z_{\mathcal{H}_\lambda}(\mathcal{H}_\lambda)$ is the center of
$\mathcal{H}_\lambda$ and, for $M=\Hom_R(N,N)$ where $N$ is
a right $\mathcal{H}_\lambda$-module,
$Z_M(\mathcal{H}_\lambda)=\Hom_{\mathcal{H}_\lambda}(N,N)$.
In particular, if $M=\End_R(V_R(m|n)^{\otimes r})$, then $Z_M(\mathcal{H}_R)=\sS_R(m|n,r)$.

We first list some properties of relative norms.

\begin{lemma}\label{RNT}
Let $M$ be an $\mathcal{H}_R$-$\mathcal{H}_R$-bimodule and
let $W_\lambda$ and $W_\mu$ be parabolic subgroups of $W$.
\begin{itemize}
\item[(a)] (Transitivity) If $W_\mu\leq W_\lambda$ and $b\in M$, then
$$N_{W,W_\mu}(N_{W_\mu,W_\lambda}(b))=N_{W,W_\lambda}(b).$$

\item[(b)] $N_{W,W_\lambda}(Z_M(\mathcal{H}_\lambda))\subseteq
Z_M(\mathcal{H}_{R}).$

\end{itemize}
\end{lemma}

Next, we list the $q$-analogues of four useful results known as Mackey decompsition,  Frobenius reciprocity,  Nakayama relation, and Higman criterion for relative projectivity. See \cite{Lan} or \cite{Feit} for their classical version for groups.

\begin{lemma}[Mackey decomposition]\label{Mackey}
 If $N$ is an
$\mathcal{H}_{R}$-$\mathcal{H}_\lambda$-bimodule, then
$$(N\otimes_{\mathcal{H}_\lambda}\mathcal{H}_R)|_{\mathcal{H}_\mu}\cong
\bigoplus_{d\in
\mathcal{D}_{\lambda\mu}}(N\otimes_{\mathcal{H}_\lambda}\mathcal{T}_d)\otimes_{\mathcal{H}_{\nu}}\mathcal{H}_\mu,$$
where $\nu$ is defined by $W_{\nu}=W^d_\lambda\cap W_\mu$ for
all $d\in \mathcal{D}_{\lambda\mu}$.
In particular,  if $M$ is an $\mathcal{H}_R$-$\mathcal{H}_R$-bimodule and $b\in Z_M(\mathcal{H}_\mathcal{\lambda})$, then
$$N_{W,
W_\lambda}(b)=\sum_{d\in \mathcal{D}_{\lambda\mu}}N_{W_\mu,
W_{\la d\cap\mu}}(\mathcal{T}_{d^{-1}}b\mathcal{T}_d).$$
\end{lemma}

The Frobenius reciprocity simply follows from the fact that induction is a left adjoint functor to restriction.

\begin{lemma}[Frobenius reciprocity]\label{Frobenius}
Let $M$ by an $\sH_R$-module and $N$ be an $\sH_\mu$-module. Then there is an $R$-module isomorphism
$$\varphi:\Hom_{\sH_\mu}(N, M|_{\sH_\mu})\overset\sim\longrightarrow \Hom_{\sH_R}(N\otimes_{\sH_\mu}\sH_R,M).$$
In particular, for $\la,\mu\in\La(m|n,r)$ and tensors $v_{\nu}=v_{\bsi_\nu}\in V_R(m|n)^{\otimes r}$ ($\nu=\la$ or $\mu$), the $R$-module isomorphism
$$\varphi:\Hom_{\mathcal{H}_\mu}(Rv_\mu,v_\lambda\mathcal{H}_R)\longrightarrow
\Hom_{\mathcal{H}_R}(v_\mu\mathcal{H}_R,v_\lambda\mathcal{H}_R)$$
induced from the Frobenius reciprocity is the restriction of the relative norm map $N_{W,W_\mu}(-)$, and so
$$N_{W,W_\mu}(\Hom_{\mathcal{H}_\mu}(Rv_\mu,v_\lambda\mathcal{H}_R))=
\Hom_{\mathcal{H}_R}(v_\mu\mathcal{H}_R,v_\lambda\mathcal{H}_R).$$\end{lemma}

The proof of the last assertion is almost identical to that in \cite[Lem.~2.5]{DU}  and is omitted.

\begin{lemma}[Nakayama relation]\label{Nakayama}Let $M$ be an $\mathcal{H}_R$-$\mathcal{H}_R$-bimodule.
If $N$ is an $\mathcal{H}_{R}$-$\mathcal{H}_\lambda$-submodule of $M$ such that $M\cong
N\otimes_{\mathcal{H}_\lambda}\mathcal{H}_{R}$. Then
$$Z_M(\mathcal{H}_{R})=N_{W,W_\lambda}(Z_N(\mathcal{H}_\lambda)).$$
Moreover, if $W_\mu=W^d_\lambda$ for some $d\in
\mathcal{D}_{\lambda\mu}$, then there exists an
$\mathcal{H}_{R}$-$\mathcal{H}_\mu$-submodule $N'\cong
N\otimes_{\mathcal{H}_\lambda}\mathcal{T}_d$ of $M$ such that
$$N_{W, W_\lambda}(Z_N(\mathcal{H}_\lambda))=N_{W, W_\mu}(Z_{N'}(\mathcal{H}_\mu)).$$
\end{lemma}

If $X$ is an $\sH_R$-module, $Y$ is an $\sH_\la$-module,
$$M=\Hom_R(X,Y\otimes_{\sH_\la}\sH_R),\;\text{ and }\;
N=\Hom(X_{\sH_\la},Y),$$
then the usual Nakayama relation (see, e.g., \cite[2.6]{DJ}) now becomes
$$N_{W, W_\lambda}(\Hom_{\sH_\la}(X|_{\sH_\la},Y))=\Hom_{\sH_R}(X,Y\otimes_{\sH_\la}\sH_R).$$



\begin{definition}
A right $\mathcal{H}_R$-module $M$ is {\it projective relative}
to $\mathcal{H}_\lambda$ or simply $\mathcal{H}_\lambda$-projective
if for every pair of right $\mathcal{H}_R$-modules
$M',M{''}$ the exact sequence
$$0\rightarrow M'\rightarrow M{''}\rightarrow M\rightarrow 0$$
split provided it is a split exact sequence as
$\mathcal{H}_\lambda$-modules.
\end{definition}

In the following result, the
notation $X\mid Y$ means that $X$ is isomorphic to a direct summand
of $Y$.

\begin{lemma}[Higman criterion]\label{Higman}
Let $M$ be a right $\mathcal{H}_R$-module. Then the
following are equivalent:
\begin{itemize}
\item[(a)] $M$ is $\mathcal{H}_\lambda$-projective;

\item[(b)]$M\mid M\otimes_{\mathcal{H}_\lambda} \mathcal{H}_R$;

\item[(c)]$M\mid U\otimes_{\mathcal{H}_\lambda} \mathcal{H}_R$
for some right $\mathcal{H}_\lambda$-module $U$;

\item[(d)]$N_{W,W_\lambda}(\Hom_{\mathcal{H}_\lambda}(M,M))=\Hom_{\mathcal{H}_R}(M,M)$.
\end{itemize}
\end{lemma}

Let $M$ be a finitely generated indecomposable right
$\mathcal{H}_R$-module. Then by \cite[3.35]{J}, there
exists a parabolic subgroup $W_\lambda$ of $W$ unique up to
conjugation such that $M$ is $\mathcal{H}_\lambda$-projective and
such that $W_\lambda$ is $W$-conjugate to a parabolic subgroup of
any parabolic subgroup $W_\mu$ of $W$ for which $M$ is
$\mathcal{H}_\mu$-projective. We call $W_\lambda$ a {\it vertex} of $M$ which is unique up to conjugation.

This notion is a generalisation of the vertex theory in the representation theory of finite groups.
Motivated from the fact that a vertex must be a $p$-subgroup, we need the notion of  $l$-parabolic subgroups. Let $l$ be a positive odd number and $l\leq r$. A
parabolic subgroup $W_\lambda$ is called $l$-{\it parabolic} if all
parts of $\lambda$ are 0, $1$, or $l$.

Write $r=sl+t$ with $0\leq t<l$ and let $P_{(r)}$ be the parabolic subgroup of $W=\fS_r$ associated with the composition $(l^s,1^t)$.  This is called in \cite{DU} `the' maximal $l$-parabolic subgroup of $W$. For any composition $\la=(\la_1,\ldots,\la_a)$ of $r$, let the maximal $l$-parabolic subgroup of $W_\la$ be the parabolic subgroup\footnote{The $l$-parabolic subgroups play a role similar to $p$-subgroups in
group representation theory. So the notation $P$ for an $l$-parabolic subgroup indicates this similarity.}
\begin{equation}\label{max l-parabolic}
P_\la=P_{(\la_1)}\times\cdots\times P_{(\la_a)}.
\end{equation}
A maximal $l$-parabolic subgroup of
$W_\lambda$ is a parabolic subgroup $P$ of $W_\la$ such that $P=_{W_\la}P_\la$ (meaning $P^x:=x^{-1}Px=P_\la$ for some $x\in W_\la$).

 Let
$\Phi_l=\Phi_l(\bsu)$ denote the $l$th cyclotomic polynomial.  The Pioncar\'e polynomial of an $l$-parabolic subgroup $W_\la$ has the form $d_{W_\la}=(d_l)^s$, where $s$ is the number of parts $l$  in
$\la$ and
$$d_l=\prod^{l-1}_{i=1}(1+\bsu+\bsu^2+\cdots+\bsu^i).$$
Here are a few simple facts.
\begin{lemma}[{\cite[1.1,1.2]{DU}}]\label{DU1.1}
Let $\lambda$ be a composition of $r$ and $x\in W$. Then

(a) If $W^x_\lambda$ is parabolic, then
$d_{W^x_\lambda}=d_{W_\lambda}$,

(b) $\Phi_l\nmid (d_{W_\lambda}/d_{P_\lambda} )$.

(c)
Let $W_\lambda,W_\mu,W_\theta$ be parabolic subgroups of $W$ such
that
$$W_\theta\leq W^x_\mu, W_\mu\leq W^y_\lambda$$
where $W^x_\mu,W^y_\lambda$ are parabolic and $x,y\in W$. Assume
$d_{P_\theta}\neq d_{P_\lambda}$. Then $\Phi_l\mid
(d_{W_\lambda}/d_{W_\theta})$.
\end{lemma}

The following result is a $q$-analogue of the fact for group representations that a vertex must be a $p$-group and
will be repeatedly used later on.

\begin{lemma}[{\cite[Th.~3.1]{DU3}}] Let $F$ be a field of characteristic 0 in which $q$ is a primitive $l$th root of 1 (with $l$ odd). If $M$ is an indecomposable $\sH_F$-module, then the vertex of $M$ is an $l$-parabolic subgroup.,
\end {lemma}

We will describe the vertices of indecomposable  direct summand of the $\sH_F$-module $V_F(m|n)^{\otimes r}$ in \S 10 in terms of the defect group of the corresponding primitive idempotent and apply this to classify all simple $\sS_F(m|n,r)$-modules when $m+n\geq r$.

\section{Some vanishing properties of relative norms on $V_R(m|n)^{\otimes r}$}

As seen from the remarks right after Definition \ref{rel norm}, the $q$-Schur superalgebra $\sS_R(m|n,r)= Z_M(\sH_R)$ where $M=\End_R(V_R(m|n)^{\otimes r})$.
Thus, by Lemma \ref{RNT}(b), we can construct elements in $\sS_R(m|n,r)$ by applying relative norms to the matrix units in $M$. We first show in this section that some of these elements are simply the 0 transformation. We will construct a basis from this type of elements in the next section.

For $\bsi,\bsj\in I(m| n,r)$, we define $e_{\bsi,\bsj}\in \End_R(V_R(m|
n)^{\otimes r})$ to be the linear map
\begin{equation}
(v_{\bsi'})e_{\bsi,\bsj}=\left\{
\begin{aligned}
&v_\bsj,\mbox{ if }\bsi'=\bsi,\\
&0,\mbox{otherwise}.
\end{aligned}
\right.
\end{equation}
If $(\bsi,\bsj)=(\bsi_\mu,\bsi_\lambda d)$ with $d\in
\mathcal{D}_{\lambda\mu}$, we use the abbreviation
$e_{\mu,\lambda d}$ instead of $e_{\bsi_\mu, \bsi_\lambda d}$.
The following result is obvious from the definition, but will be useful later on.

\begin{lemma}\label{SRNP} For $b\in V_R(m|n)^{\otimes r}=\bigoplus_{\la\in\La(m|n,r)}v_\la\sH_R$, if
the projection of $b$ on $v_\mu\mathcal{H}_R$ is $0$ for some $\mu\in\La(m|n,r)$,
 then $(b)N_{W,W_{\la d\cap \mu}}(f)=0$
 for all $\lambda\in\Lambda(m|n,r)$, $d\in\mathcal{D}^\circ_{\lambda\mu},$ and $f\in
\Hom_{\mathcal{H}_{\la d\cap\mu}}(v_\mu\mathcal{H}_R,v_\lambda\mathcal{H}_R)$,
extended  (by sending other $v_\nu\sH_R$ to 0 for all $\nu\neq \mu$) to an element in $\End_{\sH_{\la d\cap\mu}}(V_R(m|n)^{\otimes r})$. Moreover,
$N_{W,W_\mu}(e_{\mu,\mu})$ is the identity map on $v_\mu\mathcal{H}_R$ and
$0$ elsewhere.
\end{lemma}
\begin{proof} By \eqref{action2},
$v_\bsi\mathcal{H}_R=\mathrm{span}\{v_{\bsi w}\mid w\in W\}$ for any
$\bsi\in I(m|n,r)$. Hence, if the
projection of $b$ on $v_\mu\mathcal{H}_R$ is $0$, then $(b)N_{W,W_{\la d\cap\mu}}(f)=0$.

Note that if $x\in\mathcal{D}_\mu$ and $x\neq 1$, then $x^{-1}\notin
W_\mu$. Hence
\begin{equation*}
\begin{aligned}
(v_\mu)N_{W,W_\mu}(e_{\mu\mu})
&=(v_\mu)\sum_{w\in\mathcal{D}_\mu}\mathcal{T}_{w^{-1}}e_{\mu\mu}\mathcal{T}_w\\
&=(v_\mu)\sum_{w\in\mathcal{D}_\mu\cap W_{\mu}}\mathcal{T}_{w^{-1}}e_{\mu\mu}\mathcal{T}_w=v_\mu.
\end{aligned}
\end{equation*}
Hence, $N_{W,W_\mu}(e_{\mu\mu})$ is the identity on
$v_\mu\mathcal{H}_R$. By the proof above, it is $0$ elsewhere.
\end{proof}

The place permutation action of $W$ on $I(m|n,r)$ induces an action on $I(m|n,r)^2$: $(\bsi,\bsj)w=(\bsi w,\bsj w)$ for all $\bsi,\bsj\in I(m|n,r)$ and $w\in W$. Clearly, if $\bsi=\bsi_\la d$ and $\bsj=\bsi_\mu d'$ for some $d\in\sD_\la$ and $d'\in \sD_\mu$, then
$$\text{Stab}_W(\bsi,\bsj):=\{w\in W\mid (\bsi,\bsj)w=(\bsi,\bsj)\} =W_\la^d\cap W_\mu^{d'}.$$

The following result is a super version of  \cite[Lem.~2.2]{DU}.

\begin{lemma}\label{super2.2}
Let $\bsi=\bsi_\la d$ and $\bsj=\bsi_\mu d'$, where $d\in\sD_\la$ and $d'\in \sD_\mu$, and let $s\in
W_{\la d\cap\mu d'}^{ij}\cap S$, where $i,j\in\{0,1\}$.  Then $\sT_se_{\bsi,\bsj}=e_{\bsi,\bsj}\sT_s$ if and only if $(i,j)=(0,0)$ or $(1,1)$.
\end{lemma}
\begin{proof} Suppose $s=(a,a+1)$. Since $\bsi s=\bsi$ and $\bsj s=\bsj$, by the definition of the action \eqref{action2}, $(v_{\bsi'})\sT_se_{\bsi,\bsj}\neq 0$ if and only if $\bsi'=\bsi $ or $\bsi'=\bsi s$, or equivalently, $\bsi'=\bsi$. Now, write $s=d^{-1}s'd$ or $d'^{-1}s''d'$ for some $s'\in W_\la\cap S$ or $s''\in W_\mu\cap S$. If $(i,j)=(0,0)$, then $s'\in W_{\la^{(0)}}$, $s''\in W_{\mu^{(0)}}$, and
$$(v_\bsi)\sT_se_{\bsi,\bsj}=(-1)^{\widehat d}v_{\bsi_\la}\sT_d\sT_se_{\bsi,\bsj}=(-1)^{\widehat d}v_{\bsi_\la}\sT_{s'}\sT_de_{\bsi,\bsj}=qv_\bsj=v_\bsj\sT_{s}=(v_\bsi) e_{\bsi,\bsj}\sT_s.$$
Similarly, if  $(i,j)=(1,1)$, then $s'\in W_{\la^{(1)}}$, $s''\in W_{\mu^{(1)}}$, and
$$(v_\bsi)\sT_se_{\bsi,\bsj}=(-q^{-1})v_\bsj=v_\bsj\sT_{s}=(v_\bsi) e_{\bsi,\bsj}\sT_s.$$

However, if $(i,j)=(0,1)$, then $s'\in W_{\la^{(0)}}$, $s''\in W_{\mu^{(1)}}$ and
$$(v_\bsi)\sT_se_{\bsi,\bsj}=qv_\bsj\text{ and }(v_\bsi) e_{\bsi,\bsj}\sT_s=-q^{-1}v_\bsj.$$
Hence, $\sT_se_{\bsi,\bsj}\neq e_{\bsi,\bsj}\sT_s$. The proof for the (1,0) case is similar.\end{proof}

\begin{corollary}\label{SPC}
Let $\lambda,\mu\in \Lambda(m|n,r),d\in\mathcal{D}_{\lambda\mu}$, and $W_{\nu}=W^d_\lambda\cap W_\mu$.
\begin{itemize}
\item[(1)] If $d\in\mathcal{D}_{\lambda\mu}^\circ$, then $e_{\mu,\lambda d}\in
\End_{\mathcal{H}_{\nu}}(V_R(m|n)^{\otimes r})$. In particular,  we have
$$N_{W,W_{\nu}}(e_{\mu,\lambda d})\in
\End_{\mathcal{H}_R}(V_R(m|n,r)^{\otimes r}).$$

\item[(2)] If $d\in\mathcal{D}_{\lambda\mu}\setminus\mathcal{D}^\circ_{\lambda\mu}$,
then $e_{\mu,\lambda d}\notin
\End_{\mathcal{H}_{\nu}}(V_R(m|n)^{\otimes r})$.
\end{itemize}
\end{corollary}
\begin{proof} If $d\in
\mathcal{D}^\circ_{\lambda\mu}$, we have $W^d_\lambda\cap
W_\mu=W_{\nu}=W_{\nu^{(0)}}\times W_{\nu^{(1)}} $
where $W_{\nu^{(0)}}=W^d_{\lambda^{(0)}}\cap
W_{\mu^{(0)}},W_{\nu^{(1)}}=W^d_{\lambda^{(1)}}\cap
W_{\mu^{(1)}}$.
By the lemma above, for
 $w=w_0w_1\in W_{\nu}$ with $w_i\in W_{\nu^{(i)}},$
\begin{equation*}
\begin{aligned}
(v_\mu\mathcal{T}_w)e_{\mu,\lambda
d}&=q^{\ell(w_0)}(-q^{-1})^{\ell(w_1)} (v_\mu)e_{\mu,\lambda d}\\
&= v_{\bsi_\lambda d}\sT_{w_0}\sT_{w_1}\\
&=(v_\mu)e_{\mu,\la d}\sT_{w_0}\sT_{w_1}.
\end{aligned}
\end{equation*}
Hence, $\mathcal{T}_we_{\mu,\lambda d}= e_{\mu,\lambda
d}\mathcal{T}_w,$
proving $e_{\mu,\lambda d}\in
\End_{\mathcal{H}_{\nu}}(V_R(m|n)^{\otimes r})$.
The last assertion in (1) follows from Lemma \ref{RNT}(b).

If $d\in\mathcal{D}_{\lambda\mu}\backslash\mathcal{D}^\circ_{\lambda\mu}$, by
definition of $\mathcal{D}^\circ_{\lambda\mu}$, we have
$$W^d_{\lambda^{(0)}}\cap W_{\mu^{(1)}}\neq 1\mbox{ or }
W^d_{\lambda^{(1)}}\cap W_{\mu^{(0)}}\neq 1.$$
Hence, by the lemma above, there exists $s\in W_\nu$ such that $\sT_se_{\bsi,\bsj}\neq e_{\bsi,\bsj}\sT_s$.\end{proof}

The following vanishing property is somewhat surprising. Recall the notation introduced in \eqref{ladmu}.

\begin{theorem}\label{SRTC}  For $\lambda,\mu \in \Lambda(m|n,r)$ and
$d\in\mathcal{D}_{\lambda\mu}\setminus
\mathcal{D}^\circ_{\lambda\mu}$, if a parabolic subgroup $W_\eta\leq W_{\la d\cap\mu}^{00}\times
W_{\la d\cap\mu}^{11}$, then $N_{W,W_\eta}(e_{\mu,\lambda
d})=0$.
\end{theorem}
\begin{proof} 
By the hypothesis and Lemma \ref{super2.2}, $e_{\mu,\la d}\in\End_{\sH_\eta}(V_R(m|n)^{\otimes r})$. Hence, by Lemma \ref{RNT}(b),
$N_{W,W_{\eta}}(e_{\mu,\lambda d})\in \End_{\mathcal{H}_R}(V_R(m|n)^{\otimes r})$.
From the definition of $e_{\mu,\lambda d}$, it is enough by Lemma \ref{SRNP}  to
consider the action of $N_{W,W_\eta}(e_{\mu,\lambda d})$ on $v_\mu$. Thus, we have
\begin{equation*}
\begin{aligned}
(v_\mu)N_{W,W_\eta}(e_{\mu,\lambda d})&=(v_\mu)\sum_{w\in
\sD_\eta}\mathcal{T}_{w^{-1}} e_{\mu,\lambda d}\mathcal{T}_w\\
&=\sum_{w_0w_1\in \sD_\eta\cap(W_{\mu^{(0)}}\times W_{\mu^{(1)}})}(v_\mu)
\mathcal{T}_{(w_0w_1)^{-1}} e_{\mu,\lambda d}\mathcal{T}_{w_0w_1}\\
&=\sum_{w_0w_1\in \sD_\eta\cap(W_{\mu^{(0)}}\times
W_{\mu^{(1)}})}q^{\ell(w_0)}(-q^{-1})^{\ell(w_1)} v_{\bsi_\lambda
d}\mathcal{T}_{w_0w_1}
\end{aligned}
\end{equation*}
Since $d\in\mathcal{D}_{\lambda\mu}\setminus
\mathcal{D}^\circ_{\lambda\mu}$,  we have
$$W_{\nu}:=W^d_\lambda\cap W_\mu=W_\nu^{00}\times W_\nu^{01}\times W_\nu^{10}\times W_\nu^{11},$$
where $W_\nu^{ij}=W^d_{\lambda^{(i)}}\cap
W_{\mu^{(j)}},$ for all $i,j=0,1$, and
$W_\mu=W_{\nu}(\mathcal{D}_{\nu}\cap W_\mu)$.  For
$w_0\in \sD_\eta\cap W_{\mu^{(0)}}$, there exist $x_0x_1\in
(\sD_\eta\cap W^{00}_\nu)\times W^{10}_\nu$ and
$d_0\in \mathcal{D}_{\nu}\cap W_{\mu^{(0)}}$ such that
$w_0=x_0x_1d_0$. For $w_1\in W_{\mu^{(1)}} $, there are $y_0y_1\in
W^{01}_\nu \times(\sD_\eta\cap W^{11}_\nu)$ and
$d_1\in \mathcal{D}_{\nu}\cap W_{\mu^{(1)}}$ such that
$w_1=y_0y_1 d_1$. Therefore, $w_0w_1=x_0x_1y_0y_1d_0d_1$ and
$\ell(w_0w_1)=\ell(x_0)+\ell(x_1)+\ell(y_0)+\ell(y_1)+\ell(d_0)+\ell(d_1)$.
Consequently, we have
\begin{equation*}
\begin{aligned}
&(v_\mu)N_{W,W_\eta}(e_{\mu,\lambda d})\\
&=\sum_{\substack{x_0x_1\in(\sD_\eta\cap W^{00}_\nu) \times W^{10}_\nu\\
y_0y_1\in W^{01}_\nu\times(\sD_\eta\cap W^{11}_\nu)\\
d_0\in\mathcal{D}_{\nu(d)}\cap W_{\mu^{(0)}},d_1\in
\mathcal{D}_{\nu(d)}\cap W_{\mu^{(1)}}} }
q^{\ell(x_0x_1d_0)}(-q^{-1})^{\ell(y_0y_1d_1)}v_{\bsi_\lambda
d}\mathcal{T}_{x_0x_1}\mathcal{T}_{y_0y_1}\mathcal{T}_{d_0}\mathcal{T}_{d_1}\\
&=\sum_{\substack{x_0x_1\in(\sD_\eta\cap W^{00}_\nu) \times W^{10}_\nu\\
y_0y_1\in W^{01}_\nu\times(\sD_\eta\cap W^{11}_\nu)\\
d_0\in\mathcal{D}_{\nu(d)}\cap W_{\mu^{(0)}},d_1\in
\mathcal{D}_{\nu(d)}\cap W_{\mu^{(1)}} } }
(-1)^{\widehat{dd_0d_1}-\widehat{d}}(-1)^{\ell(x_1)+\ell(y_0)}q^{2\ell(x_0)-2\ell(y_1)}v_{\bsi_\lambda
dd_0d_1}\\
&=\sum_{\substack{d_0\in\mathcal{D}_{\nu(d)}\cap
W_{\mu^{(0)}},\\
d_1\in \mathcal{D}_{\nu(d)}\cap W_{\mu^{(1)}}
}}(-1)^{\widehat{dd_0d_1}-\widehat{d}}\sum_{\substack{x_0\in\sD_\eta\cap W_\nu^{00}\\
y_1\in\sD_\eta\cap W^{11}_\nu }
}q^{2\ell(x_0)-2\ell(y_1)}\sum_{\substack{x_1\in W^{10}_\nu\\
y_0\in W^{01}_\nu}}(-1)^{\ell(x_1)+\ell(y_0)}v_{\bsi_\lambda dd_0d_1}.
\end{aligned}
\end{equation*}
Now, the sum $\sum_{(x_1,y_0)\in W^{10}_\nu \times W^{01}_\nu
}(-1)^{\ell(x_1)+\ell(y_0)}$ is the value at $-1$ of a Poincar\'e polynomial of a product of symmetric groups.
Obviously, it is zero if and only if  $ W^{10}_\nu \times W^{01}_\nu \neq
1$.
However, the hypothesis $d\in \mathcal{D}_{\lambda\mu}\setminus
\mathcal{D}^\circ_{\lambda\mu}$ implies  $ W^{10}_\nu \times W^{01}_\nu \neq
1$. Hence,  we have $N_{W,W_\eta}(e_{\mu,\lambda d})=0$.
\end{proof}

\begin{corollary}\label{MC}
For $\lambda,\mu\in\Lambda(m|n,r),y\in \mathcal{D}_\lambda$, if
$W^y_{\lambda^{(0)}}\cap W_{\mu^{(1)}}\neq 1$ or $
W^y_{\lambda^{(1)}}\cap W_{\mu^{(0)}}\neq1$, then
$N_{W,1}(e_{\mu,\lambda y})=0$ in $\sS_R(m|n,r)$.
If, in addition, a parabolic subgroup $W_\eta$ is a subgroup
of $(W_{\la^{(0)}}^y\cap W_{\mu^{(0)}})\times(W_{\la^{(1)}}^y\cap W_{\mu^{(1)}})
$, then we have $N_{W,W_\eta}(e_{\mu,\lambda
y})=0$.
\end{corollary}
\begin{proof} For the double coset $W_\lambda y W_\mu$ of $W$,
there is a unique distinguish element $d\in
\mathcal{D}_{\lambda\mu}$ such that $W_\lambda y W_\mu=W_\lambda d
W_\mu$.

Since $W_\lambda y
W_\mu=W_\lambda\times\{d\}\times(\mathcal{D}_{\nu}\cap W_\mu)$ (as sets),
there is $x\in\mathcal{D}_{\nu}\cap W_\mu$ such that $y=dx$ where
$W_{\nu}=W^d_{\lambda}\cap W_{\mu}$.

Without loss of generality, we assume $W^y_{\lambda^{(0)}}\cap
W_{\mu^{(1)}}\neq 1$. Since $W^y_{\lambda^{(0)}}\cap
W_{\mu^{(1)}}=W^{dx}_{\lambda^{(0)}}\cap
W_{\mu^{(1)}}=(W^d_{\lambda^{(0)}}\cap W_{\mu^{(1)}})^x\neq 1$,
$W^{d}_{\lambda^{(0)}}\cap W_{\mu^{(1)}}\neq 1$. Thus,
$d\in\mathcal{D}_{\lambda\mu}\setminus\mathcal{D}^\circ_{\lambda\mu}$
in this case.

In order to get the claim, it is enough as above to consider the action of
$N_{W,1}(e_{\mu,\lambda y})$ over $v_\mu$. That is,
\begin{equation*}
\begin{aligned}
&(v_\mu)N_{W,1}(e_{\mu,\lambda y})\\
&=\sum_{w\in W}(v_\mu\mathcal{T}_{w^{-1}}) e_{\mu,\lambda
y}\mathcal{T}_w\\
&=\sum_{w\in W_\mu}(v_\mu\mathcal{T}_{w^{-1}}) e_{\mu,\lambda
y}\mathcal{T}_w\\
&=\sum_{w_0w_1\in W_{\mu^{(0)}}\times
W_{\mu^{(1)}}}(v_\mu
\mathcal{T}_{(w_0w_1)^{-1}}) e_{\mu,\lambda
y}\mathcal{T}_{w_0w_1}\\
&=\sum_{w_0w_1\in W_{\mu^{(0)}}\times
W_{\mu^{(1)}}}q^{\ell(w_0)}(-q^{-1})^{\ell(w_1)}v_{\bsi_\lambda
y}\mathcal{T}_{w_0}\mathcal{T}_{w_1}\\
&=\sum_{w_0w_1\in W_{\mu^{(0)}}\times
W_{\mu^{(1)}}}(-1)^{\hat{y}-\hat{d}}q^{\ell(w_0)}(-q^{-1})^{\ell(w_1)}v_{\bsi_\lambda
d}\mathcal{T}_x\mathcal{T}_{w_0}\mathcal{T}_{w_1}\\
&=(-1)^{\hat{y}-\hat{d}}v_{\bsi_\lambda
d}\mathcal{T}_x\sum_{w_0w_1\in W_{\mu^{(0)}}\times
W_{\mu^{(1)}}}q^{\ell(w_0)}(-q^{-1})^{\ell(w_1)}\mathcal{T}_{w_0}\mathcal{T}_{w_1}.
\end{aligned}
\end{equation*}

Since $$\sum_{w_0w_1\in W_{\mu^{(0)}}\times
W_{\mu^{(1)}}}q^{\ell(w_0)}(-q^{-1})^{\ell(w_1)}\mathcal{T}_{w_0}\mathcal{T}_{w_1}=x_{\mu^{(0)}}y_{\mu^{(1)}},$$
it follows that
$$\mathcal{T}_xx_{\mu^{(0)}}y_{\mu^{(1)}}=
q^{-\ell(x)}q^{2\ell(x_0)}(-1)^{\ell(x_1)}x_{\mu^{(0)}}y_{\mu^{(1)}}=q^{\ell(x_0)}(-q^{-1})^{\ell(x_1)}x_{\mu^{(0)}}y_{\mu^{(1)}},$$
where $x=x_0x_1$ with $x_0\in
W_{\mu^{(0)}}$ and $x_1\in W_{\mu^{(1)}}$ (so
$\ell(x)=\ell(x_0)+\ell(x_1)$).
Hence, by Theorem \ref{SRTC},
$$(v_\mu)N_{W,1}(e_{\mu,\lambda y})=(-1)^{\hat{y}-\hat{d}}q^{\ell(x_0)}(-q^{-1})^{\ell(x_1)}(v_\mu)N_{W,1}(e_{\mu,\lambda d})=0,$$
proving the first assertion.

For last assertion, the argument does not carry over. However, we first give a proof for the $\up$-Schur super algebra $\sS(m|n,r)$ (i.e., the $R=\sZ$ case).

By Lemmas \ref{00-11} and \ref{super2.2}, for any $s\in S$, $\mathcal{T}_se_{\mu,\la y}=e_{\mu,\la y}\mathcal{T}_s$ if and only if
$s\in (W^y_{\lambda^{(0)}}\cap W_{\mu^{(0)}})\times
(W^y_{\lambda^{(1)}}\cap W_{\mu^{(1)}})$.
Now, if $W_\eta=W_{\eta^{(0)}}\times W_{\eta^{(1)}}$ with $W_{\eta^{(i)}}\leq W_{\la y\cap\mu}^{ii}$ ($i=0,1$), then, by Lemma
\ref{super2.2},
$$\sT_ze_{\mu,\la y}=e_{\mu,\la y}\sT_z=\up^{\ell(z_0)}(-\up^{-1})^{\ell(z_1)} e_{\mu,\la y}$$
where $z=z_0z_1$ with $z_i\in W_{\eta^{(i)}}$ ($i=0,1$). Thus,
$$\aligned
0=N_{W,1}(e_{\mu,\la y})&=N_{W,W_\eta}(N_{W_\eta,1}(e_{\mu,\la y}))\\
&=d_{W_{\eta^{(0)}}}(\bsu)d_{W_{\eta^{(1)}}}(\bsu^{-1})N_{W,W_\eta}(e_{\mu,\la y}).
\endaligned$$
Hence, $N_{W,W_\eta}(e_{\mu,\la y})=0$ in $\sS(m|n,r)$. The general case follows from base change,
noting $\sS_R(m|n,r)\cong\sS(m|n,r)\otimes_\sZ R$.\end{proof}

\section{Bases for $\sS_R(m|n,r)$ in relative norms}

By Lemma \ref{iso f}, we will always identify $\sS_R(m|n,r)$ with $\End_{\sH_R}(V_R(m|n)^{\otimes r})$. For $\lambda,\mu\in\Lambda(m|n,r)$ and $d\in\mathcal{D}^\circ_{\lambda\mu}$, let
$$N_{\mu\la}^d:=N_{W,W_{\la d\cap\mu}}(e_{\mu,\lambda d}).$$
Then, by Corollary \ref{SPC}(1),
$N_{\mu\la}^d\in\sS_R(m|n,r )$.

\begin{theorem}\label{RNB} The set
$$\mathcal{B}=\{N_{\mu\la}^d\mid
\lambda,\mu\in\Lambda(m|n,r),d\in\mathcal{D}^\circ_{\lambda\mu}\}$$ forms a basis of the $q$-Schur superalgebra
$\sS_R(m|n,r )$.
\end{theorem}
\begin{proof} Consider the $\mathcal{H}_\mu$-$\mathcal{H}_\mu$-bimodule
$$M=\Hom_R(Rv_\mu,Rv_{\bsi_\lambda
d}\otimes_{\mathcal{H}_{\nu}}\mathcal{H}_\mu),$$ where
$d\in\mathcal{D}_{\lambda\mu}$ and $\nu=\la d\cap \mu$. Now,
$M$ contains an $\mathcal{H}_\mu$-$\mathcal{H}_\nu$-submodule
$N=\Hom_R(Rv_\mu,Rv_{\bsi_\lambda d})$ and
 $M\cong N\otimes_{\mathcal{H}_{\nu}}\mathcal{H}_\mu$.
Thus, by Nakayama relation in Lemma \ref{Nakayama},
$$N_{W_\mu,W_{\nu}}(\Hom_{\mathcal{H}_{\nu}}(Rv_\mu,Rv_{\bsi_\lambda
d}))=\Hom_{\mathcal{H}_\mu}(Rv_\mu,Rv_{\bsi_\lambda
d}\otimes_{\mathcal{H}_{\nu}}\mathcal{H}_\mu).$$
This together with an application of
Mackey decomposition yields
\begin{equation*}
\begin{aligned}
\Hom_{\mathcal{H}_\mu}(Rv_\mu,v_\lambda\mathcal{H}_R)&=\bigoplus_{d\in\mathcal{D}_{\lambda\mu}}\Hom_{\mathcal{H}_\mu}(Rv_\mu,Rv_\lambda\otimes
\mathcal{T}_d\otimes_{\mathcal{H}_{\nu}}\mathcal{H}_\mu)\\
&=\bigoplus_{d\in\mathcal{D}_{\lambda\mu}}\Hom_{\mathcal{H}_\mu}(Rv_\mu,Rv_{\bsi_\lambda
d}\otimes_{\mathcal{H}_{\nu}}\mathcal{H}_\mu)\\
&=\bigoplus_{d\in\mathcal{D}_{\lambda\mu}}N_{W_\mu,W_{\nu}}(\Hom_{\mathcal{H}_{\nu}}(Rv_\mu,Rv_{\bsi_\lambda
d})).
\end{aligned}
\end{equation*}

We claim that, for
$d\in\mathcal{D}_{\lambda\mu}\setminus\mathcal{D}^\circ_{\lambda\mu}$,
$\Hom_{\mathcal{H}_{\nu}}(Rv_\mu,Rv_{\bsi_\lambda d})=0$. Indeed, in this case,
there exists $w\in(W^d_{\lambda^{(0)}}\cap W_{\mu^{(1)}})\times
(W^d_{\lambda^{(1)}}\cap W_{\mu^{(0)}})$ such that $\ell(w)=1$. Assume
$w\in W^d_{\lambda^{(0)}}\cap W_{\mu^{(1)}}$. For $f\in
\Hom_{\mathcal{H}_{\nu}}(Rv_\mu,Rv_{\bsi_\lambda d})$,
$(v_\mu)f=av_{\bsi_\lambda d}$ for some $a\in R$. Since
$(v_\mu\mathcal{T}_s)f=(v_\mu)f\mathcal{T}_s$ for all $s\in W_\nu\cap S$, applying this to $s=w$ yields
$aqv_{\bsi_\lambda d}=a(-q^{-1})v_{\bsi_\lambda d}$. Because $R$ is
a domain and $q\neq (-q^{-1})$, $a=0$ and $f=0$.

Thus, by the claim,
$$\Hom_{\mathcal{H}_\mu}(Rv_\mu,v_\lambda\mathcal{H}_R)
=\bigoplus_{d\in\mathcal{D}^\circ_{\lambda\mu}}N_{W_\mu,W_{\nu(d)}}(Rv_\mu,Rv_{\bsi_\lambda
d}).$$
By applying $N_{W,W_\mu}(-)$ to both sides, Lemma \ref{RNT}(a) and Frobenius reciprocity (Lemma \ref{Frobenius}) imply
\begin{equation*}
\begin{aligned}
\Hom_{\mathcal{H}_R}(v_\mu\mathcal{H}_R,v_\lambda\mathcal{H}_R)
&=N_{W,W_\mu}(\Hom_{\mathcal{H}_\mu}(Rv_\mu,v_\lambda\mathcal{H}_R))\\
&=\bigoplus_{d\in
\mathcal{D}^\circ_{\lambda\mu}}N_{W,W_{\nu}}(\Hom_{\mathcal{H}_{\nu}}(Rv_\mu,Rv_{\bsi_\lambda
d})).
\end{aligned}
\end{equation*}
Therefore, $\{N_{W,W_{\nu}}(e_{\mu,\lambda d})\mid
d\in\mathcal{D}^\circ_{\mu\lambda}\}$ forms a basis of
$\Hom_{\mathcal{H}_R}(v_\mu\mathcal{H}_R,v_\lambda\mathcal{H}_R)$.
Hence, $\mathcal{B}$ is a basis of $\sS_R(m|n,r )$. \end{proof}

Now we describe a basis of
$\End_{\mathcal{H}_\rho}(V_R(m|n)^{\otimes r})$ for
$\rho=(\rho_1,\rho_2,\cdots,\rho_t)\models r$, which will be used in \S 7 and \S10.

For $\lambda,\mu\in\Lambda(m|n,r)$, let
$d\in\mathcal{D}_{\la\rho},d'\in\mathcal{D}_{\mu\rho}$, and
 $$W_\alpha=W^d_\la\cap W_\rho\;\;\text{ and }\;\; W_\beta=W^{d'}_{\mu}\cap
W_\rho.$$ Then we have
$\alpha=(\alpha_{(0)},\alpha_{(1)},\cdots,\alpha_{(t)})$ and
$\beta=(\beta_{(0)},\beta_{(1)},\cdots,\beta_{(t)})$ where
$\alpha_{(i)},\beta_{(i)}\in\Lambda{(m|n,\rho_i)}$ for
$i=1,2,\cdots,t$.
Let
$$\mathcal{D}^\circ_{\alpha\beta}\cap W_\rho=\{d=d_1d_2\cdots
d_t\in \mathcal{D}_{\alpha\beta}\cap W_\rho\mid d_i\in
W_{\rho_i}\cap\mathcal{D}^\circ_{\alpha_{(i)}\beta_{(i)}},0\leq i\leq t\}.$$
See Notation \ref{xxx}.
Since $V_R(m|n)^{\otimes r}=V_R(m|n)^{\otimes \rho_1}\otimes
V_R(m|n)^{\otimes \rho_2}\otimes \cdots \otimes V_R(m|n)^{\otimes \rho_t}$,
then
\begin{equation}\label{Schur algebra product}
\sS_R(m|n,\rho):=\End_{\mathcal{H}_\rho}(V_R(m|n)^{\otimes r})\cong
\sS_R(m|n,\rho_1 )\otimes
\sS_R(m|n,\rho_2 )\otimes\cdots\otimes
\sS_R(m|n,\rho_t ).
\end{equation}

Set
$$\mathcal{B}(\lambda,\mu,d,d')=\{N_{W_\rho,W^y_\alpha\cap W_\beta}(e_{\mu d',\lambda dy})\mid y\in \mathcal{D}_{\alpha\beta}^\circ\cap
W_\rho\}.$$

\begin{lemma} The set $\mathcal{B}(\lambda,\mu,d,d')$ forms a basis for the $R$-module
$$\Hom_{\mathcal{H}_\rho}(v_{\bsi_\mu
d'}\otimes_{\mathcal{H}_\beta} \mathcal{H}_\rho, v_{\bsi_\lambda
d}\otimes_{\mathcal{H}_\alpha}\mathcal{H}_\rho).$$
\end{lemma}
\begin{proof} By Lemmas \ref{Frobenius}, \ref{Mackey}, and \ref{Nakayama},
$$\aligned
\Hom_{\mathcal{H}_\rho}(v_{\bsi_\mu
d'}\otimes_{\mathcal{H}_\beta} \mathcal{H}_\rho,& v_{\bsi_\lambda
d}\otimes_{\mathcal{H}_\alpha}\mathcal{H}_\rho)=N_{W_\rho,W_\beta}(\Hom_{\sH_\beta}(v_{\bsi_\mu d'},v_{\bsi_\la d}\otimes_{\sH_\alpha}\sH_\rho|_{\sH_\beta})\\
&=N_{W_\rho,W_\beta}(\Hom_{\sH_\beta}(v_{\bsi_\mu d'},\bigoplus_{y\in
\mathcal{D}_{\alpha\beta}\cap W_\rho}v_{\bsi_\la d}\otimes\sT_y\otimes_{\sH_{\alpha y\cap \beta}}{\sH_\beta})\\
&=\bigoplus_{y\in
\mathcal{D}_{\alpha\beta}\cap W_\rho}N_{W_\rho,W_{\alpha y\cap\beta}}(\Hom_{\mathcal{H}_{\alpha y\cap\beta}}(Rv_{\bsi_\mu d'},Rv_{\bsi_\lambda dy})).\endaligned
$$
Since, for $y\in
(\mathcal{D}_{\alpha\beta}\cap
W_\rho)\setminus(\mathcal{D}^\circ_{\alpha\beta}\cap W_\rho)$,
$\Hom_{\mathcal{H}_{\alpha y\cap \beta}}(Rv_{\bsi_\mu
d'},Rv_{\bsi_\lambda dy})=0$, it follows that
$$\Hom_{\mathcal{H}_\rho}(v_{\bsi_\mu d'}\otimes_{\mathcal{H}_\beta}
\mathcal{H}_\rho, v_{\bsi_\lambda
d}\otimes_{\mathcal{H}_\alpha}\mathcal{H}_\rho)=\bigoplus_{y\in
\mathcal{D}^\circ_{\alpha\beta}\cap
W_\rho}N_{W_\rho,W_{\alpha y\cap\beta}}(\Hom_{\mathcal{H}_{\alpha y\cap\beta}}(Rv_{\bsi_\mu d'},Rv_{\bsi_\lambda dy})),$$
Hence, $\mathcal{B}(\lambda,\mu,d,d')$ forms a basis. \end{proof}

Let $$\mathcal{B}(\rho)=\bigcup_{\substack{\lambda,\mu\in
\Lambda{(m|n,r)}\\
d\in\mathcal{D}_{\la\rho},d'\in\mathcal{D}_{\mu\rho}}}\mathcal{B}(\lambda,\mu,d,d').$$

\begin{theorem}\label{SRNB} The set
$\mathcal{B}(\rho)$ is a basis of
$\sS_R(m|n,\rho)=\End_{\mathcal{H}_\rho}(V_R(m|n)^{\otimes r})$.
\end{theorem}
\begin{proof} By Lemma \ref{iso f} and {\cite[2.22]{J}}, we have
\begin{equation*}
\begin{aligned}
\End_{\mathcal{H}_\rho}(V_R(m|n)^{\otimes
r})&=\bigoplus_{\lambda,\mu\in\Lambda{(m|n,r)}}\Hom_{\mathcal{H}_\rho}(v_\mu\mathcal{H}_R,v_\lambda\mathcal{H}_R)\\
&=\bigoplus_{\substack{\lambda,\mu\in\Lambda{(m|n,r)}\\
d\in\mathcal{D}_{\la\rho},d'\in\mathcal{D}_{\mu\rho}}}\Hom_{\mathcal{H}_\rho}(v_{\bsi_\mu
d'}\otimes_{\mathcal{H}_{\beta}}\mathcal{H}_\rho,v_{\bsi_\lambda
d}\otimes_{\mathcal{H}_{\alpha}}\mathcal{H}_\rho)
\end{aligned}
\end{equation*}
The theorem follows from the above lemma. \end{proof}

However, by \eqref{Schur algebra product}, we may use the bases for $\sS_R(m|n,\rho_i)$ described in Theorem \ref{RNB} to describe a basis for $\sS_R(m|n,\rho)$. We now show that this basis coincides with the basis above.

Let
\begin{equation}\label{vecLa}
\vec{\La}(m|n,\rho)=\Lambda{(m|n,\rho_1)}\times\cdots\times
\Lambda{(m|n,\rho_t)}.\end{equation}
 Then, for $\vec{\lambda}\in
\vec{\La}(m|n,\rho)$, there exist $\lambda_{(i)}\in
\Lambda{(m|n,\rho_i)} $
such that $\vec{\lambda}=(\lambda_{(1)},\ldots,\lambda_{(t)})$.
Let $W_{\vec{\lambda}}$ be the corresponding parabolic
subgroup of $W_\rho$ and, for $\vec{\lambda},\vec{\mu}\in
\vec{\La}(m|n,\rho)$, define the set $\mathcal{D}_{\vec{\lambda}\vec{\mu};\rho}$
of distinguished
double coset representatives by
\begin{equation}\label{scrD}
\mathcal{D}^\circ_{\vec{\lambda}\vec{\mu};\rho}=
\{d\in\mathcal{D}_{\vec{\lambda}\vec{\mu}}\mid
d=d_1\cdots d_t, d_i\in W_{\rho_i}\cap
\mathcal{D}^\circ_{\lambda_{(i)}\mu_{(i)}},1\leq i\leq t\}\subseteq W_\rho.
\end{equation}
Here, again, see Notation \ref{xxx} for the notational convention.
Putting
$$N_{\vec\mu\vec\la}^d=\bigotimes_{i=1}^tN_{W_{\rho_i},W_{\la_{(i)}}^{d_i}\cap W_{\mu_{(i)}}}(e_{\mu_{(i)},\la_{(i)}d_i}),$$ then the set
$\{N^d_{\vec\mu\vec{\lambda}}\mid
\vec{\lambda},\vec{\mu}\in
\vec{\La}(m|n,\rho), d\in
\mathcal{D}^\circ_{\vec{\lambda}\vec\mu;\rho}\}$
 forms a basis of $\sS_R(m|n,\rho)$.

 If we define the multi-index $\bsi_{\vec{\mu}}\in I(m|n,r)$
in an obvious way so that
$$v_{\vec{\mu}}=v_{\bsi_{\vec{\mu}}}=v_{\mu_{(1)}}\otimes v_{\mu_{(2)}}\otimes\cdots\otimes v_{\mu_{(t)}}\in
V_R(m|n)^{\otimes r}$$ and set $e_{\vec{\mu},\vec{\lambda}d}$ to be the
$R$-linear map by sending $v_{\vec{\mu}}$ to
$v_{\bsi_{\vec{\lambda}}d}$ and the other basis vectors of
$V_R(m|n)^{\otimes r}$ to $0$, then we may simply write
$N^d_{\vec\mu\vec{\lambda}}=N_{W_\rho,W_{\vec\la}^d\cap W_{\vec\mu}}(e_{\vec\mu,\vec\la d})$.
Here we regard $\vec\nu$ as a composition of $r$ by concatenation so that $W_{\vec\nu}$ is well defined.

 The coincidence of this basis and the one in Theorem \ref {SRNB} can be seen as follows.

For
$\vec{\lambda}=(\lambda_{(1)},\ldots,\lambda_{(t)})\in\vec{\La}(m|n,\rho)$,
let $\lambda=\lambda_{(1)}+\cdots+\lambda_{(t)}
\in \Lambda{(m|n,r)}$. Then there is a unique
$x\in\mathcal{D}_{\lambda\rho}$ such that
$W_{\vec{\lambda}}=W^x_\lambda\cap W_\rho$. Similarly, for $\vec\mu\in\vec{\La}(m|n,\rho)$,
there is a unique
$y\in\mathcal{D}_{\mu\rho}$ such that
$W_{\vec{\mu}}=W^y_\mu\cap W_\rho$. Thus, for any $z\in\mathcal{D}_{\vec\la\vec\mu;\rho}^\circ$,
putting $W_{xyz}=(W_\la^x\cap W_\rho)^z\cap(W_\mu^y\cap W_\rho)$, we have
\begin{equation}\label{same}N_{W_\rho,W_{\vec\la}^d\cap W_{\vec\mu}}(e_{\vec\mu,\vec\la z})=
N^z_{\vec\mu\vec{\lambda}}=N_{W_\rho,W_{xyz}}(e_{\mu y,\la xz}).\end{equation}

We end this section with a comparison of the basis given in Theorem \ref{RNB} and the one in Lemma \ref{DR5.8}.

\begin{theorem}\label{RTSB}  If we identify
$\fT_R(m|n,r)=\bigoplus_{\lambda\in\Lambda{(m|n,r)}}x_{\lambda^{(0)}}y_{\lambda^{(1)}}\mathcal{H}_R$
with $V_R(m|n)^{\otimes r}$ by the isomorphism $f$ given in Lemma \ref{iso f}, then, for any $\lambda,\mu\in
\Lambda{(m|n,r)},d\in\mathcal{D}^\circ_{\lambda\mu}$, and $\nu=\la d\cap\mu$,
\begin{equation}\label{psi=N}
N_{W,W_{\nu}}(e_{\mu,\lambda d})
=(q^{-1})^{\ell(d)}(-1)^{\hat{d}} \psi^d_{\mu\lambda}.
\end{equation}
\end{theorem}
\begin{proof} We need to prove that $f\psi_{\mu\la}^df^{-1}$ equals the right hand side.
Since $f$ is an $\mathcal{H}_R$-module
isomorphism, it is enough to consider the actions on $v_\mu$. We have on the one hand,
\begin{equation*}
\begin{aligned}
(v_\mu)f\psi^d_{\mu\lambda}f^{-1}&=(x_{\mu^{(0)}}y_{\mu^{(1)}})\psi^d_{\mu\lambda}f^{-1}\\
&=(T_{W_\lambda d W_\mu})f^{-1}\\
&=(\sum_{w_0w_1\in (W_{\mu^{(0)}}\times W_{\mu^{(1)}})\cap
\mathcal{D}_{\nu}}(-q^2)^{-\ell(w_1)}x_{\mu^{(0)}}y_{\mu^{(1)}}T_dT_{w_0}T_{w_1})f^{-1}\\
&=(q^{\ell(d)}\sum_{w_0w_1\in (W_{\mu^{(0)}}\times
W_{\mu^{(1)}})\cap
\mathcal{D}_{\nu}}q^{\ell(w_0)}(-q^{-1})^{\ell(w_1)}x_{\mu^{(0)}}y_{\mu^{(1)}}\mathcal{T}_{dw_0w_1})f^{-1}\\
&=q^{\ell(d)}\sum_{w_0w_1\in (W_{\mu^{(0)}}\times W_{\mu^{(1)}})\cap
\mathcal{D}_{\nu}}(-1)^{\widehat{dw_0w_1}}q^{\ell(w_0)}(-q^{-1})^{\ell(w_1)}v_{\bsi_\lambda
dw_0w_1}.
\end{aligned}
\end{equation*}
On the other hand,
\begin{equation*}
\begin{aligned}
(v_\mu)N_{W,W_{\nu}}(e_{\mu,\lambda
d})&=(v_\mu)\sum_{w\in\mathcal{D}_{\nu}}\mathcal{T}_{w^{-1}}e_{\mu,\lambda
d}\mathcal{T}_w\\
&=\sum_{w_0w_1\in( W_{\mu^{(0)}}\times W_{\mu^{(1)}})\cap
\mathcal{D}_{\nu}}(v_\mu\mathcal{T}_{w_0^{-1}}\mathcal{T}_{w_1^{-1}})e_{\mu,\lambda
d}\mathcal{T}_{w_0}\mathcal{T}_{w_1}\\
&=\sum_{w_0w_1\in (W_{\mu^{(0)}}\times W_{\mu^{(1)}})\cap
\mathcal{D}_{\nu}}q^{\ell(w_0)}(-q^{-1})^{\ell(w_1)}v_{\bsi_\lambda
d}\mathcal{T}_{w_0}\mathcal{T}_{w_1}\\
&=(-1)^{\widehat{d}}\sum_{w_0w_1\in (W_{\mu^{(0)}}\times
W_{\mu^{(1)}})\cap
\mathcal{D}_{\nu}}(-1)^{\widehat{dw_0w_1}}q^{\ell(w_0)}(-q^{-1})^{\ell(w_1)}v_{\bsi_\lambda
dw_0w_1}.
\end{aligned}
\end{equation*}
Therefore, \eqref{psi=N} follows.\end{proof}

In particular, by Theorem \ref{RNB}, the set
$$\{\psi^d_{\mu\lambda}\mid\lambda,\mu\in\Lambda{(m|n,r)},d\in\mathcal{D}^\circ_{\lambda\mu}\}$$
forms a basis for $\sS_R(m|n,r)$, proving Lemma \ref{DR5.8}.

\section{A filtration of ideals of $\sS_R(m|n,r)$}

The purpose of this section is to construct a filtration of ideals of $\sS_R(m|n,r)$ labelled by $l$-parabolic subgroups. The existence of the filtration is based on a nice property of the structure constants associated with the defect groups of the relative norm basis elements in $\mathcal{B}$ (see Theorem \ref{TI}). We need certain technical results Lemmas \ref{DD}--\ref{ML} for the proof of the theorem.

\begin{lemma}\label{RTRN}
Let $\lambda,\mu\in\Lambda{(m|n,r)}$ and
$d\in\mathcal{D}^\circ_{\lambda\mu}$. Then
$$(-1)^{\hat{d}}N_{W,W^d_\lambda\cap W_\mu}(e_{\mu,\lambda
d})=(-1)^{\widehat{d^{-1}}}N_{W,W_\lambda\cap
W^{d^{-1}}_\mu}(e_{\mu d^{-1},\lambda}).$$
\end{lemma}
Note by Definition \ref{dhat} that $\widehat d=(\la,d)^\wedge$ while $\widehat{d^{-1}}=(\mu,d^{-1})^\wedge$.
\begin{proof} Let $W_\nu=W^d_\lambda\cap
W_\mu$ and $W_{\nu'}=W_\lambda\cap W^{d^{-1}}_\mu$.  By Lemma \ref{SRNP}, it is
sufficient to check that
$$(-1)^{\hat{d}}(v_\mu)N_{W,W^d_\lambda\cap W_\mu}(e_{\mu,\lambda
d})=(-1)^{\widehat{d^{-1}}}(v_\mu)N_{W,W_\lambda\cap
W^{d^{-1}}_\mu}(e_{\mu d^{-1},\lambda}).$$
Now,
\begin{equation*}
\begin{aligned}
\text{LHS}&=
(-1)^{\hat{d}}\sum_{w\in\mathcal{D}_\nu\cap
W_\mu}(v_\mu\mathcal{T}_{x^{-1}})e_{\mu,\lambda
d}\mathcal{T}_x\\
&=(-1)^{\hat{d}}\sum_{x_0x_1\in\mathcal{D}_\nu\cap
(W_{\mu^{(0)}}\times W_{\mu^{(1)}})}
(v_\mu\mathcal{T}_{x_0^{-1}}\mathcal{T}_{x_1^{-1}})e_{\mu,\lambda
d}\mathcal{T}_{x_0}\mathcal{T}_{x_1}\\
&=(-1)^{\hat{d}}\sum_{x_0x_1\in\mathcal{D}_ \nu\cap
(W_{\mu^{(0)}}\times
W_{\mu^{(1)}})}q^{\ell(x_0)}(-q^{-1})^{\ell(x_1)}v_{\bsi_\lambda
d}\mathcal{T}_{x_0}\mathcal{T}_{x_1}\\
&=(-1)^{\hat{d}}\sum_{x_0x_1\in\mathcal{D}_ \nu\cap
(W_{\mu^{(0)}}\times
W_{\mu^{(1)}})}(-1)^{\widehat{dx_0x_1}-\hat{d}}q^{\ell(x_0)}(-q^{-1})^{\ell(x_1)}v_{\bsi_\lambda
dx_0x_1}\\
&=\sum_{x_0x_1\in\mathcal{D}_ \nu\cap (W_{\mu^{(0)}}\times
W_{\mu^{(1)}})}(-1)^{\widehat{dx_0x_1}}q^{\ell(x_0)}(-q^{-1})^{\ell(x_1)}v_{\bsi_\lambda
dx_0x_1}.
\end{aligned}
\end{equation*}
By
{\cite[3.1]{DU}}, $\mathcal{D}_{\nu'}\cap d
W_\mu=d(\mathcal{D}_\nu\cap W_\mu)$. Thus,
\begin{equation*}
\begin{aligned}
\text{RHS}&=
(-1)^{\widehat{d^{-1}}}\sum_{y\in d(\mathcal{D}_{\nu}\cap
W_\mu)}(v_\mu\mathcal{T}_{y^{-1}})e_{\mu d^{-1},\lambda}\mathcal{T}_y\\
&=(-1)^{\widehat{d^{-1}}}\sum_{dy_0y_1\in
d(\mathcal{D}_\nu\cap(W_{\mu^{(0)}}\times
W_{\mu^{(1)}}))}(v_\mu\mathcal{T}_{y_0^{-1}}\mathcal{T}_{y_1^{-1}}\mathcal{T}_{d^{-1}})e_{\mu
d^{-1},\lambda}\mathcal{T}_{d}\mathcal{T}_{y_0}\mathcal{T}_{y_1}\\
&=(-1)^{\widehat{d^{-1}}}\sum_{dy_0y_1\in
d(\mathcal{D}_\nu\cap(W_{\mu^{(0)}}\times
W_{\mu^{(1)}}))}q^{\ell(y_0)}(-q^{-1})^{\ell(y_1)}(-1)^{\widehat{d^{-1}}}v_\lambda\mathcal{T}_{dy_0y_1}\\
&=(-1)^{\widehat{d^{-1}}}\sum_{y_0y_1\in
\mathcal{D}_\nu\cap(W_{\mu^{(0)}}\times
W_{\mu^{(1)}})}q^{\ell(y_0)}(-q^{-1})^{\ell(y_1)}(-1)^{\widehat{d^{-1}}}(-1)^{\widehat{dy_0y_1}}v_{\bsi_\lambda
dy_0y_1}\\
&=\sum_{y_0y_1\in \mathcal{D}_\nu\cap(W_{\mu^{(0)}}\times
W_{\mu^{(1)}})}q^{\ell(y_0)}(-q^{-1})^{\ell(y_1)}(-1)^{\widehat{dy_0y_1}}v_{\bsi_\lambda
dy_0y_1}.
\end{aligned}
\end{equation*}
So LHS=RHS, proving the lemma. \end{proof}

{\it For the next three lemmas, we fix the following notations:}
\begin{equation}\label{DD0a}
\aligned
\rho,\lambda,\mu\in&\Lambda{(m|n,r)},d'\in\mathcal{D}_{\rho\lambda}, d\in\mathcal{D}_{\lambda\mu},\text{ and } y\in\mathcal{D}_{\rho\mu},\text{ and define }\nu,\nu',\tau\text{ by }\\
W_{\nu}&=W^d_\lambda\cap W_\mu,\quad W_{\nu'}=W_\rho\cap
W^{d'^{-1}}_\lambda,\;\text{ and }\;W_\tau=W^y_\rho\cap W_\mu.
\endaligned
\end{equation}
Then, $\nu,\nu'\in\La(m'|n',r)$ for some $m'\geq m,n'\geq n$; see remarks right after \eqref{Dcirc}.
\begin{lemma}\label{DD} If $d\in\mathcal{D}^\circ_{\lambda\mu},d'
\in\mathcal{D}^\circ_{\rho\lambda}$, and  $y\in\mathcal{D}^\circ_{\rho\mu}$, then
$$W_\rho y
W_\mu\cap\mathcal{D}_{\nu'\nu}=\{hyk\mid yk\in\mathcal{D}^\circ_{\rho\nu},k\in\mathcal{D}^\circ_{\tau\nu}\cap
W_\mu, h\in \mathcal{D}_{\nu'}\cap W_\rho\}$$
and it is a subset of $\mathcal{D}^\circ_{\nu'\nu}$.
\end{lemma}
\begin{proof} Since
$y\in\mathcal{D}_{\rho\mu}$, by
{\cite[3.2]{DU}}, we have
\begin{equation}\label{DD0b}
W_\rho y
W_\mu\cap\mathcal{D}_{\nu'\nu}=\{hyk\mid yk\in\mathcal{D}_{\rho\nu},k\in\mathcal{D}_{\tau\nu}\cap W_\mu,h\in \mathcal{D}_{\nu'}\cap W_\rho\}.
\end{equation}
Since
$y\in\mathcal{D}^\circ_{\rho\mu},d\in\mathcal{D}^\circ_{\lambda\mu}$,
by definition,
$$\aligned
W^y_{\rho^{(0)}}\cap W_{\mu^{(1)}}=1,\quad&W^y_{\rho^{(1)}}\cap
W_{\mu^{(0)}}=1, \\W_{\nu^{(0)}}=W^d_{\lambda^{(0)}}\cap
W_{\mu^{(0)}},\;&\text{ and }\; W_{\nu^{(1)}}=W^d_{\lambda^{(1)}}\cap W_{\mu^{(1)}}.
\endaligned$$
Thus, noting that $k\in W_\mu$ implies $W_{\mu^{(1)}}^k=W_{\mu^{(1)}}$,
$$W^{yk}_{\rho^{(0)}}\cap
W_{\nu^{(1)}}=W^{yk}_{\rho^{(0)}}\cap(W^d_{\lambda^{(1)}}\cap
W_{\mu^{(1)}})=(W^y_{\rho^{(0)}}\cap
W_{\mu^{(1)}})^{k}\cap W^d_{\lambda^{(1)}}=1.$$
Similarly, one shows $W^{yk}_{\rho^{(1)}}\cap
W_{\nu^{(0)}}=1$. Hence, $yk\in\mathcal{D}^\circ_{\rho\nu}$.

To see $k\in\mathcal{D}^\circ_{\tau\nu}\cap W_\mu$, we have
$$W^{k}_{\tau^{(0)}}\cap W_{\nu^{(1)}}=
 (W^y_{\rho^{(0)}}\cap W_{\mu^{(0)}})^{k}\cap W^d_{\lambda^{(1)}}\cap W_{\mu^{(1)}}=W^{yk}_{\rho^{(0)}}\cap W_{\mu^{(0)}}\cap W^d_{\lambda^{(1)}}\cap W_{\mu^{(1)}}=1 $$
 since $W^{yk}_{\rho^{(0)}}\cap W_{\mu^{(1)}}=1$.
 By a similar proof, we obtain $W^{k}_{\tau^{(1)}}\cap
 W_{\nu^{(0)}}=1$, proving
 $k\in\mathcal{D}^\circ_{\tau\nu}$.

 Finally, we prove $W_\rho y W_\mu\cap \mathcal{D}_{\nu'\nu}\subseteq \mathcal{D}_{\nu'\nu}^\circ$.
 For $hyk\in W_\rho y W_\mu\cap \mathcal{D}_{\nu'\nu}$, we have
 $$W^{hyk}_{\nu'^{(0)}}\cap W_{\nu^{(1)}}=(W_{\rho^{(0)}}\cap
 W^{d'^{-1}}_{\lambda^{(0)}})^{hyk}\cap W^d_{\lambda^{(1)}}\cap
 W_{\mu^{(1)}}=W_{\rho^{(0)}}^{hyk}\cap W^{d'^{-1}hyk}_{\lambda^{(0)}}\cap W^d_{\lambda^{(1)}}\cap
 W_{\mu^{(1)}}.$$
 Since $h\in\mathcal{D}_{\nu'}\cap W_\rho$,
 $W_{\rho^{(0)}}^{hyk}=W_{\rho^{(0)}}^{yk}$. Since $k\in
 W_\mu$, $W_{\rho^{(0)}}^{hyk}\cap W_{\mu^{(1)}}=(W^y_{\rho^{(0)}}\cap
 W_{\mu^{(1)}})^{k}$. Now $W^y_{\rho^{(0)}}\cap
 W_{\mu^{(1)}}=1$,  implies $W^{hyk}_{\nu'^{(0)}}\cap
 W_{\nu^{(1)}}=1$. Similarly, one proves $W^{hyk}_{\nu'^{(1)}}\cap
 W_{\nu^{(0)}}=1$. Hence, $hyk\in \mathcal{D}^\circ_{\nu'\nu}$. \end{proof}

The following is the super version of a modified \cite[Lem.~3.3]{DU}.

\begin{lemma}\label{DD1} Let $\rho,\mu\in\La(m|n,r)$ and let $W_\alpha$ be a parabolic subgroup of $W_\mu$.
For $y\in \mathcal{D}_{\rho\mu}^\circ$ and $k\in\mathcal{D}_{\rho y\cap\mu}\cap W_\mu\cap\sD_\alpha^{-1}$, let
$W_\theta=W^{yk}_\rho\cap W_\alpha$ and $W_{\theta'}=W^y_\rho\cap
W^{k^{-1}}_\alpha=W_\theta^{k^{-1}}$.  Then there exists $c\in R$ such that
$$N_{W,W_\theta}(e_{\mu,\rho yk})=cN_{W,W_{\theta'}}(e_{\mu,\rho y})$$
\end{lemma}
\begin{proof}   Let $z=yk$ and $\tau=\rho y\cap \mu$. The hypothesis $y\in\sD_{\rho\mu}^\circ$ and $W_\alpha\leq W_\mu$ implies $yk\in\sD_{\rho\alpha}^\circ$. Hence, $W_\theta$ is parabolic.
Since $k\in\mathcal{D}_{\tau\alpha}\cap W_\mu$, it follow that $W_{\theta'}=W_\alpha^{k^{-1}}\cap W_\tau$
is parabolic.
Moreover,  both $W_\theta$ and $W_{\theta'}$ are
parabolic subgroups of (parabolic) $W_\rho^y\cap W_\mu$ and of (possibly non-parabolic) $W_\rho^z\cap W_\mu$. By the transitivity of relative
norm Lemma \ref{RNT}(a), it is enough to show that, for some $c\in R$,
$$N_{W_\mu,W_\theta}(e_{\mu,\rho
z})=cN_{W_\mu,W_{\theta'}}(e_{\mu,\rho y}).$$

Now consider the $\mathcal{H}_\mu$-$\mathcal{H}_\mu$-bimodule
$$M=\Hom_R(Rv_\mu,Rv_{\bsi_\rho
y}\otimes_{\mathcal{H}_{\theta'}}\mathcal{H}_\mu).$$
It is clear that there exists an
$\mathcal{H}_\mu$-$\mathcal{H}_{\theta'}$-bimodule
$N=\Hom_R(Rv_\mu,Rv_{\bsi_\rho y})$ and an
$\mathcal{H}_\mu$-$\mathcal{H}_{\theta}$-bimodule
$N'=\Hom_R(Rv_\mu,Rv_{\bsi_\rho z})$ such that $M\cong
N\otimes_{\mathcal{H}_{\theta'}}\mathcal{H}_\mu$ and $M\cong
N'\otimes_{\mathcal{H}_\theta}\mathcal{H}_\mu$ as $\mathcal{H}_\mu$-$\mathcal{H}_\mu$-bimodules.
Applying Lemma \ref{Nakayama} yields
$$N_{W_\mu,W_{\theta'}}(\Hom_{\mathcal{H}_{\theta'}}(Rv_\mu,Rv_{\bsi_\rho
y}))=Z_M(\sH_\mu)=N_{W_\mu,W_\theta}(\Hom_{\mathcal{H}_\theta}(Rv_\mu,Rv_{\bsi_\rho
z})).$$ Hence, there exists $c\in R$ such that
$N_{W,W_\theta}(e_{\mu,\rho z})=N_{W,W_{\theta'}}(ce_{\mu,\rho
y})$. \end{proof}

Note that, in the applications below of this lemma, we will take $W_\alpha=W_\nu$ as defined in \eqref{DD0a} or $W_\alpha=1$.

\begin{lemma}\label{ML}Maintain the notations in \eqref{DD0a} and Lemma \ref{DD}.  For $hyk\in \mathcal{D}^\circ_{\nu'\nu}$, let
$$\aligned
W_\xi&=W^{yk}_\rho\cap W_\nu\leq W_\tau^k,\quad W_\eta=W^{hyk}_{\nu'}\cap W_\nu\leq W_\xi,\\
f(\bsu)&=\frac{d_{W_{\tau^{(0)}}}(\bsu)d_{W_{\tau^{(1)}}}(\bsu^{-1})}{d_{W_{\xi^{(0)}}}(\bsu)d_{W_{\xi^{(1)}}}(\bsu^{-1})},\text{ and }g(\bsu)=\frac{d_{W_{\xi^{(0)}}}(\bsu)d_{W_{\xi^{(1)}}}(\bsu^{-1})}{d_{W_{\eta^{(0)}}}(\bsu)d_{W_{\eta^{(1)}}}(\bsu^{-1})}.\endaligned$$
Then $$N_{W,W^{hyk}_{\nu'}\cap W_\nu}(e_{\mu,\rho
yk})=cf(q^2)g(q^2)N_{W,W^y_\rho\cap W_\mu}(e_{\mu,\rho y})$$ for
some $c\in R$.
Moreover, 
if $q\in R$ is an $l$th primitive root of 1 with $l$ odd and $P_{\tau}\nleq_WP_\nu$ or
$P_{\tau}\nleq_WP_{\nu'}$, then $f(q^2)g(q^2)=0$.
\end{lemma}
\begin{proof} By the transitivity of relative norms (Lemma \ref{RNT}(a)), we
have
$$N_{W,W^{hyk}_{\nu'}\cap W_\nu}(e_{\mu,\rho
yk})=N_{W,W^{yk}_\rho\cap W_\nu}(N_{W^{yk}_\rho\cap
W_\nu,W^{hyk}_{\nu'}\cap W_\nu}(e_{\mu,\rho yk})).$$
By Lemma \ref{DD}, $W_\xi=W_{\rho yk\cap\nu}^{00}\times W_{\rho yk\cap\nu}^{11}\leq
W_{\rho yk\cap\mu}^{00}\times W_{\rho yk\cap\mu}^{11}$. Thus, we have $e_{\mu,\rho yk}\in
\Hom_{\mathcal{H}_\xi}(V_R(m|n)^{\otimes r})$ by Lemma \ref{super2.2}.
Since
\begin{equation*}
\begin{aligned}
(v_\mu)N_{W^{yk}_\rho\cap W_\nu,W^{hyk}_{\nu'}\cap
W_\nu}(e_{\mu,\rho yk})&=(v_\mu)N_{W_\xi,W_\eta}(e_{\mu,\rho yk})\\
&=\sum_{w\in\mathcal{D}_\eta\cap
W_\xi}(v_\mu\mathcal{T}_{w^{-1}})e_{\mu,\rho yk} \mathcal{T}_w\\
&=\sum_{w_0w_1\in\mathcal{D}_\eta\cap(W_{\xi^{(0)}}\times
W_{\xi^{(1)}})}(v_\mu\mathcal{T}_{w_0^{-1}}\mathcal{T}_{w_1^{-1}})e_{\mu,\rho
yk} \mathcal{T}_{w_0}\mathcal{T}_{w_1}\\
&=\sum_{w_0w_1\in\mathcal{D}_\eta\cap(W_{\xi^{(0)}}\times
W_{\xi^{(1)}})}q^{\ell(w_0)}(-q^{-1})^{\ell(w_1)}v_{\bsi_\rho
yk}\mathcal{T}_{w_0}\mathcal{T}_{w_1}\\
&=\sum_{w_0w_1\in\mathcal{D}_\eta\cap(W_{\xi^{(0)}}\times
W_{\xi^{(1)}})}q^{2\ell(w_0)}(-q^{-1})^{2\ell(w_1)}v_{\bsi_\rho yk}
\end{aligned}
\end{equation*}
According to the definition of Poincar\'e polynomial, we have
$$(v_\mu)N_{W_\xi,W_\eta}(e_{\mu,\rho yk})
=\frac{d_{W_{\xi^{(0)}}}(\bsu)d_{W_{\xi^{(1)}}}(\bsu^{-1})}{d_{W_{\eta^{(0)}}}(\bsu)d_{W_{\eta^{(1)}}}(\bsu^{-1})}\bigg|_{\bsu=q^2}v_{\bsi_\rho
yk}
=g(q^2)v_{\bsi_\rho yk}.$$
So $$N_{W,W^{hyk}_{\nu'}\cap W_\nu}(e_{\mu,\rho
yk})=g(q^2)N_{W,W^{yk}_\rho\cap W_\nu}(e_{\mu,\rho yk}).$$

By Lemma \ref{DD1}, there exists $c\in R$ such that
$$N_{W,W^{yk}_\rho\cap W_\nu}(e_{\mu,\rho yk})=cN_{W,W^y_\rho\cap
W^{k^{-1}}_\nu}(e_{\mu,\rho y}).$$ Thus,
\begin{equation}\label{first half}
N_{W,W^{hyk}_{\nu'}\cap W_\nu}(e_{\mu,\rho
yk})=cg(q^2)N_{W,W^y_\rho\cap W^{k^{-1}}_\nu}(e_{\mu,\rho y}).
\end{equation}
By Corollary \ref{SPC},  $e_{\mu,\rho y}\in
\Hom_{\mathcal{H}_\tau}(V_R(m|n)^{\otimes r})$. Since $W^y_\rho\cap
W^{k^{-1}}_\nu\leq W_\tau, $ we have by the transitivity of relative
norms
$$N_{W,W^y_\rho\cap
W^{k^{-1}}_\nu}(e_{\mu,\rho
y})=N_{W,W_\tau}(N_{W_\tau,W^y_\rho\cap
W^{k^{-1}}_\nu}(e_{\mu,\rho y})).$$
Since $(W^{yk}_\rho\cap W_\nu)^{k^{-1}}=W^y_\rho\cap W_\mu\cap
W^{k^{-1}}_\nu=W^{k^{-1}}_\xi$ and
$k\in\mathcal{D}^\circ_{\tau\nu}\cap W_\mu$, $W^{k^{-1}}_\xi$ is a
parabolic subgroup of $W$. By Lemma \ref{DU1.1}(a), we have
$d_{W^y_\rho\cap W^{k^{-1}}_\nu}=d_{W_\xi}$. Hence,
$$N_{W_\tau,W^y_\rho\cap
W^{k^{-1}}_\nu}(e_{\mu,\rho
y})=\frac{d_{W_{\tau^{(0)}}}(\bsu)d_{W_{\tau^{(1)}}}(\bsu^{-1})}{d_{W_{\xi^{(0)}}}(\bsu)d_{W_{\xi^{(1)}}}(\bsu^{-1})}\bigg|_{\bsu=q^2}e_{\mu,\rho
y}=f(q^2)e_{\mu,\rho y}.$$
Applying $N_{W,W_\tau}(\ )$ to both sides and combining it with \eqref{first half} give
$$N_{W,W^{hyk}_{\nu'}\cap W_\nu}(e_{\mu,\rho
yk})=cf(q^2)g(q^2)N_{W,W^y_\rho\cap W_\mu}(e_{\mu,\rho y}),$$
proving the first assertion.

To prove the last assertion,  we first assume that $P_\tau\nleq_W P_\nu$, equivalently, $d_{P_\nu}\mid d_{P_\tau}$  and $d_{P_\nu}\neq
d_{P_\tau}$. Since $W_\xi\leq W_\nu$, it follows that  $d_{P_{\xi^{(0)}}}\mid d_{P_{\nu^{(0)}}}$ and
$d_{P_{\xi^{(1)}}}\mid d_{P_{\nu^{(1)}}}$. Thus, $d_{P_\nu}\neq
d_{P_\tau}$ implies $d_{P_{\xi^{(0)}}}\mid d_{P_{\tau^{(0)}}}$ and
$d_{P_{\xi^{(1)}}}\mid d_{P_{\tau^{(1)}}}$,  one of which is not equal. This implies $f(q^2)=0$, and hence, $f(q^2)g(q^2)=0$. From the argument, we see that, if $P_\xi<_WP_\tau$, then $f(q^2)=0$.

We now assume that $P_\tau\leq_W P_\nu$ but $P_\tau\nleq_W P_{\nu'}$. As seen above, we may further assume $P_\xi=P_\tau$. Then, $d_{P_{\nu'}}|d_{P_\tau}$ and $d_{P_{\nu'}}\neq d_{P_\tau}$. Since $W_\eta\leq_WW_{\nu'}$ and $W_\eta\leq W_\xi$, $P_\xi=P_\tau$ implies that one of the relations $d_{P_{\eta^{(0)}}}|d_{P_{\xi^{(0)}}}$ and $d_{P_{\eta^{(1)}}}|d_{P_{\xi^{(1)}}}$ is not an equality. Hence, $g(q^2)=0$, and hence, $f(q^2)g(q^2)=0$, proving the last assertion.
\end{proof}

Let $N^d_{\mu\lambda}=N_{W,W_{\la d\cap\mu}}(e_{\mu,\lambda d})$ be a
standard basis element as described in Theorem \ref{RNB}. We define the
{\it defect group} of $N^d_{\mu\lambda}$ to be the maximal
$l$-parabolic subgroup $P_{\la d\cap\mu}$ of $W_{\la d\cap\mu}=W_\la^d\cap W_\mu$. We now have the following result about the
coefficients of the product of two standard basis elements.
Recall that $R$ is a domain and $q\in R$ is an $l$-th primitive root of 1 with $l$ odd.
\begin{theorem}\label{TI} For basis elements
$N^d_{\mu\lambda},N^{d'}_{\lambda\rho}\in\mathcal{B}$ of $\sS_R(m|n,r)$, assume that
$$N^d_{\mu\lambda}N^{d'}_{\lambda\rho}=\sum_{y\in\mathcal{D}^\circ_{\rho\mu}}a_y
N^y_{\mu\rho}$$ where $a_y\in R$. If $a_y\neq 0$, then
$$P_{\rho y\cap\mu}\leq_W P_{\lambda d\cap\mu}\mbox{ and
}P_{\rho y\cap\mu}\leq_W P_{\rho d'\cap\la}.$$
\end{theorem}
\begin{proof} We first compute in the $\up$-Schur superalgebra  $\sS(m|n,r)$ over $\sZ=\mathbb{Z}[\up,\up^{-1}]$.
By \ref{RTRN} and Lemma \ref{Frobenius} and noting $\nu=\la d\cap\mu$ and $\nu'=\la d^{\prime -1}\cap\rho$, we have
\begin{equation*}
\begin{aligned}
N^d_{\mu\lambda}N^{d'}_{\lambda\rho}&=(-1)^{\widehat{d'}}(-1)^{\widehat{d'^{-1}}}N_{W,W_\nu}(e_{\mu,\lambda
d})N_{W,W_{\nu'}}(e_{\lambda d'^{-1},\rho})\\
&=(-1)^{\widehat{d'}+\widehat{d'^{-1}}}N_{W,W_\nu}(e_{\mu,\lambda
d}N_{W,W_{\nu'}}(e_{\lambda d'^{-1},\rho}))\\
&=(-1)^{\widehat{d'}+\widehat{d'^{-1}}}N_{W,W_\nu}(e_{\mu,\lambda
d}\sum_{x\in\mathcal{D}_{\nu'\nu}}N_{W_\nu,W^x_{\nu'}\cap
W_\nu}(\mathcal{T}_{x^{-1}}e_{\lambda d'^{-1},\rho}\mathcal{T}_x))\\
&=(-1)^{\widehat{d'}+\widehat{d'^{-1}}}\sum_{x\in\mathcal{D}_{\nu'\nu}}N_{W,W^x_{\nu'}\cap
W_\nu}(e_{\mu,\lambda d}\mathcal{T}_{x^{-1}}e_{\lambda
d'^{-1},\rho}\mathcal{T}_x)
\end{aligned}
\end{equation*}
Now, by \cite[3.2]{DU} (or \eqref{DD0b}),
$$\mathcal{D}_{\nu'\nu}=\{hyk\mid y\in\mathcal{D}_{\rho\la},yk\in\mathcal{D}_{\rho\nu},k\in\mathcal{D}_{\tau\nu}\cap W_\mu,h\in \mathcal{D}_{\nu'}\cap W_\rho\}.$$
For $x\in \mathcal{D}_{\nu'\nu}$ there exist
$y\in\mathcal{D}_{\rho\mu}$ and $h,k$  such
that $x=hyk$ and
$yk\in\mathcal{D}_{\rho\nu},h\in\mathcal{D}_{\nu'}\cap
W_\rho$. By a direct computation, we have for some $b_x\in \sZ$.
$$e_{\mu,\lambda
d}\mathcal{T}_{x^{-1}}e_{\lambda d'^{-1},\rho}\mathcal{T}_x=b_x
e_{\mu,\rho yk}.$$
We claim that  if $b_x\neq0$
\begin{equation}\label{gu1}
\mathcal {T}_ze_{\mu,\rho yk}=e_{\mu,\rho yk}\mathcal{T}_z\text{ for all }z\in W^x_{\nu'}\cap
W_\nu.
\end{equation} Indeed,
$$\aligned
b_x\cdot(\text{LHS})&=\mathcal{T}_z (e_{\mu,\lambda d}\mathcal{T}_{x^{-1}}e_{\lambda
d'^{-1},\rho}\mathcal{T}_x)\\
&=e_{\mu,\lambda d}\mathcal{T}_z\mathcal{T}_{x^{-1}}e_{\lambda
d'^{-1},\rho}\mathcal{T}_x\quad\text{(by Lemma \ref{SPC})}\\
&=e_{\mu,\lambda d}\mathcal{T}_{x^{-1}} \mathcal{T}_{z'} e_{\lambda
d'^{-1},\rho}\mathcal{T}_x\quad(z'\in W_{\nu'}\text{ with }zx^{-1}=x^{-1}z')\\
&=e_{\mu,\lambda d}\mathcal{T}_{x^{-1}}  e_{\lambda
d'^{-1},\rho} \mathcal{T}_{z'} \mathcal{T}_x\\
&= (e_{\mu,\lambda d}\mathcal{T}_{x^{-1}}e_{\lambda
d'^{-1},\rho}\mathcal{T}_x)\mathcal{T}_z  \\
&=b_x\cdot(\text{RHS})
\endaligned$$
Since $W^x_{\nu'}\cap
W_\nu\leq W_\rho^{yk}\cap W_\nu\leq W_\rho^{yk}\cap W_\mu$ and $W_{\nu^{(i)}}=W_{\la d\cap\mu}^{ii}\leq W_{\mu^{(i)}}$ for $i=0,1$, it follows that $W^x_{\nu'}\cap
W_\nu=\prod_{i,j=0,1}W_{\nu' x\cap\nu}^{ij}$ is a product of parabolic subgroups and each $W_{\nu'x\cap\nu}^{ij}\leq
W_{\rho^{(i)}}^{yk}\cap W_{\mu^{(j)}}$.
Thus, if $b_x\neq0$, we conclude that $W^x_{\nu'}\cap
W_\nu=W_{\nu'x\cap\nu}^{00}\times W_{\nu'x\cap\nu}^{11}\leq(W_{\rho^{(0)}}^{yk}\cap W_{\mu^{(0)}})\times( W_{\rho^{(1)}}^{yk}\cap W_{\mu^{(1)}})$,
since any element in $ W_{\nu'x\cap\nu}^{01}\leq W_{\rho^{(0)}}^{yk}\cap W_{\mu^{(1)}} $ or in $W_{\nu'x\cap\nu}^{10}\leq W_{\rho^{(1)}}^{yk}\cap W_{\mu^{(0)}} $ does not satisfy \eqref{gu1}; see Lemma \ref{super2.2}.

Now, if $y\in\mathcal{D}_{\rho\mu}\setminus \mathcal{D}^\circ_{\rho\mu}$, then $ W_{\rho^{(0)}}^{yk}\cap W_{\mu^{(1)}}=(W_{\rho^{(0)}}^{y}\cap W_{\mu^{(1)}} )^k\neq1$  or $ W_{\rho^{(1)}}^{yk}\cap W_{\mu^{(0)}}=(W_{\rho^{(1)}}^{y}\cap W_{\mu^{(0)}} )^k\neq1$. Thus, by Corollary \ref{MC},
 $N_{W,W^x_{\nu'}\cap W_{\nu}}(e_{\mu,\rho yk})=0$.
Thus, for $y\in \mathcal{D}^\circ_{\rho\mu}$,  if we write by Lemma \ref{DD}
$$W_\rho y
W_\mu\cap\mathcal{D}_{\nu'\nu}=\{h_iyk_j\mid 1\leq i\leq m_y, 1\leq j\leq n_y\},$$ then
\begin{equation*}
\begin{aligned}
N^d_{\mu\lambda}N^{d'}_{\lambda\rho}
&=(-1)^{\widehat{d'}+\widehat{d'^{-1}}}\sum_{hyk\in\mathcal{D}^\circ_{\nu'\nu}}N_{W,W^x_{\nu'}\cap
W_\nu}(b_{hyk}e_{\mu,\rho yk})\\
&=(-1)^{\widehat{d'}+\widehat{d'^{-1}}}\sum_{y\in\mathcal{D}^\circ_{\rho\mu}}\sum_{i,j}N_{W,W^{h_iyk_j}_{\nu'}\cap
W_\nu}(b_{h_iyk_j}e_{\mu,\rho yk_j}).
\end{aligned}
\end{equation*}

By Lemma \ref{ML}, there are Laurent polynomials $f_{y,j}(\bsu)$ associated with $yk_j$ and
$g_{i,y,j}(\bsu)$ associated with $h_iyk_j$ such that in $\sS(m|n,r)$,
$$N^d_{\mu\lambda}N^{d'}_{\lambda\rho}=(-1)^{\widehat{d'}+\widehat{d'^{-1}}}\sum_{y\in\mathcal{D}^\circ_{\rho\mu}}(\sum_{i,j}b'_{h_iyk_j}f_{y,j}(\bsu)g_{i,y,j}(\bsu))N_{W,W^y_\rho\cap W_\mu}(e_{\mu,\rho y})
.$$

We now look at the product in $\sS_R(m|n,r)=\sS(m|n,r)\otimes_\sZ R$ by specializing $\up$ to $q\in R$ and obtain
$N^d_{\mu\lambda}N^{d'}_{\lambda\rho}=\sum_{y\in\mathcal{D}^\circ_{\rho\mu}}a_yN_{\mu\rho}^y$
in $\sS_R(m|n,r)$, where
$$a_y=(-1)^{\widehat{d'}+\widehat{d'^{-1}}}\sum_{i=1}^{m_y}\sum_{j=1}^{n_y}b'_{i,y,j}f_{y,j}(q^2)g_{i,y,j}(q^2).$$
By Lemma \ref{ML} again, if $P_{\rho y\cap\mu}\nleq_WP_{\lambda d\cap\mu}$ or
$P_{\rho y\cap\mu}\nleq_W P_{\rho d'\cap\la}$, then all
$f_{y,j}(q^2)g_{i,y,j}(q^2)=0$. Hence, $a_y=0$.  \end{proof}

Let $P$ be an $l$-parabolic subgroup of $W$. We define
\begin{equation}\label{ideal I}
I_R(P,r)=\text{span}\{N^d_{\mu\lambda}\mid \la,\mu\in\La(m|n,r),d\in\sD_{\la\mu}^\circ,P_{\lambda d\cap\mu}\leq_W P\}.
\end{equation}  We also write $I_R(P,r)=I_R(P,r)_{m|n}$ if $m|n$ needs to be mentioned.
By the above theorem, we
have
\begin{corollary}\label{MC2}
The space $I_R(P,r)$ is an ideal of $\sS_R(m|n,r)$.
\end{corollary}

Let $r=sl+t$, where $0\leq t<l$. For each $k$, $0\leq k\leq
s$, let $P_k$ be the $l$-parabolic subgroup of $W$ associated with the composition
$(1^{t+(s-k)l},l^k)$. Then we have a chain of ideals in $\sS_R(m|n,r)$
\begin{equation}\label{filtration}
0\subseteq I_R(P_0,r)\subseteq I_R(P_1,r)\subseteq\cdots\subseteq
I_R(P_s,r)=\sS_R(m|n,r).
\end{equation}
In Definition \ref{defect group}, we will use the sequence to define the defect group of a primitive idempotent and discuss the effect of Brauer homomorphisms on the defect groups.

\section{Alternative characterisation of the ideals $I_F(P, r)$}

{\it For the rest of the paper, we assume that $R=F$ is a field of characteristic 0 and $q\in F$ is a primitive $l$th root of 1 with $l$ odd. }
\begin{lemma}\label{P=1}
We have $$I_F(\{1\},r)=N_{W,1}(\End_F(V_F(m|n)^{\otimes r})).$$
\end{lemma}
\begin{proof} Clearly,
$$\{e_{\mu d,\lambda d'}\mid
\lambda,\mu\in\Lambda{(m|n,r)},d\in\mathcal{D}_\mu,d'\in\mathcal{D}_\lambda\}$$
is a basis of $\End_F(V_F(m|n)^{\otimes r})$. By Lemmas \ref{RNT}, \ref{Mackey}, and
\ref{SRNP}, we have
\begin{equation*}
\begin{aligned}
N_{W,1}(e_{\mu, \lambda
d'})&=N_{W,W_\mu}(e_{\mu\mu})N_{W,1}(e_{\mu
d,\lambda d'})\\
&=N_{W,W_\mu}(e_{\mu\mu}N_{W,1}(e_{\mu
d,\lambda d'}))\\
&=N_{W,W_\mu}(e_{\mu\mu}\sum_{x\in\mathcal{D}^{-1}_\mu}N_{W_\mu,1}(\mathcal{T}_{x^{-1}}e_{\mu
d,\lambda d'}\mathcal{T}_x))\\
&=N_{W,W_\mu}(\sum_{x\in\mathcal{D}^{-1}_\mu}N_{W_\mu,1}(e_{\mu\mu}\mathcal{T}_{x^{-1}}e_{\mu
d,\lambda d'}\mathcal{T}_x))\\
&=\sum_{x\in\mathcal{D}^{-1}_\mu\cap
d^{-1}W_\mu}N_{W,1}(e_{\mu\mu}\mathcal{T}_{x^{-1}}e_{\mu d,\lambda
d'}\mathcal{T}_x)\\
&=N_{W,1}(e_{\mu\mu}\mathcal{T}_d e_{\mu d,\lambda
d'}\mathcal{T}_{d^{-1}}).
\end{aligned}
\end{equation*}
Write
$e_{\mu\mu}\mathcal{T}_d e_{\mu d,\lambda
d'}\mathcal{T}_{d^{-1}}=\sum_{y\in \mathcal{D}_\la}a_y
e_{\mu,\lambda y} $ where $a_y\in F$. Then
$$N_{W,1}(e_{\mu d,\lambda d'})=\sum_{y\in
\mathcal{D}_\lambda}a_yN_{W,1}(e_{\mu,\lambda y}).$$

For $y\in \mathcal{D}_\lambda$, there exists
$x\in\mathcal{D}_{\lambda\mu},k \in \sD_{\la x\cap\mu}\cap W_\mu$ such that $y=xk$. By Corollary
\ref{MC}, if
$x\in\mathcal{D}_{\lambda\mu}\setminus\mathcal{D}^\circ_{\lambda\mu}$,
then $N_{W,1}(e_{\mu,\lambda y})=0$.
If $x\in\mathcal{D}^\circ_{\lambda\mu}$, then we apply Lemma \ref{DD1} with $W_\alpha=1$ to give,  for some $c\in F$,
$$N_{W,1}(e_{\mu,\lambda y})=cN_{W,1}(e_{\mu,\la x})=cd_{W_{\la x\cap\mu}}(q^2)N_{W,W_{\la x\cap \mu}}(e_{\mu,\la x})\in I_F(\{1\},r),$$
since $d_{W_{\la x\cap\mu}}(q^2)=0$ if $P_{\la x\cap\mu}\neq1$.
 The result is proved. \end{proof}

 We now generalise this result to an arbitrary $l$-parabolic subgroup. {\it For the rest of the section,
 fix a non-negative integer $k$ and let }
$$\theta=(\underbrace{l,\cdots,l}_{k},1,\cdots,1).$$ Then $W_\theta=_WP_k$
 is an
$l$-parabolic subgroup of $W$. For
$\lambda\in\Lambda{(m|n,r)},d\in\mathcal{D}_{\mu\theta}$, if $W_\theta\leq W_\la^d$
then, by Lemma \ref{00-11},
 $W_\theta=(W_{\la^{(0)}}^d\cap W_\theta)\times (W_{\la^{(1)}}^d\cap W_\theta)$, and both
$W_{\la^{(0)}}^d\cap W_\theta$ and $W_{\la^{(1)}}^d\cap W_\theta$ are $l$-parabolic.

\begin{lemma}\label{key}
Let $\theta\models r$ be given as above. If $\lambda\in\Lambda{(m|n,r)},
d\in\mathcal{D}_{\lambda\theta}$ and $W_\theta\leq
W^d_\lambda$, then there exist $\theta'\in\Lambda{(m'|n',r)}$ for some $m'\geq m$ and $n'\geq n$ and
 $w\in N_W(W_\theta)\cap
\mathcal{D}_{\theta\theta}$ such that $W_{\theta'}=W_\theta$, $dw^{-1}\in\sD_{\la\theta}^\circ$, and
$$e_{\lambda d, \theta' w}\in \End_{\mathcal{H}_\theta}(V_F(m'|n')^{\otimes
r}).$$
Moreover, $N_{W,W_\theta}(e_{\lambda d,\theta'w})\in I_F(W_\theta,r)_{m'|n'}.$
\end{lemma}
\begin{proof}
Since $d\in \mathcal{D}_{\lambda\theta}$ and $W_\theta\leq
W^d_\lambda$, by \cite[1.3]{DU2}, we have
$$\bsi_\lambda
d=(\underbrace{i_1,\cdots,i_1}_l,\underbrace{i_2,\cdots,i_2}_l,\cdots,\underbrace{i_{k},\cdots,i_{k}}_l,j_1,\cdots,j_{t}).$$
The sequence $( \widehat{i_1},\widehat{i_2},\ldots, \widehat{i_k})$ is called {\it the parity of the $k$ blocks of length $l$ in $\bsi_\la d$}.

Let $a=\#\{j\mid \widehat{i_j}=0, 1\leq j\leq k\}$ and $b=
\#\{j\mid \widehat{i_j}=1, 1\leq j\leq k\}$. We will call $a$ (resp., $b$) the {\it number of even} (resp., {\it odd}) {\it blocks of length $l$}. Then $a+b=k$ and
$$
\bsi_{\la^{(0)}}d=\bsi_\la d|_{[1,m]}\text{ and } \bsi_{\la^{(1)}}d=\bsi_\la d|_{[m+1, m+n]}.$$

Consider 0-1 sequence $( \widehat{i_1},\widehat{i_2},\ldots, \widehat{i_k})$. There is a shortest
$x\in\fS_k$ such that
$$( \widehat{i_1},\widehat{i_2},\ldots, \widehat{i_k})x=(0^a,1^b).$$ Then $g=\ell(x)$ is the number of inversions in the sequence (see proof of \cite[Lem.~1.3]{DU3}). Every such inversion $\widehat{i_c}>\widehat{i_d}$ with $c<d$ determines $l^2$ inversions in the sequence $\underbrace{i_c,\cdots,i_c}_l,\underbrace{i_d,\cdots,i_d}_l$.

Let $\theta'^{(0)}=(l^{a})$ and $\theta'^{(1)}=(l^b,1^{t})$. Then, by adding some 0's at the end of $\theta'^{(i)}$ if necessary, we may assume $\theta'=(\theta'^{(0)}|\theta'^{(1)})\in\La(m'|n',r)$ for some $m'\geq m,n'\geq n$. Clearly,  $W_\theta=W_{\theta'}$. Define $w\in\sD_{\theta\theta}$ such that
$$\bsi_{\theta'}w=(\underbrace{i'_1,\cdots,i'_1}_l,\underbrace{i'_2,\cdots,i'_2}_l,\cdots,\underbrace{i'_{k},\cdots,i'_{k}}_l,j'_1,\cdots,j'_{t})$$
where $\widehat{i'_s}=\widehat{i_s}$ for $1\leq s\leq  k$ and $(j)w=j$ for all $kl<j\leq r$. Since the first $k$ blocks in $\bsi_{\theta'}$ contains $gl^2$ inversion, we have $\ell(w)=gl^2$ and $w\in N_W(W_\theta)\cap\mathcal{D}_{\theta\theta}$. Moreover,
$$\bsi_\lambda
dw^{-1}=((i_1,\cdots,i_1,i_2,\cdots,i_2,\ldots,i_{k},\cdots,i_{k})_{\text{sorted}},j_1,\cdots,j_{t}).$$
Now the other inversions in $\bsi_\la d$ is unchanged during this sorting process. We conclude that
$\ell(d)=\ell(dw^{-1})+\ell(w)$. Hence, $dw^{-1}\in\sD_{\la\theta'}$.

Also, since Stab$_W(\bsi_{\la^{(0)}}d,\bsi_{\theta'^{(0)}}w)=_WP_a$ and Stab$_W(\bsi_{\la^{(1)}}d,\bsi_{\theta'^{(1)}}w)=_WP_b$, it follows that $W^d_{\lambda^{(0)}}\cap W^w_{\theta'^{(1)}}=1 \mbox{ and }
W^d_{\lambda^{(1)}}\cap W^w_{\theta'^{(0)}}=1$. Hence,
$dw^{-1}\in\mathcal{D}_{\lambda\theta'}^\circ$.

By Lemma
\ref{super2.2},  $e_{\lambda d, \theta'w}\in
\End_{\mathcal{H}_\theta}(V_F(m|n)^{\otimes r})$.
So the transitivity of relative norms (Lemma \ref{RNT}(a)) gives
$$N_{W,W_\theta}(e_{\lambda d, \theta'w}) \in \End_{\mathcal{H}_F}(V_F(m'|n')^{\otimes r}).$$

The last assertion is clear since $N_{W,W_\theta}(e_{\lambda d, \theta'w})$ is a linear combination
of basis elements of the form $N_{W,W_\la^x\cap W_{\theta'}}(e_{\la x,\theta'})$ each of which  has a defect group $\leq_WW_{\theta'}$ and hence, is in $I_F(W_\theta,r)_{m'|n'}$.
 \end{proof}

\begin{corollary}\label{MC3}
Let $\theta=(l^k,1^t)$ be as above. For $\lambda,\mu\in \Lambda(m|n,r), d\in
\mathcal{D}_{\la\theta},d'\in\mathcal{D}_{\mu\theta}$, assume
$W_\theta\leq (W^{d}_{\lambda^{(0)}}\cap W^{d'}_{\mu^{(0)}})\times (W^{d}_{\lambda^{(1)}}\cap W^{d'}_{\mu^{(1)}})
 $.
Then there is $\theta'\in
\Lambda(m'|n',r)$ for some $m'\geq m,n'\geq n$ and $w\in\mathcal{D}_{\theta\theta}$ such that
$W_\theta=W_{\theta'}$  
and
$$N_{W,W_\theta}(e_{\mu d,\lambda d'})=N_{W,W_\theta}(e_{\mu
d',\theta'w})N_{W,W_\theta}(e_{\theta'w,\lambda d}).$$
\end{corollary}
\begin{proof} We first prove that $\bsi_\lambda d=(i_1,\cdots,i_r)$ and $\bsi_\mu
d'=(i'_1,\cdots,i'_r)$ have the same parity sequence for the $k$ blocks of length $l$.
Since $W_\theta\leq W^{00}_{\lambda
d\cap\mu d'}\times W^{11}_{\lambda d\cap\mu d'}$, by Lemma
\ref{super2.2}, $\mathcal{T}_s e_{\mu d',\lambda d}= e_{\mu
d',\lambda d}\mathcal{T}_s$ for $s=(j,j+1)\in W_\theta\cap S$. Applying this equality to
$v_{\bsi_\mu d'}$ gives
$(-1)^{\widehat{i'_j}}q^{(-1)^{\widehat{i'_j}}} v_{\bsi_\lambda
d}=(-1)^{\widehat{i_j}}q^{(-1)^{\widehat{i_j}}} v_{\bsi_\lambda
d}$.
Hence, $\widehat{i_j}=\widehat{i'_j}$ for $1\leq j\leq kl$.

By the proof of Lemma \ref{key}, there are
$\theta'\in\Lambda(m'|n',r)$ and a common $w\in
N_W(W_\theta)\cap\mathcal{D}_{\theta\theta}$ such that $e_{\mu
d',\theta'w}, e_{\theta'w,\lambda d}\in
\End_{\mathcal{H}_\theta}(V_F(m|n)^{\otimes r})$.
Then
$$N_{W,W_\theta}(e_{\mu d',\theta'w}), N_{W,W_\theta}( e_{\theta'w,\lambda
d})\in \End_{\mathcal{H}_F}(V_F(m'|n')^{\otimes r}).$$
Also, from the construction of $w$, we see that Stab$_W(\bsi_{\theta'} w)=W_{\theta'}$.
Thus, $v_{\bsi_{\theta'}w}\sT_xe_{\theta'w,\la d}\neq0$ if and only if $x\in W_{\theta'}$.
Hence,
\begin{equation*}
\begin{aligned}
&N_{W,W_\theta}(e_{\mu d',\theta'w})N_{W,W_\theta}(
e_{\theta'w,\lambda
d})\\
&=N_{W,W_\theta}(e_{\mu d',\theta'w}N_{W,W_\theta}(
e_{\theta'w,\lambda
d}))\\ &=N_{W,W_\theta}(e_{\mu d',\theta'w}\sum_{x\in\sD_{\theta\theta}}N_{W_\theta,W_\theta^x\cap W_\theta}(\sT_{x^{-1}}
e_{\theta'w,\lambda
d}\sT_x))\\
&=\sum_{x\in\mathcal{D}_{\theta\theta}\cap W_{\theta'}}N_{W,W^x_\theta\cap
W_\theta}(e_{\mu
d',\theta'w}\mathcal{T}_{x^{-1}}e_{\theta'w,\lambda
d}\mathcal{T}_x)\\
&= N_{W,W_\theta}(e_{\mu d',\lambda d}),\end{aligned}
\end{equation*}
 as
desired. \end{proof}

We now establish the main result of this section and leave its application to \S10 where we will use this result to compute the vertices of indecomposable summands of $V_F(m|n)^{\otimes r}$.

\begin{theorem}\label{IT}
Let $W_\theta=_WP_k$ be the $l$-parabolic subgroup of $W$ associated with composition $(l^k,1^{r-kl})$. Then
$$I_F(W_\theta,r)=N_{W,W_\theta}(\End_{\mathcal{H}_\theta}(V_F(m|n)^{\otimes
r})).$$
\end{theorem}
\begin{proof} We first show that
$$N_{W,W_\theta}(\End_{\mathcal{H}_\theta}(V_F(m|n)^{\otimes
r}))\subseteq I_F(W_\theta,r).$$

By \ref{SRNB}, it is enough to prove that
$$N_{W,W^y_\alpha\cap W_\beta}(e_{\mu d',\lambda dy})\in
I_F(W_\theta,r)$$ for all
$\mu,\lambda\in\Lambda{(m|n,r)},d'\in\mathcal{D}_{\mu\theta},d\in\mathcal{D}_{\lambda\theta},W_\alpha=W^d_\la\cap
W_\theta,W_\beta=W^{d'}_\mu\cap
W_\theta,y\in\mathcal{D}^\circ_{\alpha\beta}\cap W_\theta$.

As in {\cite[4.7]{DU}}, we proceed induction on
$k$.

If $k=0$, i.e $W_{\theta}=1,W^y_\alpha\cap
W_\beta=1$, then the assertion follows from Lemma \ref{P=1}.

Assume that $W_\theta\neq 1$. Let $W_\gamma=W^y_\alpha\cap
W_\beta\leq W_\theta$.

If the maximal parabolic  subgroup
$P_\gamma$ of $W_\gamma$ is conjugate a proper parabolic subgroup of
$W_\theta$, i.e., $P_\gamma<_{W} W_\theta$, then, by induction,
$$N_{W,W_\gamma}(e_{\mu d,\lambda
d'y})=f(q^2)^{-1}N_{W,P_\gamma}(e_{\mu
d,\lambda d'y})\in I_F(P_\gamma,r),$$
where $f(\bsu)=\frac{d_{W_{\gamma^{(0)}}}(\bsu)d_{W_{\gamma^{(1)}}}(\bsu^{-1})}{d_{P_{\gamma^{(0)}}}(\bsu)d_{P_{\gamma^{(1)}}}(\bsu^{-1})}$.
  Hence,
$N_{W,W_\gamma}(e_{\mu d,\lambda d'y})\in I_F(W_\theta,r)$.

Thus, it remains to look at the case where $W_\gamma=W_\theta$. That is,
$$W_\theta=(W^{d}_\lambda\cap W_\theta)^y\cap(W^{d'}_\mu\cap
W_\theta)=W^{dy}_\lambda\cap W^{d'}_\mu\cap W_\theta.$$
This forces that $W_\alpha=W_\beta=W_\theta$ and $y=1$ as
$y\in\mathcal{D}_{\alpha\beta}^\circ\cap W_\theta$. In particular,
$W_\theta\leq W_{\la d\cap\mu d'}^{00}\times W_{\la d\cap\mu d'}^{11}$.
Hence, by Lemma \ref{key} and Corollaries \ref{MC3} and \ref{MC2}, there exist $m'\geq m,n'\geq n$, $\theta'\in\La(m'|n',r)$, and $w\in\sD_{\theta\theta}\cap N_W(W_\theta)$ such that
$$N_{W,W_\theta}(e_{\mu d',\lambda d})=N_{W,W_\theta}(e_{\mu
d',\theta'w})N_{W,W_{\theta}}(e_{\theta'w,\lambda d})\in
\vep I_F(W_\theta,r)_{m'|n'}\vep,$$
where $\vep$ is defined in Remark \ref{m'|n' case}.
Hence, by identifying $I_F(W_\theta,r)$ with $\vep I_F(W_\theta,r)_{m'|n'}\vep $,
we proved that all $N_{W,W^y_\beta\cap W_\alpha}(e_{\mu d,\lambda d'y})\in I_F(W_\theta,r)$.






We next prove that
$$I_F(W_\theta,r)\subseteq N_{W,W_\theta}(\End_{\mathcal{H}_\theta}(V_F(m|n)^{\otimes
r})).$$
Equivalently, we want to prove every basis element $N_{W,W_\nu}(e_{\mu,\lambda d})\in I_F(W_\theta,r)\cap \mathcal{B}$,
where $W_{\nu}=W^d_\lambda\cap
W_\mu,d\in\mathcal{D}^\circ_{\lambda\mu}$ satisfying $P_\nu\leq_W
W_\theta$, is in the R.H.S.

Let $z$ be a distinguished representative of the $P_\nu-W_\theta$
double coset such that $P^z_\nu\leq W_\theta$ and $P^z_\nu$
is also parabolic (of course $l$-parabolic). $z$ is also a
distinguished representative of the $P_\nu-P^z_\nu$ double coset.
Let $\tau,\tau'$ be the decompositions of $r$ such that
$W_\tau=P_\nu,W_{\tau'}=P^z_\nu$.

We consider $\mathcal{H}_F$-$\mathcal{H}_F$-bimodule
$$M=\Hom_F(v_\mu\mathcal{H}_F,v_{\bsi_\lambda
d}\otimes_{\mathcal{H}_\tau}\mathcal{H}_F)$$
which has an $\mathcal{H}_F$-$\mathcal{H}_\tau$-bisubmodule
$N=\Hom_F(v_\mu\mathcal{H}_F,Rv_{\bsi_\lambda d})$ and  an
$\mathcal{H}_F$-$\mathcal{H}_{\tau'}$-bisubmodule
$N'=\Hom_F(v_\mu\mathcal{H}_F,Rv_{\bsi_\lambda
d}\otimes_{\mathcal{H}_\tau}\mathcal{T}_z)$.
Since
\begin{equation*}
\begin{aligned}
M&\cong\bigoplus_{w\in\mathcal{D}_\tau}\Hom_F(v_\mu\mathcal{H}_F,v_{\bsi_\lambda
d}\otimes \mathcal{T}_w)\\
&\cong\bigoplus_{w\in\mathcal{D}_\tau}\Hom_F(v_\mu\mathcal{H}_F,v_{\bsi_\lambda
d})\otimes \mathcal{T}_w\\
&\cong N\otimes_{\mathcal{H}_\tau}\mathcal{H}_F
\end{aligned}
\end{equation*}
as $\mathcal{H}_F$-$\mathcal{H}_F$-bimodules and, similarly, $M\cong
N'\otimes_{\mathcal{H}_{\tau'}}\mathcal{H}_F$, it follows from Lemma
\ref{RNT} that
$$N_{W,W_\tau}(\Hom_{\mathcal{H}_\tau}(v_\mu\mathcal{H}_F,v_{\bsi_\lambda
d}))=N_{W,W_{\tau'}}(\Hom_{\mathcal{H}_{\tau'}}(v_\mu\mathcal{H}_F,v_{\bsi_\lambda
d}\otimes_{\mathcal{H}_\tau}\mathcal{T}_z)).$$ Thus, there exists
$h\in \Hom_{\mathcal{H}_{\tau'}}(v_\mu\mathcal{H}_F,v_{\bsi_\lambda
d}\otimes_{\mathcal{H}_\tau}\mathcal{T}_z)$ such that
$$N_{W,P_\nu}(e_{\mu,\lambda d})=N_{W,P^z_\nu}(h).$$
Hence, putting $g(\bsu)=\frac{d_{W_{\nu^{(0)}}}(\bsu)d_{W_{\nu^{(1)}}}(\bsu^{-1})}{d_{P_{\nu^{(0)}}}(\bsu)d_{P_{\nu^{(1)}}}(\bsu^{-1})}$,
\begin{equation*}
\begin{aligned}
N_{W,W_\nu}(e_{\mu,\lambda d})&=g(q^2)^{-1}N_{W,P_\nu}(e_{\mu,\lambda d})\\
&=N_{W,P^z_\nu}(g(q^2)^{-1}h)\\
&=N_{W,W_\theta}(N_{W_\theta,P^z_\nu}(g(q^2)^{-1}h))\in \text{RHS}
\end{aligned}
\end{equation*}
proving the theorem. \end{proof}

\section{Quantum matrix superalgebras}

We follow \cite{M} to introduce the quantum matrix superalgebras.

\begin{definition}\label{DefSA}
Let $\mathcal{A}_q(m|n)$ be the associated superalgebra over $F$
generated by $x_{ij},1\leq i,j\leq m+n$ subject to the following
relations
\begin{enumerate}
\item $x^2_{ij}=0$ for $\hat{i}+\hat{j}=1$;

\item $x_{ij}x_{ik}=(-1)^{(\hat{i}+\hat{j})(\hat{i}+\hat{k})}q^{(-1)^{\hat{i}+1}}x_{ik}x_{ij}$
for $j<k$;

\item $x_{ij}x_{kj}=(-1)^{(\hat{i}+\hat{j})(\hat{k}+\hat{j})}q^{(-1)^{\hat{j}+1}}x_{kj}x_{ij}$
for $i<k$;

\item $x_{ij}x_{kl}=(-1)^{(\hat{i}+\hat{j})(\hat{k}+\hat{l})}x_{kl}x_{ij}$
for $i<k$ and $j>l$;

\item $x_{ij}x_{kl}=(-1)^{(\hat{i}+\hat{j})(\hat{k}+\hat{l})}x_{kl}x_{ij}+(-1)^{\hat{k}\hat{j}+\hat{k}\hat{l}+\hat{j}\hat{l}}(q^{-1}-q)x_{il}x_{kj}$
for $i<k$ and $j<l$.
\end{enumerate}
\end{definition}

Note that if all indices $i,j,k,l$ are taken from $[1,m]$ or $[m+1,m+n]$ then the relations coincide with
those for the quantum matrix algebra; see, e.g., \cite[(3.5a)]{PW}.

Manin proved that $\sA_q(m|n)$ has also the (usual) coalgebra structure with
comultiplication $\Delta:\sA_q(m|n)\rightarrow \sA_q(m|n)\otimes
\sA_q(m|n)$ and counit $\varepsilon:\sA_q(m|n)\rightarrow F$ defined by
$$\Delta(x_{ij})=\sum_{k=1}^{m+n}x_{ik}\otimes x_{kj} \mbox{ and }
\varepsilon(x_{ij})=\delta_{ij},\forall 1\leq i,j\leq m+n.$$
Further, the $\mathbb{Z}_2$ grading degree of $x_{ij}$ is
$\hat{i}+\hat{j}\in\mathbb{Z}_2$. Hence, $\sA_q(m|n)$ is a super bialgebra.

We now describe a basis for $\sA_q(m|n)$.
For $\bsi=(i_1,i_2,\ldots,i_r),\bsj=(j_1,j_2,\ldots,j_r)\in I(m|n,r)$, let
$$x_{\bsi,\bsj}=x_{i_1j_1}x_{i_2j_2}\cdots x_{i_rj_r}.$$
In particular, for
$\lambda,\mu\in\Lambda{(m|n,r)},d\in\mathcal{D}^\circ_{\lambda\mu}$,
define $x_{\mu,\lambda d}:=x_{\bsi_\mu,\bsi_\lambda d}.$

The elements $x_{\mu,\la d}$ are described in \cite{DR} by matrices which follows Manin \cite{M}.
Let $M_{m+n}(\mathbb N)$ be the $(m+n)\times(m+n)$ matrix semigroup over $\mathbb N$ and let
$$\aligned
M(m|n)&=\{(a_{ij})\in M_{m+n}({\mathbb N})\mid a_{ij}=0,1 \text{ if }\hat i+\hat j=1\}\\
M(m|n,r)&=\{A\in M(m|n)\mid r=|A|\},
\endaligned$$
where $|A|$ is the sum of the entries of $A$. Then, by \cite[(3.2.1)]{DR}, every triple $\la,\mu\in\La(m|n,r)$, and $d\in\sD_{\la\mu}^\circ$ defines a unique matrix $\jmath(\la,d,\mu)^t$, the transpose of $\jmath(\la,d,\mu)$, whose concatenation of row 1, row 2, and so on is the composition $\la d\cap\mu$ and whose row sum (resp., column) vector is $\mu$ (resp., $\la$). We will write $x^A=x_{\mu,\la d}$ if $A=\jmath(\la,d,\mu)^t$.

.
\begin{lemma}\label{DR9.7} (1) ({\cite{M},\cite[9.3]{DR}})
The set $\{x^A\}_{A\in M(m|n)}$ forms a basis for $\sA_q(m|n)$.

(2) (\cite[9.7]{DR}) Let $\sA_q(m|n,r)$ be the subspace spanned by
$$\mathcal{B}^\vee=\{x_{\mu,\la d}\mid
\lambda,\mu\in\Lambda{(m|n,r)},d\in\mathcal{D}^\circ_{\lambda\mu}\}=
\{x^A\}_{A\in M(m|n,r)}.$$
Then $\sA_q(m|n,r)$ is a subcoalgebras with basis $\mathcal{B}^\vee$ and
the quantum Schur superalgebra $\sS_F(m|n,r)$ is isomorphic to the dual
algebra $\mathcal{A}_q(m|n,r)^*$.
\end{lemma}

Recall from \cite[9.8]{DR} that the isomorphism $\sS_F(m|n,r)\cong \sA_q(m|n,r)^*$ is obtained
from an isomorphism
$$\mathcal{A}_q(m|n,r)^*\cong
\End_{\mathcal{H}_F}(V_F(m|n)^{\otimes r})$$
which is induced by an $\mathcal{A}_q(m|n,r)$-comodule structure on $V_F(m|n)^{\otimes
r}$. This structure is the restriction of the comodule structure $\delta:V_F(m|n)^{\otimes r}\to \sA_q(m|n)\otimes V_F(m|n)^{\otimes r}$ defined by,
for any $\bsi\in I(m|n,r)$,
\begin{equation}\label{comodule map}
\delta(v_\bsi)=\sum_{\bsj\in I(m|n,r)}(-1)^{\sum_{1\leq k<l\leq
r}\widehat{j_k}(\widehat{j_l}+\widehat{i_l})}x_{\bsi,\bsj}\otimes
v_\bsj.\end{equation}
Now, the (left)
$\mathcal{A}_q(m|n,r)$-comodule $V_F(m|n)^{\otimes r}$ turns into  a right
$\mathcal{A}_q(m|n,r)^*$-module with the action given by
\begin{equation}\label{module map}
v\cdot f=(f\otimes\text{id}_{V_F(m|n)^{\otimes r}})\delta(v),\forall
f\in \mathcal{A}_q(m|n,r)^*,v\in V_F(m|n)^{\otimes r}.
\end{equation}
This action commutes with the Hecke algebra action as shown in \cite[9.7]{DR} and results in the isomorphism above.

We now make a comparison between the basis $\{\psi^d_{\mu\lambda}\mid
\lambda,\mu\in\Lambda{(m|n,r)},d\in\mathcal{D}^\circ_{\lambda\mu}\}$ (or the relative norm basis $\mathcal{B}$) for $\sS_F(m|n,r)$ and the dual basis
$$\{x^*_{\mu,\lambda d}\mid
\lambda,\mu\in\Lambda{(m|n,r)},d\in\mathcal{D}^\circ_{\lambda\mu}\}.$$ First we observe
$v_{\bsi_\nu}\cdot x^*_{\mu,\la d}=0$ if $\mu\neq \nu$.

\begin{theorem}\label{N-x basis}
For
$\lambda,\mu\in\Lambda{(m|n,r)},d\in\mathcal{D}^\circ_{\lambda\mu}$, and $\nu=\la d\cap\mu$,
we have
$$N_{W,W_\nu}(e_{\mu,\lambda d})=(-1)^{\sum_{1\leq k<l\leq r}\widehat{(\bsi_\lambda
d)_k}(\widehat{(\bsi_\lambda d)_l}+\widehat{(\bsi_\mu)_l})}x^*_{\mu,\lambda d}.$$
\end{theorem}
\begin{proof} For any
$\lambda,\mu\in\Lambda{(m|n,r)},d\in\mathcal{D}^\circ_{\lambda\mu}$.
Since $N_{W,W_\nu}(e_{\mu,\lambda d})$ and $x^*_{\mu, \la d}$ are in
$\End_{\mathcal{H}_F}(V_F(m|n)^{\otimes r})$, it is enough to
consider their actions on $v_\mu$.

By the proof of Theorem  \ref{RTSB}, we know
$$(v_\mu)N_{W,W_\nu}(e_{\mu,\la d})=(-1)^{\widehat{d}}\sum_{w_0w_1\in(W_{\mu^{(0)}}\times W_{\mu^{(1)}})\cap\mathcal{D}_{\nu}}
(-1)^{\widehat{dw_0w_1}}q^{\ell(w_0)}(-q^{-1})^{\ell(w_1)}v_{\bsi_\lambda
dw_0w_1} $$ where $W_{\nu}=W^d_\lambda\cap W_\mu$.

On the other hand,
\begin{equation*}
\begin{aligned}
v_\mu\cdot x^*_{\mu,\la d}&=( x^*_{\mu,\la d}\otimes
\text{id}_{V_F(m|n)^{\otimes r}})\delta(v_\mu)\\
&=\sum_{\bsj\in I(m|n,r)}(-1)^{\sum_{1\leq k<l\leq
r}\widehat{\bsj_k}(\widehat{\bsj_l}+\widehat{(\bsi_\mu)_l})} x^*_{\mu,\la d}(x_{\bsi_\mu,
\bsj})v_\bsj.
\end{aligned}
\end{equation*}

If $x^*_{\mu,\la d}(x_{\bsi_\mu, \bsj})\neq 0$, then there
exists $w^{-1}\in W_\mu$ such that $(\bsi_\mu,\bsi_\lambda
d)=(\bsi_\mu,\bsj w^{-1})$. So $\bsj=\bsi_\la dw\in (\bsi_\la d)\cdot W_\mu$, the orbit of $\bsi_\la d$.
Then
\begin{equation*}
\begin{aligned}
v_\mu\cdot x^*_{\mu,\la d}&=\sum_{\bsj\in(\bsi_\la d)\cdot W_\mu}(-1)^{\sum_{1\leq k<l\leq
r}\widehat{\bsj_k}(\widehat{\bsj_l}+\widehat{(\bsi_\mu)_l})}x^*_{\mu,\la d}(x_{\bsi_\mu, \bsj})v_\bsj\\
&=\sum_{w\in W_\mu\cap \mathcal{D}_{\nu},\bsj=\bsi_\lambda
dw}(-1)^{\sum_{1\leq k<l\leq
r}\widehat{\bsj_k}(\widehat{\bsj_l}+\widehat{(\bsi_\mu)_l})}x^*_{\mu,\la d}(x_{\bsi_\mu,
\bsj})v_\bsj.
\end{aligned}
\end{equation*}

For $w\in \mathcal{D}_{\nu}\cap W_\mu$, there are $w_0\in
W_{\mu^{(0)}},w_1\in W_{\mu^{(1)}}$ such that $w=w_0w_1$.
Let $\bsi=\bsi_\la d=(i_1,\ldots,i_r)$ and $\bsj=\bsi_\la dw=(j_1,\dots,j_r)$. Then $j_k=i_{w(k)}$.
Let
$$\aligned
\widetilde\sJ&=\{(k,l)\mid 1\leq k<l\leq r, j_k>j_l\},\quad\widetilde \sJ^*=\{(k,l)\mid 1\leq k<l\leq r, j_k<j_l\},\text{ and }
\\
\sJ&=\sJ_{[1,m+n]}=\{(k,l)\in\widetilde \sJ\mid a_{i-1}+1\leq k<l\leq a_{i-1}+\mu_i, i\in[1,m+n]\},\endaligned$$
where $a_i=\sum_{j=1}^{i-1}\mu_j$ for $1\leq i\leq m+n$ and $a_0=0$.

Define $\widetilde \sI$, $\widetilde \sI^*$, $\sI$ etc. similarly with respect to $\bsi$.
We also define $\sJ^{(0)}=\sJ_{[1,m]}$ and $\sJ^{(1)}=\sJ_{[m+1,m+n]} $ similarly (so that $\sJ=\sJ^{(0)}\cup \sJ^{(1)}$). Then $\ell(w)=|\sJ|$,  $\ell(w_i)=|\sJ^{(i)}|$ ($i=0,1$) and
$$\sJ=\sJ^{(0)}\cup
\sJ_{00}^{(1)}\cup\sJ_{10}^{(1)}\cup\sJ_{11}^{(1)},$$
where $\sJ_{ij}^{(1)}=\{(k,l)\in\sJ^{(1)}\mid \widehat j_k=i,\widehat j_l=j\}$ for all $ij\in\{00,10,11\}$.
Note that $\sI=\emptyset$,
$$\{(j_k,j_l)\mid(k,l)\in\widetilde\sJ\setminus\sJ\}=\{(i_k,i_l)\mid(k,l)\in\widetilde\sI\},$$
and
$$A:=\{(j_k,j_l)\mid(k,l)\in\widetilde\sJ^*\cup\sJ\}=\{(i_k,i_l)\mid(k,l)\in\widetilde\sI^*\}=:B.$$
Hence, $(-1)^{\sum_{(k,l)\in\widetilde\sJ\setminus\sJ}\widehat{j_k}(\widehat{j_l}+\widehat{(\bsi_\mu)_l})}=
(-1)^{\sum_{(k,l)\in\widetilde\sI}\widehat{i_k}(\widehat{i_l}+\widehat{(\bsi_\mu)_l})}$.
On the other hand, if $(j_k,j_l)\in A$ with $(k,l)\in\sJ$, then $(j_l,j_k)\in B$. However, for $(k,l)\in\sJ^{(0)}$,  $\widehat{(\bsi_\mu)_l}=0$, while for $(k,l)\in\sJ^{(1)}$,  $\widehat{(\bsi_\mu)_l}=1$. Hence,
$(-1)^{\widehat{j_k}(\widehat{j_l}+\widehat{(\bsi_\mu)_l)}}=(-1)^{\widehat{j_l}(\widehat{j_k}+\widehat{(\bsi_\mu)_l)}}$ for $(k,l)\in\sJ^{(0)}\cup\sJ^{(1)}_{00}\cup\sJ^{(1)}_{11}$, while $(-1)^{\widehat{j_k}(\widehat{j_l}+\widehat{(\bsi_\mu)_l)}}=(-1)(-1)^{\widehat{j_l}(\widehat{j_k}+\widehat{(\bsi_\mu)_l)}}$ for $(k,l)\in\sJ^{(1)}_{10}$.
Hence,
$$\aligned
(-1)^{\sum_{1\leq k<l\leq
r}\widehat{j_k}(\widehat{j_l}+\widehat{(\bsi_\mu)_l})}&=(-1)^{\sum_{(k,l)\in\widetilde\sJ}\widehat{j_k}(\widehat{j_l}+\widehat{(\bsi_\mu)_l})}\times(-1)^{\sum_{(k,l)\in\widetilde\sJ^*}\widehat{j_k}(\widehat{j_l}+\widehat{(\bsi_\mu)_l})}\\
&=(-1)^{\sum_{(k,l)\in\widetilde\sJ\setminus\sJ}\widehat{j_k}(\widehat{j_l}+\widehat{(\bsi_\mu)_l})}\times(-1)^{\sum_{(k,l)\in\widetilde\sJ^*\cup\sJ}\widehat{j_k}(\widehat{j_l}+\widehat{(\bsi_\mu)_l})}\\
&=(-1)^{\sum_{(k,l)\in\widetilde\sI}\widehat{i_k}(\widehat{i_l}+\widehat{(\bsi_\mu)_l})}\times(-1)^{\sum_{(k,l)\in\widetilde\sJ^*\cup\sJ}\widehat{i_k}(\widehat{i_l}+\widehat{(\bsi_\mu)_l})+|\sJ_{10}^{(1)}|}\\
&=(-1)^{\sum_{1\leq
k<l\leq r}\widehat{i_k}(\widehat{i_l}+\widehat{(\bsi_\mu)_l})}\times (-1)^{|\sJ_{10}^{(1)}|}.
\endaligned$$

Let  $\bsj'=\bsi_\la dw_0$ and define $\sJ'$ as $\sJ$ above with $\bsj$ replaced by $\bsj'$. Then, by
Definition \ref{DefSA} and (\cite[1.3]{DU2}),
$$x_{\bsi_\mu,\bsi_\la d w_0}=(-1)^{\sum_{(k,l)\in\sJ'}\widehat{j'_k}\widehat{j'_l}}q^{\ell(w_0)}x_{\mu,\la d}=(-1)^{\widehat{dw_0}-\widehat d}q^{\ell(w_0)}x_{\mu,\la d},$$
and
$$x_{\bsi_\mu,\bsi_\la d w_0w_1}=(-1)^{\star}q^{-\ell(w_1)}x_{\mu,\la dw_0}=(-1)^\ell(-1)^{\widehat{dw_0w_1}-\widehat {dw_0}}q^{-\ell(w_1)}x_{\mu,\la dw_0},$$
where
$$\aligned
\star&=\sum_{(k,l)\in\sJ^{(1)}}(1+\widehat{j_k})(1+\widehat{j_l})\\
&=\sum_{(k,l)\in\sJ^{(1)}}(1+\widehat{j_k}+\widehat{j_l})+\sum_{(k,l)\in\sJ^{(1)}}\widehat{j_k}\widehat{j_l}\\
&=\ell+(\widehat{dw_0w_1}-\widehat {dw_0}),\endaligned$$
where $\ell=\sum_{(k,l)\in\sJ^{(1)}}(1+\widehat{j_k}+\widehat{j_l})$ and so
$(-1)^\ell=(-1)^{|\sJ_{00}^{(1)}|+|\sJ_{11}^{(1)}|}$.
Hence,
$$(-1)^\ell\times (-1)^{|\sJ_{10}^{(1)}|}=(-1)^{\ell(w_1)}$$ and substituting gives
$$x_{\bsi_\mu, \bsi_\lambda
dw_0w_1}=(-1)^{|\sJ_{10}^{(1)}|}(-1)^{\widehat{dw_0w_1}-\widehat{d}}q^{\ell(w_0)}(-q^{-1})^{\ell(w_1)}x_{\mu,\lambda
d}.$$
Therefore,
\begin{equation*}
\begin{aligned}
v_\mu\cdot x^*_{\mu,\la d}=&(-1)^{\sum_{1\leq k<l\leq
r}\widehat{i_k}(\widehat{i_l}+\widehat{(\bsi_\mu)_l})}(-1)^{\widehat{d}}\\
&\sum_{w_0w_1\in(W_{\mu^{(0)}}\times
W_{\mu^{(1)}})\cap\mathcal{D}_{\nu(d)}}(-1)^{\widehat{dw_0w_1}}q^{\ell(w_0)}(-q^{-1})^{\ell(w_1)}v_{\bsi_\lambda
dw_0w_1}\\
&=(-1)^{\sum_{1\leq k<l\leq
r}\widehat{(\bsi_\lambda d)_k}(\widehat{(\bsi_\lambda
d)_l}+\widehat{(\bsi_\mu)_l})}(v_\mu)N_{W,W_\nu}(e_{\mu,\la d})\end{aligned}
\end{equation*}
proving the theorem. \end{proof}

 By Theorem \ref{RTSB}, we have the following.

\begin{corollary}
For
$\lambda,\mu\in\Lambda{(m|n,r)},d\in\mathcal{D}^\circ_{\mu\la}$,
we have
$$\psi_{\mu,\lambda}^ d=q^{\ell(d)}(-1)^{\widehat d}(-1)^{\sum_{1\leq k<l\leq
r}\widehat{(\bsi_\lambda d)_k}(\widehat{(\bsi_\lambda
d)_l}+\widehat{(\bsi_\mu)_l})} x^*_{\mu,\la d}.$$
\end{corollary}

\section{Frobenius morphisms and Brauer homomorphisms}

Let $\sA_q(m)$ (resp., $\sA_q(n)$) be the subalgebras of $\sA_q(m|n)$ generated by $x_{ij}$ for
$i,j\in[1,m]$ (resp., $i,j\in[m+1,m+n]$).  Then $\sA_q(m)\otimes\sA_q(n)$ is a subalgebra, but not a subcoalgebra. We now show that $\sA_q(m)\otimes\sA_q(n)$ contains a central subbialgebra of
$\sA_q(m|n)$.

\begin{proposition}\label{SAC}
For $1\leq i,j\leq m$ or
$m+1\leq i,j\leq m+n$, the elements $x^l_{ij}$ are in the center of $\sA_q(m|n)$.
\end{proposition}

\begin{proof} Since $\hat i=\hat j$, the (first) signs on the right hand side of the relations in Definition \ref{DefSA} are all $+1$. Note also that $q^l=(q^{-1})^l=1,[\![l]\!]_q=0$.

For $j<k$, by \ref{DefSA}(2)
$$x^l_{ij}x_{ik}=q^{l(-1)^{\hat{i}+1}}x_{ik}x^l_{ij}=x_{ik}x^l_{ij}.$$

For $i<k$, by \ref{DefSA}(3)
$$x^l_{ij}x_{kj}=q^{l(-1)^{\hat{j}+1}}x_{kj}x^l_{ij}=x_{kj}x^l_{ij}.$$

For $i<k$ and $j>l$, by \ref{DefSA}(4)
$$x^l_{ij}x_{kl}=x_{kl}x^l_{ij}.$$

Finally, for $i<k$ and $j<l$, we claim
$$[x^s_{ij},x_{kl}]=(-1)^{\hat{k}\hat{j}+\hat{k}\hat{l}+\hat{j}\hat{l}}(q^{-1}-q)[\![s]\!]_{q^{2(-1)^{\hat{i}+1}}}x^{s-1}_{ij} x_{il}x_{kj}.$$
Indeed, this is clear, by \ref{DefSA}(5), for $s=1$. In general, we apply induction to
$[x^{s+1}_{ij},x_{kl}]=x_{ij}[x_{ij}^s,x_{kl}]+[x_{ij},x_{kl}]x_{ij}^s$ to prove the claim.
Now taking $s=l$ and noting ${[\![l]\!]}_{q^{\pm2}}=0$ give
$x^l_{ij}x_{kl}=x_{kl}x^l_{ij}$ in this case.\end{proof}

\begin{proposition}\label{CCA}
For $1\leq i,j\leq m$ or $m+1\leq i,j\leq m+n$, we have
\begin{equation*}
\begin{aligned}
&\Delta(x_{ij}^l)=\sum_{k=1}^mx^l_{ik}\otimes x^l_{kj}\mbox{ for
}1\leq i,j\leq m;\\
&\Delta(x_{ij}^l)=\sum_{k=m+1}^{m+n}x^l_{ik}\otimes x^l_{kj}\mbox{
for }m+1\leq i,j\leq m+n.
\end{aligned}
\end{equation*}
\end{proposition}

\begin{proof}
Since $\mathcal{A}_q(m|n)$ is a bialgebra,
$$\Delta(x^l_{ij})=(\Delta(x_{ij}))^l=(\sum^{m+n}_{k=1}x_{ik}\otimes
x_{kj})^l=(A+B)^l,$$
where $A=\sum^{m}_{k=1}x_{ik}\otimes
x_{kj}$ and $B=\sum^{m+n}_{k=m+1}x_{ik}\otimes x_{kj}$.
Putting $u=q^{-2(-1)^{\hat{i}+1}}$ (and so $u^l=1$), we have by the relations in Definition
\ref{DefSA}, $AB=uBA$.
Thus, by the quantum binomial theorem \cite[(7.1.a)]{PW},
$$\Delta(x^l_{ij})=\sum^l_{h=0}\left[\!\!\left[l\atop
h\right]\!\!\right]_uA^hB^{l-h}.$$ Since
$\left[\!\!\left[l\atop h\right]\!\!\right]_u=0$ for $0<h<l$, it turns out $\Delta(x^l_{ij})=A^l+B^l$.

Now, by the quantum multinomial theorem (see, e.g., \cite[Ex.~0.14]{DDPW}), \ref{DefSA}(1) implies  $B^l=0$ for  $1\leq i,j\leq m$, and
$A^l=0$ for  $m+1\leq i,j\leq m+n$. Hence, we have reduced the proof to the non-super case and the assertion follows from
\cite[7.2.2]{PW}.\end{proof}

\begin{corollary} The subalgebra generated by $x_{ij}^l$ with $i,j\in[1,m]$ or $i,j\in[m+1,m+n]$ is also a subcoalgebra in the center of $\sA_q(m|n)$, and hence, is a central subbialgebra.
\end{corollary}

The existence of the subbialgebra is the key to the definition of the Frobenius morphism.
Let $A(m,n)$ be the polynomial algebra over $F$
generated by
$t_{ij}$ for $1\leq i,j\leq m$ or $m+1\leq i,j\leq m+n$. If we also define polynomial algebras $A(m)=F[t_{i,j}]_{1\leq i,j\leq m}$ and $A(n)=F[t_{i,j}]_{m+1\leq i,j\leq m+n}$, then $A(m,n)\cong A(m)\otimes_FA(n)$. We will identify the two polynomial algebras in the sequel.
Then, by the lemma above, the map
\begin{equation*}
\begin{aligned}
\mathcal{F}:A(m,n)&\longrightarrow
\mathcal{A}_q(m|n)\\
t_{ij}&\longmapsto x^l_{ij}
\end{aligned}
\end{equation*}
is a bialgebra monomorphism. This is called the {\it Frobenius morphism}.

For the fixed $r\in\mathbb{N}$, set
\begin{equation}\label{R_r}
\mathcal{R}_r=\sR_r(l)=\{(r_{-1},r_0)\mid r_{-1},r_0\in\mathbb N,
r_{-1}+lr_0=r\}.
\end{equation}

For $\vecr=(r_{-1},r_0)\in\mathcal{R}_r$, we
put
$$\mathcal{A}_q(m|n,\vecr)=\mathcal{A}_q(m|n,r_{-1})\otimes
A(m,n)_{r_0}$$ where
$A(m,n)_{r_0}=\sum_{s+t=r_0}A(m)_s\otimes A(n)_t$ is the $r_0$th
homogenous component of $A(m,n)$. It is
easy to show that $\mathcal{A}_q(m|n,\vecr)$ is a
subcoaglebra of $\mathcal{A}_q(m|n)\otimes A(m,n)$.
Let
\begin{equation}\label{Schur vecr}
\sS_F(m|n,\vecr):=\sA_q(m|n,\vecr)^*\cong\sS_F(m|n,r_{-1})\otimes S(m,n)_{r_0},
\end{equation}
where
\begin{equation}\label{S(m,n)}
S(m,n)_{r_0}=A(m,n)_{r_0}^*=\bigoplus_{i=0}^{r_0}(S(m,i)\otimes
S(n,r_0-i)),
\end{equation}
 which is a sum of tensor products of certain classical Schur algebras.

Consider the map
$$\mu\circ(1\otimes \mathcal{F}):\mathcal{A}_q(m|n)\otimes A(m,n)\rightarrow \mathcal{A}_q(m|n)$$ where $\mu$ is the
multiplication in $\mathcal{A}_q(m|n)$, then $\mu\circ(1\otimes
\mathcal{F})$ is a coalgebra map and
$$\mu\circ(1\otimes
\mathcal{F})(\mathcal{A}_q(m|n,\vecr))\subseteq
\mathcal{A}_q(m|n,r).$$ Thus, upon  restriction,  we obtain a coalgebra map
$$F_{\vecr}:\mathcal{A}_q(m|n,\vecr)\rightarrow
\mathcal{A}_q(m|n,r).$$

\begin{definition} If we identify $\sS_F(m|n,r)$ with $\sA_q(m|n,r)^*$ under Lemma \ref{DR9.7}(2), then
the {\it Brauer homomorphism} associated with $\vecr\in\mathcal{R}_r$ is the {\it surjective} map
$$\phi_\vecr:\sS_F(m|n,r)\rightarrow
\sS_F(m|n,\vecr):=(\mathcal{A}_q(m|n,\vecr))^*$$
 dual to the coalgebra map
$F_{\vecr}:\mathcal{A}_q(m|n,\vecr)\rightarrow
\mathcal{A}_q(m|n,r)$.
\end{definition}

We now determine the kernel of
$\phi_\vecr$. First, we determine the image of
$F_{\vecr}$.

For $\vecr=(r_{-1},r_0)\in\mathcal{R}_r$, define the
$l$-parabolic subgroup $P_{\vecr}$ of $W$ to be the parabolic
subgroup associated with $(1^{r_{-1}},l^{r_0})\in\Lambda{(m|n,r)}$. Thus, with the notation used in \eqref{filtration}, $P_\vecr=P_{r_0}$.

For $\bsi=(i_1,i_2,\cdots,i_r)\in I(m|n,r)$, let
$$l\bsi=(\underbrace{i_1,\cdots,i_1}_l,\underbrace{i_2,\cdots,i_2}_l,\cdots,\underbrace{i_r,\cdots,i_r}_l)\in
I(m|n,rl).$$

\begin{lemma}\label{image}
The image of $F_{\vecr}$ is the subspace spanned by
$$\{x_{\mu,\lambda d}\mid
\lambda,\mu\in\Lambda{(m|n,r)},d\in\mathcal{D}^\circ_{\lambda\mu},P_{\vecr}\leq_W
W^d_\lambda\cap W_\mu\}.$$
\end{lemma}
\begin{proof}
We first observe that the set
$$\{x_{\bsi_{-1},\bsj_{-1}}\otimes t_{\bsi_0,\bsj_0}\mid
\bsi_{-1},\bsj_{-1}\in I(m|n,r_{-1}),\bsi_0,\bsj_0\in
I(m,n;r_0)\}$$
spans $\mathcal{A}_q(m|n,\vecr)$,
where
$$I(m,n;r_0)=\bigcup_{s+t=r_0}I(m,s)\times(\mathbf{m}+ I(n,t))\text{ with }\mathbf{m}=(\underbrace{m,\ldots,m}_n).$$
To get a basis, by Lemma \ref{DR9.7}(2), we simply take the $\bsi,\bsj$ involved to satisfy the conditions that $\bsi$ weakly increasing and $j_a\leq j_b$ whenever $i_a=i_b$. So
we may assume that
$x_{\bsi_{-1},\bsj_{-1}}=x_{\mu_{(-1)},\lambda_{(-1)}d_{-1}},t_{\bsi_0,\bsj_0}=t_{\mu_{(0)},\lambda_{(0)}d_0}$
for $\lambda_{(i)},\mu_{(i)}$ into $m+n$ parts and
$d_i\in\mathcal{D}^\circ_{\lambda_{(i)},\mu_{(i)}},i=-1,0$.

Now,
$F_{\vecr}(x_{\bsi_{-1},\bsj_{-1}}\otimes
t_{\bsi_0,\bsj_0})=x_{\bsi_{-1},\bsj_{-1}}x_{l\bsi_0,l\bsj_0}$. By Proposition
\ref{CCA}, every factor $x_{i_kj_k}^l$ of $x_{l\bsi_0,l\bsj_0}$, where $i_k,j_k\in[1,m]$ or $[m+1,m+n]$, is in the
center of $\mathcal{A}_q(m|n)$. Thus, we may move them around so that the product
$x_{\bsi_{-1},\bsj_{-1}}x_{l\bsi_0,l\bsj_0}$ can be written as $x_{\bsi,\bsj}$ with $\bsi$ weakly increasing and $\bsj$ satisfying $j_a\leq j_b$ whenever $i_a=i_b$ and, if $i_a=i_{a+1}$ and $j_a=j_{a+1}$, then $i_a,j_a\in[1,m]$ or $i_a,j_a\in[m+1,m+n]$. This means that
 we may find
$\lambda,\mu\in\Lambda{(m|n,r)}$ and
$d\in\mathcal{D}^\circ_{\lambda\mu}$ such that
$F_{\vecr}(x_{\bsi_{-1},\bsj_{-1}}\otimes
t_{\bsi_0,\bsj_0})=x_{\bsi_\mu,\bsi_\lambda d}$. Clearly, by the construction, $P_{\vecr}\leq_W W_\la^d\cap W_\mu$.

Conversely,
if $\lambda,\mu\in\Lambda{(m|n,r)}$ and
$d\in\mathcal{D}^\circ_{\lambda\mu}$ such that
$P_{\vecr}\leq_W W^d_\lambda\cap W_\mu$, then both $W^d_{\la^{(0)}}\cap W_{\mu^{(0)}}$
and $W^d_{\la^{(1)}}\cap W_{\mu^{(1)}}$ contain $l$-parabolic subgroups.
Thus, if we write $\bsi=(i_1,i_2,\cdots,i_r)=\bsi_\mu$ and $\bsj=(j_1,j_2,\cdots,j_r)=\bsi_\la d$, then $x_{\mu,\lambda d}=x_{\bsi,\bsj}$ has factors of the form $x_{i_kj_k}^l$, where $i_k,j_k\in[1,m]$ or $[m+1,m+n]$. Moving all such factors to the right, we may rewrite $x_{\bsi\bsj}=x_{\bsi_{-1},\bsj_{-1}}x_{l\bsi_0,l\bsj_0}$. Clearly, $(\bsi_{-1},\bsj_{-1})$ satisfies the even-odd trivial intersection property.
Hence, we see that $x_{\bsi,\bsj}=F_{\vecr}(x)$ for some
$x\in\mathcal{A}_q(m|n,\vecr)$.
\end{proof}

As an application of Theorem \ref{N-x basis}, we now determine the kernel of the Brauer homomorphism.

\begin{theorem}\label{kernel}
If $\{x^*_{\mu,\lambda d}\mid
\lambda,\mu\in\Lambda{(m|n,r)},d\in\mathcal{D}^\circ_{\lambda\mu}\}$
is the dual basis of $\mathcal{B}^\vee$ in Lemma \ref{DR9.7}(2), then the kernel
of $\phi_\vecr$,  for $\vecr=(r_{-1},r_0)\in\mathcal{R}_r$,  is spanned by
$$\{x^*_{\mu,\lambda d}\mid\lambda,\mu\in\Lambda{(m|n,r)},d\in\mathcal{D}^\circ_{\lambda\mu},P_{\vecr}\nleq_W
W^d_\lambda\cap W_\mu\}.$$
In particular, $\ker \phi_\vecr=I_F(P_{r_0-1},r).$\end{theorem}
\begin{proof} Suppose that
$\phi_\vecr(x^*_{\mu,\lambda d})=0$. Thus we
have $x^*_{\mu,\lambda d}(F_{\vecr}(a))=0$ for all
$a\in \mathcal{A}_q(m|n,\vecr)$. The lemma above implies that
$P_{\vecr}\nleq_W W^d_\lambda\cap W_\mu$.

Now,
the largest $l$-parabolic subgroup $P$ satisfying $P_{\vecr}\not\leq_W P$ is $P_{r_0-1}$,
and by Theorem \ref{N-x basis},  $I_F(P,r)$ is spanned by
$$\{x^*_{\mu,\la d}\mid \la,\mu\in\La(m|n,r), d\in\sD_{\la\mu}^\circ,P\leq_W
W^d_\lambda\cap W_\mu\}.$$
The last assertion follows from the filtration \eqref{filtration} and Theorem \ref{N-x basis}.
\end{proof}


\section{Shrinking defect groups by Brauer homomorphisms}

In this section we continue to assume that $R=F$ is a field. We shall prove that the Brauer
homomorphism $\phi_\vecr$ sends certain
primitive idempotents in $\sS_F(m|n,r)$ to primitive idempotents in $\sS_F{(m|n,\vecr)}$ with the trivial defect group. Using the filtration \eqref{filtration}, we make the following definition.

\begin{definition}\label{defect group}
For a primitive idempotent $e\in \sS_F(m|n,r)$, there is a number
$k=k(e)$ such that $e\in I_F(P_k,r),e\not\in I_F(P_{k-1},r)$. We set
$D(e)=P_{k}$ and call $D(e)$ the {\it defect group} of $e$.
\end{definition}

Every primitive idempotent $e$ of $\sS_F(m|n,r)$ defines an indecomposable $\mathcal{H}_F$-submodule $V_F(m|n)^{\otimes r}e$
of $V_F(m|n)^{\otimes r}$. We now determine its defect group.
The following  result is a super version of {\cite[Th.~4.8]{DU}}. For completeness, we include a proof.

\begin{theorem}\label{VAP} If $e\in \sS_F(m|n,r)$ is a primitive idempotent, then
the defect group $D(e)$ of $e$ is the vertex of the indecomposable $\sH_F$-module $V_F(m|n)^{\otimes r}e$.
\end{theorem}
\begin{proof} Let $T=V_F(m|n)^{\otimes r}$ and suppose $W_\theta$ is the vertex of $Te$. Then $Te$ is $\sH_\theta$-projective and so
$$\aligned
e\sS_F(m|n,r)e&=\End_{\sH_F}(Te)=N_{W,W_\theta}(\End_{\sH_\theta}(Te))= N_{W,W_\theta}(e\End_{\sH_\theta}(T)e)\\
&=
eN_{W,W_\theta}(\End_{\sH_\theta}(T))e
=eI_F(W_\theta,r)e,\;\text{by Theorem \ref{IT}}
\endaligned$$
Hence, $e\in I_F(W_\theta,r)$ and $D(e)\leq_WW_\theta$. However, by Theorem \ref{IT} again,
the equalities above continue to hold with $W_\theta$ replaced by $D(e)$. Hence, $Te$ is $\sH_{D(e)}$-projective which implies $W_\theta\leq_WD(e)$. Hence, $W_\theta=_WD(e)$.
\end{proof}

\begin{corollary}\label{BHP}
Let $e\in \sS_F(m|n,r) $ be a primitive
idempotent and let $
\phi_{\vecr}  $ be the Brauer homomorphism associated with $\vecr\in \mathcal R_r$ . Then
$\phi_{\vecr}(e)\neq 0$ if and only if
$P_{\vecr}$ is conjugate to a parabolic subgroup of $D(e)$.
\end{corollary}

\begin{proof}For $\vecr=(r_{-1},r_0)$, by Theorem \ref{kernel}, $\phi_{\vecr}(e)\neq 0$ if and only if $e\not\in\text{ker}(\phi_\vecr)=I_F(P_{r_0-1},r)$. This is equivalent to $P_\vecr\leq_W D(e)$.
 \end{proof}

Choose $\vecr\in\mathcal{R}_r$ satisfying $D(e)=P_\vecr$. Then we will prove that $\phi_\vecr(e)$ is primitive and determine its defect group. This is the key to the classification theorem in next section.

For $\vecr\in\mathcal R_r$, let
\begin{equation}\label{rhotheta}
\begin{aligned}
\rho&=\rho_{\vecr}=(\rho_{-1},\rho_0)=(r_{-1},lr_0)\\
\theta&=\theta_{\vecr}=(\underbrace{1,\cdots,1}_{r_{-1}},\underbrace{l,\cdots,l}_{r_0})
\end{aligned}
\end{equation}
Since $W_\rho\leq W$,
$V_F(m|n)^{\otimes r}$  is a $\mathcal{H}_F$-module, and
$$V_F(m|n)^{\otimes r}=V_F(m|n)^{\otimes
\rho_{-1}}\otimes V_F(m|n)^{\otimes \rho_0},$$
restriction makes $V_F(m|n)^{\otimes r}$ into an $\mathcal{H}_\rho$-module.

Let $\sS_F(m|n,\rho)=\End_{\mathcal{H}_\rho}(V_F(m|n)^{\otimes r})$,
then (see \eqref{Schur algebra product})
\begin{equation}\label{yyy}
\sS_F(m|n,\rho)\cong \sS_F(m|n,\rho_{-1})\otimes
\sS_F(m|n,\rho_0).
\end{equation}
Recall from \eqref{vecLa} and \eqref{scrD} that
$$\vec{\La}(m|n,\rho)=\Lambda{(m|n,\rho_{-1})}\times
\Lambda{(m|n,\rho_0)},$$
 and,
 for $\vec{\lambda},\vec{\mu}\in
\vec{\La}(m|n,\rho)$,
$$\mathcal{D}^\circ_{\vec{\lambda}\vec{\mu};\rho}=
\{d\in\mathcal{D}_{\vec{\lambda}\vec{\mu}}\mid
d=d_1d_2\in W_{\rho_{-1}}\times W_{\rho_0},d_1\in
\mathcal{D}^\circ_{\la_{(-1)}\mu_{(-1)}},d_2\in
\mathcal{D}^\circ_{\lambda_{(0)}\mu_{(0)}}\}\subseteq W_\rho.$$
So $\sS_F(m|n,\rho)$ has a basis (see \eqref{same})
$$\{N_{W_\rho,W_{\vec\la}^d\cap W_{\vec\mu}}(e_{\vec{\mu},\vec{\lambda}d})\mid
\vec{\lambda},\vec{\mu}\in
\vec{\La}(m|n,\rho), d\in
\mathcal{D}^\circ_{\vec{\lambda}\vec{\mu};\rho}\}.$$

Let $P$ be an $l$-parabolic subgroup of $W_\rho$ and
define $I_F(P,\rho)$ to be the subspace of $\sS_F(m|n,\rho)$
spanned by $N^d_{\vec\mu\vec{\lambda}}$ where
$\vec{\lambda},\vec{\mu}\in
\vec{\La}(m|n,\rho),d\in\mathcal{D}^\circ_{\vec{\lambda}\vec{\mu};\rho}$
and the maximal $l$-parabolic subgroup of $W_{\vec\la}^d\cap W_{\vec\mu}$ is conjugate in $W_\rho$ to
parabolic subgroup of $P$.
By \eqref{yyy}, the first two items of the following results are the $W_\rho$ version of Corollary \ref{MC2} and
Theorem \ref{IT}.

\begin{lemma}\label{ISS} Let $P=W_\theta$ be an arbitrary $l$-parabolic subgroup of $W_\rho$.
 \begin{itemize}
 \item[(1)] The space $I_F(P,\rho)$ is an ideal of $\sS_F(m|n,\rho)$.
\item[(2)]
$I_F(W_\theta,\rho)=N_{W_\rho,W_\theta}(\End_{\mathcal{H}_\theta}(V_F(m|n)^{\otimes
r})).$
\item[(3)] If $\xi$ is a composition of $r$ such that $W_\xi\leq
W_\rho$ and $P_\xi=W_\theta$ is the maximal l-parabolic subgroup of $W_\xi$, then
$$N_{W_\rho,W_\theta}(\End_{\mathcal{H}_\theta}(V_F(m|n)^{\otimes
r}))=N_{W_\rho,W_\xi}(\End_{\mathcal{H}_\xi}(V_F(m|n)^{\otimes
r})).$$
\end{itemize}
\end{lemma}
\begin{proof} (3) The RHS is clearly contained in the LHS. The rest of the proof follows from a similar argument as in the proof of \cite[4.3]{DD}. The Pioncar\'e polynomials involved in the proof there have to be replaced by the product of the Pioncar\'e polynomials in $q^2$ for the even parts and the Pioncar\'e polynomials in $q^{-2}$ for the odd parts.\end{proof}

We fix an element $\vecr=(r_{-1},r_0)\in\mathcal R_r$,
and define $\rho=\rho_{\vecr}$ and $\theta=\theta_{\vecr}$ as in \eqref{rhotheta}.
Let
$$\mathcal{A}_q(m|n,\rho)=\mathcal{A}_q(m|n,r_{-1})\otimes
\mathcal{A}_q(m|n,lr_0).$$
Recall the definition of the Frobenius morphism $\mathcal{F}$. Define the coalgebra homomorphism
$$F_\rho=\text{id}\otimes \mathcal{F}:\mathcal{A}_q(m|n,\vecr)\rightarrow
\mathcal{A}_q(m|n,\rho).$$
 Consider the
dual of $F_\rho$:
$$\psi_\rho=F_\rho^*:\sS_F(m|n,\rho)\rightarrow
\sS_F(m|n,\vecr).$$ Since $F_\rho$ is injective,
$\psi_\rho$ is a surjective algebra homomorphism.
On the other hand, the multiplication map
$$\mu: \mathcal{A}_q(m|n,\rho)\rightarrow \mathcal{A}_q(m|n,r),$$
 induces an algebra monomorphism
$$\iota_\rho=\mu^*:\sS_F(m|n,r)\hookrightarrow\sS_F(m|n,\rho),$$
defined by $\iota(f)=f\circ\mu$ for all $f\in\sS_F(m|n,r)$.
Thus, taking the dual of the relation $\mu\circ\mathcal{F}_\rho=\mathcal{F}_\vecr$ gives the following commutative diagram (recall $\rho=\rho_\vecr$):
\begin{equation}\label{comm diag}
{\unitlength=1cm
\begin{picture}(5,3.2)
\put(-.5,2.5){$\sS_F(m|n,\rho)$}
\put(4.1,2.5){$\sS_F(m|n,r)$}
\put(-0.5,0.5){$\sS_F(m|n,\vecr)$}
\put(3.9,2.6){\vector(-1,0){2.3}}
\put(4.2,2.4){\vector(-3,-2){2.7}}
\put(0.5,2.3){\vector(0,-1){1.35}}
\put(2.5,2.8){$\iota_\rho$}
\put(2.5,0.95){$\phi_\vecr$} \put(0.05,1.6){$\psi_\rho$}
\end{picture}}
\end{equation}
We claim that $\iota_\rho$ coincides with the inclusion
$\End_{\mathcal{H}_R}(V_F(m|n)^{\otimes r})\subseteq \End_{\mathcal{H}_\rho}(V_F(m|n)^{\otimes
r})$.
Indeed, for $i=-1,0$, let
$
\delta_{i}:V_F(m|n)^{\otimes \rho_{i}}\rightarrow \mathcal{A}_q(m|n,\rho_i)\otimes
V_F(m|n)^{\otimes \rho_i}$ be the comodule structure map as defined in \eqref{comodule map} and let
$\delta_\rho=(23)(\delta_{-1}\otimes\delta_0):V_F(m|n)^{\otimes r}\to\sA_q(m|n,\rho)\otimes V_F(m|n)^{\otimes r}$.
Then it is direct to check that $(\mu\otimes1)\circ \delta_\rho=\delta:V_F(m|n)^{\otimes r}\to\sA_q(m|n,r)\otimes V_F(m|n)^{\otimes r}$. (The signs involved in both sides are the same!) Thus, for any $v\in V_F(m|n)^{\otimes r}$ and $f\in\sS_F(m|n,r)$, by \eqref{module map},
$$\aligned
v\cdot f&=(f\otimes 1)\delta(v)=(f\otimes 1)(\mu\otimes 1)\delta_\rho(v)\\
&=(f\mu\otimes 1)\delta_\rho(v)=v\cdot\iota_\rho(f),
\endaligned$$
proving the claim.


By  the definition of $\psi_\rho$ and Theorem \ref{kernel}, we have
\begin{equation}\label{kernel psi}
\ker
(\psi_\rho)=\sS_F(m|n,r_{-1})\otimes I_F(P_{r_0-1},r_0l)=\sum_{P_{r_0}\nleq_{W_{\rho}}P}I_F(P,\rho).
\end{equation}

Consider the ideal of the form $I_F(P_{k},r_{-1})\otimes S(m,n)_{r_0}$  in  $\sS_F(m|n,\vecr)$ and  the ideal
$$I_F(P_k,\vecr):=\phi_\vecr(I_F(P_{r_0+k},r))=I_F(P_{r_0+k},r)/I_F(P_{r_0-1},r).$$ We now show that both ideals are the same.

\begin{lemma}\label{new}
For $\vecr=(r_{-1},r_0)\in\sR_r$ and $r_0\leq k\leq s$, where $r=sl+t$ with $0\leq t<l$, the restriction $\phi_\vecr:I_F(P_k,r)\to I_F(P_{k-r_0},r_{-1})\otimes S(m,n)_{r_0}$ is surjective. In other words, we have
\begin{equation}
I_F(P_k,\vecr)=I_F(P_{k-r_0},r_{-1})\otimes S(m,n)_{r_0}.
\end{equation}\end{lemma}

\begin{proof} By Theorem \ref{N-x basis}, $I_F(P_k,r)$ is spanned by all $x_{\mu,\la d}^*$ satisfying
$P_{\la d\cap\mu}\leq_W P_k$. We first prove that $\phi_\vecr(I_F(P_k,r))\subseteq I_F(P_{k-r_0},r_{-1})\otimes S(m,n)_{r_0}$. This is seen as follows.

Let $W_\theta=P_k$ and, for any $N_{W,W_\theta}(a)\in I_F(P_k,r)=N_{W,W_\theta}(\End_{\sH_\theta}(V_F(m|n)^{\otimes r}))$, write by Lemma \ref{Mackey}
$$N_{W,W_\theta}(a)=\sum_{d\in\mathcal{D}_{\theta\rho}}N_{W_\rho,W^d_\theta\cap
W_\rho}(\mathcal{T}_{d^{-1}}a\mathcal{T}_d).$$
If we put $W_\alpha=W^d_\theta\cap
W_\rho=W_{\alpha_{(-1)}}\times W_{\alpha_{(0)}}$, where $\alpha_{(i)}$ is a composition of  $\rho_{i}$
for $i=-1,0$, then
$$\End_{\sH_\alpha}(V_F(m|n)^{\otimes r})=\End_{\sH_{\alpha_{(-1)}}}(V_F(m|n)^{\otimes r_{-1}})\otimes \End_{\sH_{\alpha_{(0)}}}(V_F(m|n)^{\otimes r_0l}).$$
So there exist $a_i(d)\in \End_{\sH_{\alpha_{(i)}}}(V_F(m|n)^{\otimes \rho_i})$ such that $\mathcal{T}_{d^{-1}}a\mathcal{T}_d=a_{-1}(d)\otimes a_0(d)$. Hence,
$$N_{W,W_\theta}(a)=\sum_{d\in\mathcal{D}_{\theta\rho}}N_{W_{\rho_{-1}},W_{\alpha_{(-1)}}}(a_{-1}(d))\otimes N_{W_{\rho_{0}},W_{\alpha_{(0)}}}(a_{0}(d)).$$
If $P_{k-r_0}<_WP_{\alpha_{(-1)}}$, then $P_{\alpha_{(0)}}\leq_WP_{r_0-1}$ and so
$\mathcal{F}^*(N_{W_{\rho_{0}},W_{\alpha_{(0)}}}(a_{0}(d)))=0$. Hence,
$$\phi_\vecr(N_{W,W_\theta}(a))=\sum_{{d\in\mathcal{D}_{\theta\rho}}\atop{ P_{k-r_0}=_WP_{\alpha_{(-1)}(d)}}}N_{W_{\rho_{-1}},W_{\alpha_{(-1)}}}(a_{-1}(d))\otimes \mathcal{F}^*(N_{W_{\rho_{0}},W_{\alpha_{(0)}}}(a_{0}(d))),$$
which is in $ I_F(P_{k-r_0},r_{-1})\otimes S(m,n)_{r_0}$.

We now prove the surjectivity.
By the proof of Lemma \ref{image}, every basis element $x_{\mu_{(-1)},\la_{(-1)}d_{-1}}^*\otimes
t_{\mu_{(0)},\la_{(0)}d_0}^*$ in $I_F(P_{k-r_0},r_{-1})\otimes S(m,n)_{r_0}$ has a pre-image $x_{\mu,\la d}^*$. Let $P_{\la d\cap\mu}$ be the maximal $l$-parabolic subgroup of $W_\la^d\cap W_\mu$.
Then the same proof shows that
 $P_{\la d\cap\mu}=_WP_{k-r_0}\times P_{r_0}$ where $P_{k-r_0}$ is conjugate to the maximal $l$-parabolic subgroup of $W_{\la_{(-1)}}^{d_{-1}}\cap W_{\mu_{(-1)}}\leq \fS_{r_{-1}}$. Hence, $x_{\mu,\la d}^*\in I_F(P_k,r),$ proving the surjectivity.
\end{proof}

\begin{corollary}\label{RBH} Maintain the nation above.
For $\vecr\in \mathcal
R_r,$ let $\rho=\rho_{\vecr},\theta=\theta_{\vecr}$ be defined in \eqref{rhotheta} and let  $k\geq r_0$ with $kl\leq r$.
Then
$$\phi_{\vecr}(I_F(P_k,r))=
\psi_\rho(I_F(P_k,\rho)).$$
\end{corollary}
\begin{proof}  If $r=sl+t$ with $0\leq t< l$, $r_0\leq k\leq s$, then one sees easily that
$$I_F(P_k,\rho)=I_F(P_{k-r_0}, r_{-1})\otimes\sS_F(m|n,r_0l)+
\sum_{i=1}^{s-k}I_F(P_{k-r_0+i}, r_{-1})\otimes I_F(P_{r_0-i},r_0l).$$
Hence, $\psi_\rho(I_F(P_k,\rho))=
 \psi_\rho(I_F(P_{k-r_0}, r_{-1})\otimes\sS_F(m|n,r_0l))$
 and, by Lemma \ref{new},
 $$\phi_\vecr(I_F(P_k,r))=I_F(P_{k-r_0})\otimes S(m,n)_{r_0}=\psi_\rho(I_F(P_{k-r_0}, r_{-1})\otimes\sS_F(m|n,r_0l)).$$
 Hence, the assertion follows.
 \end{proof}

Applying the Brauer homomorphism $\phi_\vecr$ to the filtration \eqref{filtration} gives rise to a filtration of ideals of $\sS_F(m|n,\vecr)$:
\begin{equation}\label{filtration1}
0\subseteq I_F(P_0,\vecr)\subseteq I_F(P_1,\vecr)\subseteq\cdots\subseteq
I_F(P_{s-r_0},\vecr)=\sS_F(m|n,\vecr).
\end{equation}
Like Definition \ref{defect group}, we may use this sequence to define the {\it defect group} $D(\bar e)$ of a primitive idempotent $\bar e$ of $\sS_F(m|n,\vecr)$.
In particular, a primitive idempotent $\bar{e}\in\sS_F(m|n,\vecr)$ has the trivial defect group if
$\bar{e}\in I_F(P_0,\vecr)$.


\begin{theorem}\label{ML3}
Let $e,e'$ be idempotents $\sS_F(m|n,r)$. Suppose $e$ is primitive
with defect group $D(e)=P_k$. Suppose $\vecr=(r_{-1},r_0)\in\mathcal{R}_r$ and $P_\vecr\leq_WD(e)$ (so $r_0\leq k$). Then
$$\phi_{\vecr}(e\sS_F(m|n,r)e')=\phi_{\vecr}(e)\sS_F(m|n,\vecr)\phi_{\vecr}(e')$$
and $\phi_\vecr(e)$ is also primitive with defect group $P_{k-r_0}$. In particular,
if $D(e)=P_\vecr$, then $\phi_\vecr(e)$ has the trivial defect group.
\end{theorem}
\begin{proof} The first assertion is clear since $\phi_\vecr$ is an algebra homomorphism and  $\sS_F(m|n,\vecr)$ is the homomorphic image of $\sS_F(m|n,r)$ under $\phi_{\vecr}$.

To see the last assertion, we first notice that $\phi_\vecr(e)\neq0$ by Corollary \ref{BHP}. Since $e$ is primitive, $e\sS_F(m|n,r)e$
is a local ring. Applying the surjective map $\phi_\vecr$ to this local ring yields that
$\phi_{\vecr}(e\sS_F(m|n,r)e)=\phi_{\vecr}(e)\sS_F(m|n,\vecr)\phi_{\vecr}(e)$
is also a local ring. Hence, $\phi_{\vecr}(e)$ is a primitive idempotent
in $\sS_F(m|n,\vecr)$
 and $\phi_\vecr(e)\in I_F(P_{k-r_0},\vecr)$. The fact $D(\phi_\vecr(e))=P_{k-r_0}$ is clear from Isomorphism Theorems for ring homomorphisms,  since $\text{ker}(\phi_\vecr)=I_F(P_{r_0-1},r)\subseteq I_F(P_{k},r)$ for all $k\geq r_0$.
 \end{proof}

\section{Classification of irreducible $\sS_F(m|n,r)$-modules }

We first interpret the algebra $\sS_F(m|n,\vecr)$ as an endomorphism algebra of a certain tensor space. By \eqref{kernel psi} and the commutative diagram \eqref{comm diag}, we obtain for $\vecr=(r_{-1},r_0)$ and $\rho=(r_{-1},r_0l)$
$$\sS_F(m|n,\vecr)=\psi_\rho(\sS_F(m|n,\rho))\cong\sS_F(m|n,r_{-1})\otimes \bar\sS_F(m|n,r_0l),$$
where $\bar\sS_F(m|n,r_0l):=\sS_F(m|n,(0,r_0))=\sS_F(m|n,r_0l)/I_F(P_{r_0-1},r_0l).$

On the other hand, taking the dual of the Frobenius map
$$\mathcal{F}|_{r_0}:A(m,n)_{r_0}\to\sA_q(m|n,r_0l)$$ induces an algebra isomorphism
\begin{equation}\label{homog}
\bar\sS_F(m|n,r_0l)\cong\bigoplus_{i=0}^{r_0}(S(m,i)\otimes
S(n,r_0-i))=S(m,n)_{r_0}.
\end{equation}

Recall the even part $V_F(m|n)_0$ and the odd part $V_F(m|n)_1$ of the superspace $V_F(m|n)$. Fix $r_0\geq 0$ and define the subspace:
$$
(V_F(m|n)^{\otimes r_0})_\rmL=\bigoplus_{i=0}^{r_0} (V_F(m|n)_0)^{\otimes i}\otimes (V_F(m|n)_1)^{\otimes
r_0-i}.$$
 This is called the  ``Levi part'' of $V_F(m|n)^{\otimes r_0}$. Let the product of symmetric groups $\fS_{\{1,\ldots,i\}}\times\fS_{\{i+1,\ldots,r_0\}}$ act on the
 summand $(V_F(m|n)_0)^{\otimes i}\otimes (V_F(m|n)_1)^{\otimes
r_0-i}$ by place permutation for $\fS_{\{1,\ldots,i\}}$ and signed place permutation\footnote{In the Appendix, we will see that the signed permutation is naturally induced from the original Hecke algebra action.}
 for $\fS_{\{i+1,\ldots,r_0\}}$. In this way, putting
$$(\fS_{r_0})_\rmL:=\prod_{i=0}^{r_0}
(\fS_{\{1,\ldots,i\}}\times\fS_{\{i+1,\ldots,r_0\}}),$$
$(V_F(m|n)^{\otimes r_0})_\rmL$ becomes an $(\fS_{r_0})_\rmL$-module.
Now, the RHS of \eqref{homog} can be interpreted as the endomorphism algebra
$\End_{(\fS_{r_0})_\rmL}((V_F(m|n)^{\otimes r_0})_\rmL)$ so that we have an algebra isomorphism
\begin{equation}\label{Brauer iso}
\bar\sS_F(m|n,r_0l)\cong\End_{(\fS_{r_0})_\rmL}((V_F(m|n)^{\otimes r_0})_\rmL).
\end{equation}

In general, for $\vecr=(r_{-1},r_0)\in \mathcal{R}_r$, let
$$\aligned
V_F(m|n)^{\boxtimes \vecr}&:=V_F(m|n)^{\otimes
r_{-1}}\otimes (V_F(m|n)^{\otimes
r_0})_L\\
&=
\bigoplus_{i=0}^{r_0}V_F(m|n)^{\otimes
r_{-1}}\otimes (V_F(m|n)_0)^{\otimes i}\otimes (V_F(m|n)_1)^{\otimes
r_0-i}.\endaligned
$$
Let $F(\fS_{r_0})_\rmL$ be the group algebra of $(\fS_{r_0})_\rmL$ and let
$$\aligned
\mathcal{H}_\vecr&:=\mathcal{H}_F(\fS_{r_{-1}})\otimes F(\fS_{r_0})_\rmL\\
&\cong
\bigoplus_{i=0}^{r_0}\mathcal{H}_F(\fS_{r_{-1}})\otimes F \fS_{\{1,\ldots,i\}}\otimes F\fS_{\{i+1,\ldots,r_0\}}.\endaligned$$
Then
$$\aligned
\sS_F(m|n,\vecr)&=\bigoplus_{i=0}^{r_0}\sS_F(m|n,r_{-1})\otimes S(m,i)\otimes
S(n,r_0-i)\\
&\cong
\End_{\mathcal{H}_{\vecr}}(V_F(m|n)^{\boxtimes\vecr}).
\endaligned$$

If $\bar{e}\in
\sS_F(m|n,\vecr) $ is a primitive idempotent, then there are primitive idempotents
$e_{-1}\in \sS_F(m|n,r_{-1})$, $e_0\otimes e_1\in S(m,i)\otimes
S(n,r_0-i)$ for some $i$ such that $\bar e=e_{-1}\otimes e_0\otimes e_1$. By Lemma \ref{new}, $D(\bar{e})=D(e_{-1})$.
We need the following result. 

\begin{lemma}\label{MT}
(a) Let $e,e'$ be the idempotents of $\sS_F(m|n,r)$ with $e$ primitive
and $D(e)=P_{\vecr}$ for some $\vecr=(r_{-1},r_0)\in\mathcal{R}_r$. Then
$$V_F(m|n)^{\otimes r}e\mid V_F(m|n)^{\otimes r}e'\mbox{ if and only if
} V_F(m|n)^{\boxtimes\vecr}\phi_{\vecr}(e)\mid
V_F(m|n)^{\boxtimes\vecr}\phi_{\vecr}(e').$$

(b) If $\bar{e}\in \sS_F(m|n,\vecr)$ is a
primitive idempotent with $D(\bar{e})=1$, then $\bar{e}$ is
equivalent to $\phi_{\vecr}(e')$ for some primitive
idempotent $e'\in \sS_F(m|n,r)$ with $D(e')=P_{\vecr}$.

(c) Let $e$ be a primitive idempotent of $\sS_F(m|n,r)$. If $D(e)=1$, then
$V_F(m|n)^{\otimes r}e$ is a projective indecomposable
$\mathcal{H}_F$-module.
\end{lemma}

\begin{proof}  With Corollary \ref{BHP}, Theorem \ref{ML3}, and Lemma \ref{new},
the proof for (a) and (b) is standard;
see \cite[4.1]{S}. So we omit it.

(c) By Theorem \ref{VAP}, the vertex of $V_F(m|n)^{\otimes r}e$ is
$D(e)$. If $\mathcal{H}_F$-module $M$ has the trivial vertex, then Higman Criterion
in Lemma \ref{Higman}(a) tells that $M$ is
a projective $\mathcal{H}_F$-module.
\end{proof}

Let $\Lambda^+(r)$ (resp. $\Lambda^+(N,r)$, $\La_{l\reg}^+(r)$) be the set
of partitions (resp., partitions with at most $N$ parts, $l$-regular partitions) of $r$. (A partition is called $l$-regular if no part is repeated $l$ or more times.)
For each $\vecr=(r_{-1},r_0)\in\sR_r$, let 
$$
\sP_\vecr=\begin{cases}\La_{l\reg}^+(r), &\text{ if } r_0=0;\\
\bigcup_{i=0}^{r_0}\{(\lambda,\xi,\eta)\mid \lambda\in
\La_{l\reg}^+(r_{-1}),\xi\in \Lambda^+(m,i),\eta\in
\Lambda^+(n,r_0-i)\},&
\text{ if }r_0>0\\
\end{cases}$$
and let
\begin{equation}\label{index set}
\sP_r=\bigcup_{\vecr\in\sR_r}\sP_\vecr.
\end{equation}

Let $E_r^\vecr$ (resp., $E_\vecr^1$) be the set
of nonequivalent primitive idempotents with  defect group  $P_\vecr$ (resp., the trivial defect group) in $\sS_F(m|n,r)$ (resp., $\sS_F(m|n,\vecr)$).
Set
$$E_r=\bigcup_{\vecr\in\sR_r}E_r^\vecr$$ to be the set of nonequivalent primitive idempotents of
$\sS_F(m|n,r)$.

\begin{theorem}\label{ML4} Assume $m+n\geq r$.
\begin{enumerate}
\item
For each $\vecr\in\mathcal R_r$, the
Brauer homomorphim $\phi_\vecr$ induces a bijection between the sets $E_r^\vecr$ and
$E_\vecr^1$.
\item There is a bijection $\pi:E_\vecr^1\to\sP_\vecr$.
\end{enumerate}
Hence,
there is a bijective map $\pi$ from $E_r$ to $\mathcal{P}_r$.
\end{theorem}
\begin{proof} For statement (1), let $\bar E_r^\vecr=\phi(E_r^\vecr)$ where $\phi=\phi_\vecr$. By Theorem \ref{ML3}, every element in $\bar E_r^\vecr$ is primitive with the trivial defect group. Thus, we may regard $\bar E_r^\vecr$ as a subset of $E_\vecr^1$ and consider the map $\phi:E_r^\vecr\to E_\vecr^1$. By Lemma \ref{MT}(a) we see that $\phi$ is injective, and by Lemma \ref{MT}(b), we see that $\phi$ is surjective, proving (1).

We now prove (2).
Pick $\bar e\in E_\vecr^1$. Then
$D(\bar e)=1$ and there exists $i\in[1,r_0]$ and primitive idempotents $e_{-1}\in\sS_F(m|n,r_{-1})$, $e_0\in S(m,i)$, and $e_1\in S(n,r_0-i)$\linebreak such that $\bar e=e_{-1}\otimes e_0\otimes e_1$. Since $D(e_{-1})=D(\bar e)=1$,
by Lemma \ref{MT}(c),
$V_F(m|n)^{r_{-1}}e_{-1}$ is a
projective indecomposable
$\mathcal{H}_{r_{-1}}$-module,  where $\mathcal{H}_{r_{-1}}=\sH_F(\fS_{r_{-1}})$.  Since the PIMs of $\sH_{r_{-1}}$ are labelled by
$\La_{l\reg}^+(r_{-1})$ (see \cite{DJ}), $e_{-1}$ determines a unique $\la\in\La_{l\reg}^+(r_{-1})$.
Similarly, idempotents $e_0$ and $e_1$ determines irreducible $S(m,i)$-module $L(\xi)$ and $S(n,r_0-i)$-module $L(\eta)$, respectively, where $\xi\in\Lambda^+(m,i)$ and
$\eta\in\Lambda^+(n,r_0-i)$. Hence, $\bar e$ determines a unique triple $(\la,\xi,\eta)\in\sP_\vecr$ and, putting $\pi(\bar e)=(\la,\xi,\eta)$ defines a map $\pi:E_\vecr^1\to\sP_\vecr$. It is clear that $\pi$ is injective.

For the subjectivity, we assume that $m+n\geq r$. In particular, $m+n\geq r_{-1}$. Thus, $\sH_F(\fS_{r_{-1}})$ is a direct summand of the tensor space $V_F(m|n)^{\otimes r_{-1}}$.
Thus, for every triple $(\la,\xi,\eta)\in\sP_\vecr$, the PIM $Q_\la$ of $\sH_{r_{-1}}$ corresponding to $\la$ is a direct summand of the tensor space. Hence, there is  a primitive idempotent $e_{-1}$ such that $Q_\la\cong V_F(m|n)^{\otimes r_{-1}}e_{-1}$.
Similar, there exists primitive idempotents $e_0,e_1$ such that
$$(V_F(m|n)_0)^{\otimes i}e_0\cong L(\xi)\text{ and } (V_F(m|n)_1)^{\otimes r_0-i} e_1\cong L(\eta),$$
where $L(\xi)$ (reps., $L(\eta)$) is the irreducible $S(m,i)$-module (reps. $S(n,r_0-i)$-module) with highest weight $\xi$ (reps., $\eta$).
Hence, $\pi(e_{-1}\otimes e_0\otimes e_1)=(\la,\xi,\eta)$ and $\pi$ is surjective.
\end{proof}

For each $e\in E_r$, if $\pi(e)=(\la,\xi,\eta)$, define
$$L_q(\la,\xi,\eta)=\sS_F(m|n,r)e/\text{Rad}(\sS_F(m|n,r)e).$$
When $\pi(e)=\la\in\La_{l\reg}^+(r)$, we write $L_q(\la)$.

\begin{corollary}\label{tensor product theorem} Assume $m+n\geq r$. 
\begin{enumerate}
\item The set
$\{L(\lambda,\xi,\eta)\mid (\lambda,\xi,\eta)\in \mathcal{P}_r\}$ forms a
complete set of non-isomorphic irreducible $\sS_F(m|n,r)$-modules.
\item If $(\lambda,\xi,\eta)\in \mathcal{P}_\vecr$ with $P_\vecr>1$, then the $\sS_F(m|n,\vecr)$-module
$L_q(\la)\otimes L(\xi)\otimes L(\eta)$ becomes an $\sS_F(m|n,r)$-module by inflation and we have
$$L_q(\la,\xi,\eta)\cong L_q(\la)\otimes L(\xi)\otimes L(\eta).$$
\end{enumerate}
\end{corollary}

\begin{remarks}\label{Alperin}
 (1) Each element of $\mathcal{P}_r$ can also be regarded as  a pair $(P, Q)$ where $P$ is an $l$-parabolic subgroup and $Q$ is a PIM of the quotient algebra defined by $P$.  Equivalence classes associated with an equivalence relation on all such pairs play the role of ``weights'' as described in Alperin's weight conjecture. 

 (2) If we know the classification of all primitive idempotents with the trivial defect group in $\sS_F(m|n,r)$ 
 and label them by $\La_r$, then the condition $m+n\geq r$ can be removed by replacing the set $\La_{l\reg}^+(r_{-1})$ by the labelling set $\La_{r_{-1}}$. 

Moreover, when $m+n<r$, the classification can possibly be obtained by the Schur functor:
$$\mathscr{S}_\vep:\sS_F(m'|n',r)\mathsf{-Mod}\longrightarrow \sS_F(m|n,r)\mathsf{-Mod}, M\longmapsto \vep M,$$
where  $m'\geq m$, $n'\geq n$, $m'+n'\geq r$, and $\vep$ is given in Remark \ref{m'|n' case}.
See \cite[6.5]{BK} for the Schur superalgebra case and Appendix II for a comparison.

(3) Though our approach didn't offer a construction for irreducible modules, the tensor product structure shown in Corollary \ref{tensor product theorem}(2) reduces the problem  to the construction for $l$-restricted partitions.

(4) If $r<l$, then every idempotent $e\in\sS_F(m|n,r)$ has the trivial defect group and defines a PIM of the
 Hecke algebra $\sH_F$.  Thus, there is a bijection between the PIMs of $\sH_F$ and the
 non-equivalent primitive idempotents in $\sS_F(m|n,r)$ where $m+n\geq r$. This is the semismple case similar to the situation described in Proposition \ref{DR8.1}.
\end{remarks}

In the classical (i.e., non-quantum) case, Donkin gave a classification of irreducible $S_k(m|n,r)$-modules under the assumption $m,n\geq r$,  where $k$ is a field of a positive characteristic $p$. We end this section with a comparison.

Assume now $m,n\geq r$. Let $S_k(m|n,r)=\sS(m|n,r)\otimes_\sZ k$ be the Schur superalgebra obtained by specialising $\up$ to 1. By \cite[(3)]{D}, non-isomorphic irreducible $S_k(m|n,r)$-modules
are in one-to-one correspondence with the set
$$\{(\la,\mu)\in\La^+\times\La^+\mid r=|\la|+p|\mu|\},$$
where $\La^+$ denotes the set of all partitions. Clearly,  this set is the same as the set
$$\mathcal{P}'_r=\{(\tau,\nu)\mid
\tau\in\Lambda^+(s),\nu\in\Lambda^+(t), s,t\in\mathbb N, s+lt=r\}.$$
With the assumption $m,n\geq r$ and taking $l=p$, the set $\sP_r$ defined in \eqref{index set} has the form
$$\aligned
\sP_r&=\{(\la,\xi,\eta)\mid \la\in
\La_{l\reg}^+(r_{-1}),\xi\in \La^+(i),\eta\in
\Lambda^+(r_0-i),i\in[0,r_0],(r_{-1},r_0)\in\sR_r\}\\
&=\{(\la,\xi,\eta)\mid \la\in
\La_{l\reg}^+(i),\xi\in \La^+(j),\eta\in
\Lambda^+(k),i, j,k\in\mathbb{N},i+l(j+k)=r\}
\endaligned
$$

We claim that there is a bijective map
$g:\mathcal{P}_r\rightarrow \mathcal{P}'_r$ defined
by setting $g((\lambda,\xi,\eta))=(\lambda^\iota+l\xi^\iota,\eta)$ for
$(\lambda,\xi,\eta)\in\mathcal{P}_r$, where $\mu^\iota$ is the partition dual to $\mu$.
Indeed, for $(\lambda,\xi,\eta)\in\mathcal{P}_r$, since
$\lambda\in\Lambda^+_l{(i)},\xi\in\Lambda^+(j),\eta\in\Lambda^+(k)$, then
$\lambda^\iota+l\xi^\iota\in\Lambda^+(i+lj)$. Thus,
$(\lambda^\iota+l\xi^\iota,\eta)\in\sP'_r$. So $g$ is a map from $\mathcal{P}_r$ to $\mathcal{P}'_r$.
For $(\mu,\alpha,\beta)\in\mathcal{P}_r$, if
$(\mu,\alpha,\beta)\neq (\lambda,\xi,\eta)$, it is clear, by the fact that $\la$ and
$\mu$ are $l$-regular partitions, that $(\mu^\iota+l\alpha^\iota,\beta)\neq
(\lambda^\iota+l\xi^\iota,\eta)$. Hence $g$ is injective. For
$(\tau,\nu)\in\mathcal{P}'_r$, suppose
$\tau=(\tau_1,\tau_2,\cdots,\tau_m)$. Assume
$\tau_j=\lambda_j+l\mu_j$ where $0\leq\lambda_j<l$, $j=1,\cdots,m$.
Set
$\lambda^\iota=(\lambda_1,\lambda_2,\cdots,\lambda_m),\mu^\iota=(\mu_1,\mu_2,\cdots,\mu_m)$.
Then $(\lambda,\mu,\nu)\in\mathcal{P}_r$ and
$g(\lambda,\mu,\nu)=(\tau,\nu)$. We have proved the claim.

By the claim, we see that under the assumption $m,n\geq r$, our labelling set is the same as Donkin's labelling set. Thus, this shows that the quantum classification in the $m,n\geq r$ case is a $q$-analogue of the classical classification. 

\section{Appendix I: Brauer homomorphisms without Frobenius}
In \S11, the isomorphism \eqref{Brauer iso} was established by taking the dual of the Frobenius morphism. However,
the Brauer homomorphism, originated from the group representation theory, has its own definition, see \cite[\S3]{DU3}. In this section, we provide a proof for \eqref{Brauer iso} without using the Frobenius morphism. Note that this proof is much simpler than the proof for $q$-Schur algebras given in \cite{DU3}.
We first look at the structure of $\bar\sS_F(m|n,r_0l)$.
\begin{lemma} The algebra
$\bar\sS_F(m|n,r_0l)$ has a direct sum decomposition into centraliser subalgebras
$$\bar\sS_F(m|n,r_0l)=\bigoplus_{i=0}^r\bar\sS_F(m|n,r_0l)_i.$$\end{lemma}

 \begin{proof} Let $W=\fS_{r_0l}$ (i.e., we assume $r=r_0l$). The maximal $l$-parabolic subgroup of $W$ is unique and equal to $P_{r_0}$.
By Theorem \ref{RNB}, $\sS_F(m|n,r_0l)$ has basis
$$\mathcal{B}=\{N_{W,W^d_\lambda\cap W_\mu}(e_{\mu,\lambda d})\mid
\mu,\lambda\in\Lambda(m|n,r_0l),d\in\mathcal{D}^\circ_{\lambda\mu}\}.$$
If the maximal $l$-parabolic subgroup $P_{\lambda d\cap \mu}$
of $W^d_\lambda\cap W_\mu$ is not equal to $P_{r_0}$, then it must be conjugate to a parabolic subgroup of $P_{r_0-1}$ and so $N_{W,W^d_\lambda\cap
W_\mu}(e_{\mu,\lambda d})\in I_F(P_{r_0-1},r_0l)$.
Thus, it suffices to look at the basis elements satisfying
$P_{\lambda d\cap \mu}=P_{r_0}$. Then  $P_{r_0}\leq W_\la ^d\cap W_\mu$.
Since $P_{r_0}$ is the {\it unique} maximal $l$-parabolic subgroup of $W=\fS_{r_0l}$, $\la$ and $\mu$ must be of the form
$$\aligned
\lambda&=l\la'=(\lambda'_1l,\lambda'_2l,\cdots,\lambda'_ml\mid\lambda'_{m+1}l,\cdots,\lambda'_{m+n}l)\\\mu&=l\mu'=(\mu'_1l,\mu'_2l,\cdots,\mu'_ml\mid\mu'_{m+1}l,\cdots,\mu'_{m+n}l),\endaligned$$
where $\la',\mu'\in\Lambda(m|n,r_0)$ and $d\in N_{\fS_{r_0l}}(P_{r_0})$, the normaliser of $P_{r_0}$.
Moreover, $N_{\fS_{r_0l}}(P_{r_0})=P_{r_0}\rtimes\widetilde{\fS}_{r_0}$, where $\widetilde{\fS}_{r_0}\cong {\fS}_{r_0}$ is the subgroup of $\fS_{r_0l}$ generated by
$$\aligned
\tilde s_1&=(1,l+1)(2,l+2)\cdots(l,2l), \\
\tilde s_2&=(l+1,2l+1)(l+2,2l+2)\cdots(2l,3l),\\
&\ldots\ldots,\\
\tilde s_{r_0-1}&=((r_0-2)l+1,(r_0-1)l+1)((r_0-2)l+2,(r_0-1)l+2)\cdots((r_0-1)l,r_0l).\endaligned$$
Thus, $d\in\mathcal{D}^\circ_{\lambda\mu}\cap \widetilde{\fS}_{r_0}$.
This together with the even-odd trivial intersection property forces by the uniqueness of $P_{r_0}$
$$P_{r_0}=P^{00}_{\lambda d\cap \mu}\times P^{11}_{\lambda d\cap \mu}=P_{\mu^{(0)}}\times P_{\mu^{(1)}}=P_{\la^{(0)}}\times P_{\la^{(1)}}.$$
 (Here $P^{ii}_{\lambda d\cap \mu}$ denotes the maximal $l$-parabolic subgroup of $W_{\la d\cap\mu}^{ii}$.)
Hence, we must have $P_{\lambda^{(0)}}=P_{\mu^{(0)}}$ and
$P_{\lambda^{(1)}}=P_{\mu^{(1)}}$, and consequently,
$$|\lambda^{(0)}|=|\mu^{(0)}|=il\text{ and }
|\lambda^{(1)}|=|\mu^{(1)}|=(r_0-i)l$$
for some $0\leq i\leq r_0$. This also forces that  $d$ must have a decomposition $d=d_0d_1$ for some $d_0\in
\mathcal{D}_{\lambda^{(0)}\mu^{(0)}}\cap W'$ and $
d_1\in \mathcal{D}_{\lambda^{(1)}\mu^{(1)}}\cap W''$, where
$$W'=\widetilde{\fS}_{\{1,\ldots,i\}}\text{ and }W''=\widetilde{\fS}_{\{i+1,\ldots,r_0\}},$$
such that $W_{\la d\cap\mu}^{ii}=W_{\la^{(i)}}^{d_i}\cap W_{\mu^{(i)}}$ for $i=0,1$.

Let $\zeta_{\mu,\la}^{d}$ be the image of $N_{W,W_\la^d\cap W_\mu}(e_{\mu,\la d})$ in $\bar\sS_F(m|n,r_0l)$. Then
$\bar\sS_F(m|n,r_0l)$ has  basis
$\bar\sB=\cup_{i=0}^{r_0}\bar\sB_i$, where
$$\bar\sB_i=\{\zeta_{\mu\la}^{d}\mid\lambda,\mu\in l\Lambda(m,i)\times l\Lambda(n,r_0-i),d\in \mathcal{D}^\circ_{\lambda\mu}\cap\widetilde\fS_{r_0}\}.$$
Let $\bar\sS_F(m|n,r_0l)_i$ be the subspace spanned by
 $\bar\sB_i$ and let $e_i=\sum_{\la:|\la^{(0)}|=il}\zeta_{\la\la}^1$. Then $1=\sum_{i=0}^{r_0}e_i$, $e_ie_j=0$ for $i\neq j$, and
$\bar\sS_F(m|n,r_0l)_i=e_i
\bar\sS_F(m|n,r_0l)e_i$.  Our assertion follows. \end{proof}



Let $\sT_\rmL$ be the subspace of $\sT=V_F(m|n)^{\otimes r}$ spanned by tensors of the form
$$v_{j_1}^l\cdots v_{j_i}^lv_{j_{i+1}}^l\cdots v_{j_{r_0}}^l\quad(0\leq i\leq r_0)$$
such that $1\leq j_1,\ldots,j_{i}\leq m$
and  $m+1\leq j_{i+1},\ldots,j_{r_0}\leq m+n$ and let $\sT_\rmL'$ be the subspace spanned by the rest of the tensors so that $\sT=\sT_\rmL\oplus\sT_\rmL'$.
By looking at the action of elements in $\bar\sB$ on the quotient space $\sT/\sT_\rmL'$,
we directly prove the following.

\begin{proposition} We have algebra isomorphisms:
$$\bar\sS_F(m|n,r_0l)\cong \End_{(\widetilde{\fS}_{r_0})_\rmL}(\sT_\rmL)\cong
\End_{({\fS}_{r_0})_\rmL}(V_F(m|n)^{\otimes r_0})_\rmL).$$
\end{proposition}

\begin{proof} The second isomorphism is clear.
Suppose $N_{W, W_\nu}(e_{\mu,\lambda
d})\in{\mathcal{B}}$ has image $\zeta_{\mu\la}^{d}\in\bar\sB$ ($\nu={\lambda d\cap \mu}$). We first inspect the action of $N_{W, W_\nu}(e_{\mu,\lambda
d})$ on $v_\mu$:
\begin{equation*}
\begin{aligned}
(v_\mu)N_{W, W_{\nu}}(e_{\mu,\lambda d})
&=(v_\mu)\sum_{x\in\mathcal{D}_{\nu}\cap W_\mu}\mathcal{T}_{x^{-1}} (e_{\mu,\lambda d})\mathcal{T}_{x}\\
&=\sum_{x_0x_1\in\mathcal{D}_{\nu}\cap
(W_{\mu^{(0)}}\times
W_{\mu^{(1)}})}q^{\ell(x_0)}(-q^{-1})^{\ell(x_1)}(-1)^{\widehat{dx_0x_1}-\widehat{d}}v_{\bsi_\lambda
dx_0x_1}
\end{aligned}
\end{equation*}
Since $d=d_0d_1$ with $d_0\in
\mathcal{D}_{\lambda^{(0)}\mu^{(0)}}\cap W'$ and $
d_1\in \mathcal{D}_{\lambda^{(1)}\mu^{(1)}}\cap W''$,
it follows that $(-1)^{\widehat{d}}=(-1)^{\ell(d_1)}$, $(-1)^{\widehat{dx_0x_1}}=(-1)^{\ell(d_1x_1)}=(-1)^{\ell(d_1)}(-1)^{\ell(x_1)}$,
 and $v_{\bsi_\lambda
d}=v_{\bsi_{\lambda^{(0)}}d_0} v_ {\bsi_{\lambda^{(1)}}d_1}$.
Also, since, for any $x_0x_1\in \mathcal{D}_{\nu}\cap
W_{\mu}$ but $x_0x_1\not\in W'\times W''$, $v_{\bsi_{\lambda dx_0x_1}}\in\sT_\rmL'$
and since $l^2|\ell(x_0)$ and $l^2|\ell(x_1)$ for all $x_0\in W'$ and $x_1\in W''$, we have
$$\aligned
(v_\mu)N_{W, W_{\nu}}(e_{\mu,\lambda
d})&\equiv\sum_{x_0x_1\in\mathcal{D}_{\nu}\cap
W_{\mu}\cap (W'\times W'')}q^{\ell(x_0)}(q^{-1})^{\ell(x_1)}v_{\bsi_{\lambda ^{(0)}}d_0x_0}
v_{\bsi_{\lambda^{(1)}}d_1x_1}(\text{mod }\sT_{\rmL'})\\
&=\sum_{x_0x_1\in\mathcal{D}_{\nu}\cap
W_{\mu}\cap (W'\times W'')}v_{\bsi_{\lambda ^{(0)}}d_0x_0}
v_{\bsi_{\lambda^{(1)}}d_1x_1}(\text{mod }\sT_{\rmL'}).\endaligned
$$

This formula allows us to define an action of $\zeta_{\mu\la}^{d}$ on $\sT/\sT_\rmL'$ by setting, for any $\nu=l\nu'$ and $y\in\sD_\nu\cap\widetilde{\fS}_{r_0}$,
$$\bar v_{\bsi_\nu y}\cdot \zeta_{\mu\la}^d=\delta_{\nu,\mu}
\sum_{x_0x_1\in\mathcal{D}_{\nu}\cap
W_{\mu}\cap (W'\times W'')}\bar v_{\bsi_{\lambda ^{(0)}}d_0x_0}
\bar v_{\bsi_{\lambda^{(1)}}d_1x_1}y=\delta_{\nu,\mu}(\bar v_\mu)\xi_{\mu\la}^d y, $$
where $\bar v_\bsi=v_\bsi+\sT_\rmL'$ and $\xi_{\mu\la}^d=\sum_{x\in\mathcal{D}_{\nu}\cap
W_{\mu}\cap (W'\times W'')} x^{-1}e_{\mu,\la d}x\in\End_{(\widetilde{\fS}_{r_0})_\rmL}(\sT_\rmL)$.
Finally, it is easy to check that the linear isomorphism
$$h:\bar\sS_F(m|n,r_0l)\longrightarrow \End_{(\widetilde{\fS}_{r_0})_\rmL}(\sT_\rmL),\;
\zeta_{\mu\la}^d\longmapsto \xi_{\mu\la}^d$$
is an algebra homomorphism.
\end{proof}

\section{Appendix II: a comparison with a result of Brundan--Kujawa}
In \cite{BK}, Brundan--Kujawa obtained a complete classification of irreducible modules of Schur superalgebras for all $m,n,r$. Their approach is quite different. First, they used a construction of irreducible modules via Verma modules over the super hyperalgebra  $U_k$ of $\mathfrak{gl}(m|n)$. Second, they used an isomorphism  of the Hopf superalgebra $U_k$ and the distribution superalgebra Dist$(G)$ of the supergroup $G=GL(m|n)$ \cite[3.2]{BK} and  a category equivalence between $G$-supermodules and integrable
$U_k$-supermodules \cite[3.5]{BK} to obtain a complete set of irreducible $G$-supermodules. Finally, by determining the polynomial representations of $G$ and applying the Schur functor to Donkin's result, they reached an explicit description of the index set $\La^{++}(m|n,r)$. We now show that, if $m+n\geq r$ and $l=p$, then there is a bijection from $\mathcal P_r$ to $\La^{++}(m|n,r)$.


There are a quite few ingredients in order to describe  the set $\La^{++}(m|n,r)$. Following \cite{BK},
 we identify $\lambda\in\La^+(r)$ with its
Young diagram
$$\lambda=\{(i,j)\in\mathbb{Z}_{>0}\times\mathbb{Z}_{>0} \mid j\leq
\lambda_i\}$$ and refer to $(i,j)\in\lambda$ as the node in the
$i$th row and $j$th column. Then the rim of $\lambda$ is defined to be
the set of all nodes $(i,j)\in\lambda$ such that
$(i+1,j+1)\notin\lambda$. The $p$-rim is a certain subset of the
rim, defined as the union of the $p$-segments. The first $p$-segment
is simply the first $p$ nodes of the rim, reading along the rim from
left to right. The next $p$-segment is then obtained by reading off
the next $p$ nodes of the rim, but starting from the column
immediately to the right of the rightmost node of the first
$p$-segment. The remaining $p$-segments are obtained by repeating
this process. Of course, all but the last $p$-segment contain
exactly $p$ nodes, while the last may contain less.

Let $J(\lambda)$ be the partition obtained from $\lambda$ by
deleting every node in the $p$-rim that is at the rightmost end of a
row of $\lambda$ but that is not the $p$th node of a $p$-segment.
Let $ j(\lambda) = |\lambda|-|J(\lambda)|$ be the total number of
nodes deleted. Clearly, for partition $\mu,\nu$, $j(\mu + p\nu) = j(\mu).$
Let  
$$\Lambda^+(m|n,r)=\{\lambda\in
\Lambda(m|n, r):\lambda_1\geq\cdots\geq
\lambda_m,\lambda_{m+1}\geq\cdots\lambda_{m+n}\}.$$
For $\la\in\La^+(m|n,r)$, let $t(\la)=(\la_{m+1},\ldots,\la_{m+n})$. Note that, if we use the notation
$\la= (\lambda^{(0)}\mid\lambda^{(1)})$, then $t(\la) =\la^{(1)}$.

We can now describe the index set $\La^{++}(m|n,r)$ for irreducible modules of the Schur superalgebra:
\begin{equation}\label{BK,6.5}
\Lambda^{++}(m|n,r)=\{\lambda\in\Lambda^+(m|n,r)\mid
j(t(\lambda))\leq \lambda_m\}.
\end{equation}

The functions $J$ and $j$ are also used to characterise the Mullineux conjugation.
A partition $\lambda=(\lambda_1,\lambda_2,\cdots)$ is called restricted if
 $\lambda_i-\lambda_{i+1}<p$ for $i=1,2,\cdots$. Let
 $\mathcal{RP}(r)$ denote the set of all restricted partitions of
 $r$.
The  Mullineux conjugation is the bijective function
(\cite[6.1]{BK},\cite[4.1]{Mu}):
\begin{equation}\label{BK,6.1}
\ttM : \mathcal{RP}(r)\rightarrow
\mathcal{RP}(r),\quad  \lambda\longmapsto\ttM(\lambda)=(j(\la),j(J(\la)),j(J^2(\la)),\ldots).
\end{equation}

We need another function. 

 Let $G=GL(m|n)$ and let $T$ be the usual maximal torus of $G_{\text{ev}}:=GL(m)\times GL(n)$. The
character group $X(T)=\Hom(T,G_m)$ is the free abelian group on
generators
$\varepsilon_1,\cdots,\varepsilon_m,\varepsilon_{m+1},\cdots,\varepsilon_{m+n}$.
 For
 $\lambda\in\Lambda(m+n):=\mathbb N^{m+n}$, we will identify $\lambda$ with
$\sum_{i=1}^{m+n}\lambda_i\varepsilon_i\in X(T)$.
 Denote by $\mathcal{D}_{m,n}$   the set of
representatives of the right coset of
$\mathfrak{S}_m\times\mathfrak{S}_n$ in $\mathfrak{S}_{m+n}$. 

Let $B_1$ denote the standard Borel subgroup which is the stabiliser of the full flag associated with a fixed ordered basis $(v_1,\ldots,v_{m+n})$. For $w\in\sD_{m,n}$, let $B_w$ be the stabiliser of the full flag associated with 
$(v_{w(1)},\ldots,v_{w(m+n)})$. With $B_w$, define irreducible module $L_w(\la)$ for $\la\in X(T)$; see \cite[\S2]{BK}. Then the relation $L_1(\lambda)\cong L_w(\mathbf{r}_w(\lambda))$
defines a bijective map 
$$\mathbf{r}_w:X(T)\longrightarrow X(T).$$  
In particular, if we take $w$ to be the longest element $w_1$ in
$\mathcal{D}_{m,n}$, the first part of the following lemma gives an explicit formula for $\mathbf{r}_{w_1}$.

\begin{lemma}[{\cite[6.3, 5.5]{BK}}]\label{6.3,5.5}
(1) If $r\leq m,n$ and $\lambda\in \mathcal{RP}(r)$, then
$\mathbf{r}_{w_1}(\lambda)=y(\ttM(\lambda))$, where $y:\mathcal{RP}(r)\rightarrow
\Lambda^+(m|n,r)$ is defined by
$y(\lambda)=\sum^{n}_{i=1}\lambda_i\varepsilon_{m+i}.$

(2) Assume $r\leq M$. Then
$$\Lambda^{++}(M|N,r)=\{\lambda\in\Lambda^+(M|N,r)\mid
\lambda_{M+1}\equiv\cdots\equiv\lambda_{M+N}\equiv0( \mbox{mod }
p)\}.$$ 
In particular, for $\la\in \Lambda^{++}(M|N,r)$, if we write 
$\lambda=\lambda^{0}+p\lambda^1$ such that
$\lambda^0$ is $p$-restricted, then $\lambda^0_{M+1}=\cdots=\lambda^0_{M+N}=0$.\end{lemma}

Let $
w=\left(
\begin{smallmatrix}
1&\cdots&m&m+1&\cdots&m+n&m+n+1&\cdots&m+2n\\
1&\cdots&m&m+n+1&\cdots&m+2n&m+1&\cdots&m+n
\end{smallmatrix}
\right)\in \mathcal{D}_{m+n,n}.
$

\begin{lemma}\label{permutation}
Assume $r\leq m+n$. If $\lambda= \lambda^0+p\lambda^1\in\Lambda^{++}(m+n|n,r)$ with $\la^0$ $p$-restricted and
$t(\lambda^0)=(\lambda^0_{m+1},\cdots,\lambda^0_{m+n})$, then
$$\mathbf{r}_w(\lambda)=\sum^m_{i=1}\lambda^0_i\varepsilon_i+\sum^{n}_{i=1}(\ttM(t(\lambda^0)))_i\varepsilon_{m+n+i}+p\lambda^1.$$
\end{lemma}
\begin{proof}
Set $M = m + n$ and $ N = n$. Then $w$ is the $w_1$ with respect to the subgroup 
$\fS_{\{m+1,\ldots,m+n\}}\times\fS_{\{m+n+1,\ldots,m+2n\}}$ in $\fS_{\{m+1,\ldots,m+2n\}}$.
Note that, 
first,  $t(\lambda^0)$ is also $p$-restricted and $r\leq m+n$ implies $|t(\lambda^0)|\leq n$;
second, by Lemma
\ref{6.3,5.5}(2), $\lambda^0_{m+n+i}=0$ for $i=1,\cdots,n$.
 Thus, applying \cite[Lem.~4.2]{BK} and Lemma \ref{6.3,5.5} yields\footnote{In fact, we apply \cite[4.2]{BK} repeatedly to the sequence of odd roots
\begin{equation*}
\begin{aligned}
&\varepsilon_{m+n}-\varepsilon_{m+n+1},\varepsilon_{m+n-1}-\varepsilon_{m+n+1},\cdots,\varepsilon_{m+1}-\varepsilon_{m+n+1};\\
&\varepsilon_{m+n}-\varepsilon_{m+n+2},\varepsilon_{m+n-1}-\varepsilon_{m+n+2},\cdots,\varepsilon_{m+1}-\varepsilon_{m+n+2};\\
&\cdots\cdots\\
&\varepsilon_{m+n}-\varepsilon_{m+2n},\varepsilon_{m+n-1}-\varepsilon_{m+2n},\cdots,\varepsilon_{m+1}-\varepsilon_{m+2n}.
\end{aligned}
\end{equation*}
and follow the proof of \cite[6.3]{BK}.} (noting \eqref{BK,6.1})\pagebreak
$$\aligned
\mathbf{r}_w(\lambda)&=\mathbf{r}_w(\la^0)+p\la^1=\sum^m_{i=1}(\lambda^0)_i\varepsilon_i+\sum^{n}_{i=1}j(J^{i-1}(t(\lambda^0)))\varepsilon_{m+n+i}+p\lambda^1\\
&=\sum^m_{i=1}(\lambda^0)_i\varepsilon_i+\sum^{n}_{i=1}(\ttM(t(\lambda^0)))_i\varepsilon_{m+n+i}+p\lambda^1,\endaligned
$$
as required.
\end{proof}

Define $\tau:\Lambda(m|n,r)\rightarrow
\Lambda(m+n|n,r)$ by setting $\tau(\lambda)=(\lambda^{(0)},0,\cdots,0|\lambda^{(1)})$.
Let ${}^{\tau\!}\Lambda(m|n,r):=\tau(\Lambda(m|n,r))$.

\begin{theorem}
Assume $r\leq m+n$ and $l=p$. Let  $w\in \mathcal{D}_{m+n,n}$ be the
same one in Lemma \ref{permutation}. Then there is a
bijection from $\sP_r$ to ${}^{\tau\!}\Lambda^{++}(m|n,r)$.
\end{theorem}
\begin{proof}We first define a function
$$h:\mathcal{P}_r\longrightarrow \La^{++}(m+n|n,r),\quad (\lambda,\xi,\eta)\longmapsto(\lambda^\iota+p\xi|p\eta).$$
 By Lemma \ref{6.3,5.5}(2), $h$ is well defined and is injective.
Since $\lambda$ is $p$-regular, so is
$t(\lambda)=(\lambda_{m+1},
\cdots,\lambda_{m+n})$.  Thus,
$\ttM(t(\lambda)^\iota)$ is also $p$-restricted and $|t(\lambda)|\leq n$ since $r\leq m+n$. By Lemma
\ref{permutation},
$$\mathbf{r}_w(h(\lambda,\mu,\xi))=\sum^m_{i=1}\lambda_i\varepsilon_i+
\sum^{n}_{i=1}\ttM(t(\lambda)^\iota)_i\varepsilon_{m+n+i}+p(\sum_{i=1}^m\xi_i\varepsilon_i+\sum_{i=1}^{n}\eta_i\varepsilon_{m+n+i})
.$$
So $\mathbf{r}_w(h(\lambda,\mu,\xi))\in  {{}^{\tau\!}\Lambda^{+}(m|n,r)}$. Now,
by the proof of  \cite[6.5]{BK},
\begin{equation}\label{polynomial weights relationship}
\begin{aligned}
{}^{\tau\!}\Lambda^{++}(m|n,r)&=\mathbf{r}_w(\Lambda^{++}(m+n|n,r))\cap{{}^{\tau\!}\Lambda^{+}(m|n,r)}\\
&=\{\mu\in {{}^{\tau\!}\Lambda^{+}(m|n,r)}\mid \mathbf{r}^{-1}_w(\mu)\in
\Lambda^{++}(m+n|n,r)\}.
\end{aligned}
\end{equation}
Hence, $\mathbf{r}_w(h(\lambda,\mu,\xi))\in
{{}^{\tau\!}\Lambda^{++}(m|n,r)}$ and $\mathbf{r}_w\circ h$
 defines an injective map
$$\mathbf{r}_w\circ h:\sP_r\longrightarrow
{}^{\tau\!}\Lambda^{++}(m|n,r).$$

We now show that $\mathbf{r}_w\circ h$ is surjective.
 For $\mu\in {{}^{\tau\!}\Lambda^{++}(m|n,r)}$, write $\tau^{-1}(\mu) =(\mu^{(0)}|\mu^{(1)})\in \Lambda^{++}(m|n,r) 
 $ and
  $\mu^{(1)}=\mu^{(1),0}+p\mu^{(1),1}$ with $\mu^{(1),0}$ $p$-restricted.  By \eqref{BK,6.5}, $j(\mu^{(1)})\leq \mu_m$ (and so $|\mu^{(1)}|\leq n$).
By (\ref{polynomial weights relationship}),
$\mathbf{r}^{-1}_w(\mu)\in \Lambda^{++}(m+n|n,r)$. This element is computed in the proof of
\cite[6.5]{BK} (noting  \eqref{BK,6.1}):
\begin{equation*}
\begin{aligned}
\mathbf{r}^{-1}_w(\mu)&=\sum_{i=1}^m\mu_i\varepsilon_i+\sum_{i=1}^nj(J^{i-1}((\mu^{(1),0})))\varepsilon_{m+i}+
p\sum_{i=1}^{n}(\mu^{(1),1})_i\varepsilon_{m+n+i}\\
&=\sum_{i=1}^m\mu_i\varepsilon_i+\sum_{i=1}^n\ttM(\mu^{(1),0})_i\varepsilon_{m+i}+
p\sum_{i=1}^{n}(\mu^{(1),1})_i\varepsilon_{m+n+i}.
\end{aligned}
\end{equation*}
Then
$\la:=(\mathbf{r}^{-1}_w(\mu))^{(0)}=\sum_{i=1}^m\mu_i\varepsilon_i+\sum_{i=1}^nM(\mu^{(1),0})_i\varepsilon_{m+i},
(\mathbf{r}^{-1}_w(\mu))^{(1)}=p\mu^{(1),1} $. Since $\mu^{(1),0}$
is $p$-restricted, $\ttM(\mu^{(1),0})$ is $p$-restricted. Hence, if we write
$$\la=\la^0+p\la^1,$$
where 
$\la^0$ is  $p$-restricted,
then $\la^1_{m+k}=0, k=1,\cdots,n.$ Set
$d=|\la^1|$. Removing the $n$ zeros at the end of $\la^1$ produces a partition in $\Lambda^+(m,d)$.
Thus, we obtain a triple $((\la^0)^\iota,\la^1,\mu^{(1),1})\in\sP_r$ such that
$\mathbf{r}_w\circ h((\la^0)^\iota,\la^1,\mu^{(1),1})=\mu$, proving the desired surjectivity. Therefore, $\mathbf{r}_w\circ h$ is a bijection from $\sP_r$ to
${}^{\tau\!}\Lambda^{++}(m|n,r)$.
\end{proof}

\end{document}